\PassOptionsToPackage{dvipsnames}{xcolor}

\documentclass[12pt,biblatex]{amsart}

\usepackage{times}
\usepackage{pifont}
\usepackage{subfiles}
\usepackage{xspace}
\usepackage{gensymb}

\usepackage{amsmath,amssymb,commath}

\usepackage[utf8]{inputenc}
\usepackage[T1]{fontenc}

\usepackage{hyperref}
\usepackage[capitalize,noabbrev]{cleveref}
\usepackage[backgroundcolor=yellow]{todonotes}

\usepackage{tikz}
\usetikzlibrary{matrix,arrows,calc,decorations.markings,patterns,shapes}

\makeatletter
\@ifclasswith{amsart}{externalize}%
{
		\usetikzlibrary{external} 
		\tikzexternalize[
			optimize=false,
			prefix=tikz/,
			mode=list and make]
		\makeatletter
		\renewcommand{\todo}[2][]{\tikzexternaldisable\@todo[#1]{#2}\tikzexternalenable}
		\makeatother
}%
{
		\newcommand{\tikzexternaldisable}{}
		\newcommand{\tikzexternalenable}{}
}
\makeatother

\usepackage{tikz-cd}
\usepackage{mathscinet}
\usepackage{enumitem}

\usepackage{array,multirow}

\makeatletter
\@ifclasswith{amsart}{biblatex}%
{
	\usepackage[backend=bibtex,doi=true,isbn=false,url=false,style=alphabetic]{biblatex}
	
	\addbibresource{shifted.bib}
	\newcommand{\biblio}
	{\printbibliography}
}
{
	\newcommand{\biblio}%
	{
		\bibliographystyle{alpha}
		\bibliography{shifted.bib}}
}
\makeatother

\author[S.~Matsumoto]{Sho Matsumoto}
\address{
	Graduate School of Science and Engineering, Kagoshima University 
	1-21-35, Korimoto, Kagoshima, Japan 
}
\email{shom@sci.kagoshima-u.ac.jp}

\author[P.~\'Sniady]{Piotr \'Sniady}
\address{
	Institute of Mathematics, Polish Academy of Sciences,
	\mbox{ul.~\'Sniadec\-kich 8,} \linebreak 00-656 Warszawa, Poland
} 
\email{psniady@impan.pl}

\newcommand{\sniady}{the second-named author\xspace}

\newcommand{\Z}{\mathbb{Z}}
\newcommand{\C}{\mathbb{C}}
\newcommand{\Q}{\mathbb{Q}}
\newcommand{\Young}{\mathbb{Y}}

\newcommand{\KO}{\mathbb{A}}
\newcommand{\KOodd}{\KO_{\bullet}^{\operatorname{odd}}}

\newcommand{\A}{\mathcal{A}}
\newcommand{\B}{\mathcal{B}}
\newcommand{\E}{\mathbb{E}}
\newcommand{\F}{\mathcal{F}}
\newcommand{\G}{\mathcal{G}}
\newcommand{\HH}{\mathcal{H}}

\newcommand{\Sequences}{\mathcal{X}}

\newcommand{\w}{\mathbb{w}}
\newcommand{\Alphabet}{\mathbb{A}}

\newcommand{\zero}{\bar{0}}
\newcommand{\one}{\bar{1}}

\newcommand{\LimitP}{\mathcal{P}}
\newcommand{\LimitQ}{\mathcal{Q}}

\newcommand{\R}{\mathbb{R}}
\newcommand{\X}{\mathbf{X}}
\newcommand{\Sym}[1]{\mathfrak{S}_{#1}}
\newcommand{\Spin}[1]{\widetilde{\mathfrak{S}}_{#1}}

\newcommand{\SGA}[1]{\C \Sym{#1}^-} 

\newcommand{\PP}{\mathbb{P}}
\newcommand{\tableaux}{\mathcal{T}}
\newcommand{\kumu}{\kappa}

\newcommand{\ChSpin}{\Ch^{\mathrm{spin}}}
\newcommand{\Ch}{\mathrm{Ch}}

\newcommand{\N}{\mathbb{N}}

\newcommand{\irrepSn}{\rho}
\newcommand{\irrepSp}{\psi}

\newcommand{\double}{D^*}

\newcommand{\class}{\pi}

\newcommand{\PHIeasy}[2]{\phi^{#1}\left(#2\right)}

\newcommand{\free}{{\mathbf{r}}}
\newcommand{\covarianceDisjoint}{\mathbf{k}^\bullet}
\newcommand{\covarianceProba}{\mathbf{k}}

\newcommand*\circled[1]{\tikz[baseline=(char.base)]{
		\node[shape=circle,draw,inner sep=2pt] (char) {$#1$};}}

\DeclareFontEncoding{FMS}{}{}
\DeclareFontSubstitution{FMS}{futm}{m}{n}
\DeclareFontEncoding{FMX}{}{}
\DeclareFontSubstitution{FMX}{futm}{m}{n}
\DeclareSymbolFont{fouriersymbols}{FMS}{futm}{m}{n}
\DeclareSymbolFont{fourierlargesymbols}{FMX}{futm}{m}{n}
\DeclareMathDelimiter{\VERT}{\mathord}{fouriersymbols}{152}{fourierlargesymbols}{147}

\newcommand{\vertiii}[1]{ \left\VERT #1 \right\VERT }

\newcommand{\Part}{\mathcal{P}}
\newcommand{\SP}{\mathcal{SP}}
\newcommand{\OP}{\mathcal{OP}}

\newcommand{\SW}[2]{\mathbb{P}^{\textrm{SW}}_{{#1},{#2}}}
\newcommand{\SWW}{\textrm{SW}}

\DeclareMathOperator{\GL}{GL}
\DeclareMathOperator{\PGL}{PGL}

\DeclareMathOperator{\DEGREE}{deg}
\DeclareMathOperator{\Tr}{Tr}
\DeclareMathOperator{\tr}{tr}
\DeclareMathOperator{\dimm}{dim}
\DeclareMathOperator{\Cov}{Cov}

\DeclareMathOperator{\id}{id}

\DeclareMathOperator{\spin}{spin}

\DeclareMathOperator{\lin}{span}

\theoremstyle{definition}
\newtheorem{definition}{Definition}[section]

\theoremstyle{plain}
\newtheorem{theorem}[definition]{Theorem}
\newtheorem{proposition}[definition]{Proposition}
\newtheorem{corollary}[definition]{Corollary}
\newtheorem{lemma}[definition]{Lemma}
\newtheorem{observation}[definition]{Observation}

\theoremstyle{remark}
\newtheorem{example}[definition]{Example}
\newtheorem{remark}[definition]{Remark}

\begin{document}

\subjclass[2010]{%
Primary
20C25;  
Secondary 
20C30,  
60F05, 	
05E10   
}

\keywords{projective representations of the symmetric groups, 
random strict partitions, random shifted tableaux, limit shape,
Kerov's CLT, shifted Schur--Weyl measure}

\title[Random strict partitions]{Random strict partitions \\ and random shifted tableaux}

\begin{abstract}
We study asymptotics of random \emph{shifted} Young diagrams which correspond to
a given sequence of reducible \emph{projective} representations of the symmetric
groups. We show limit results (Law of Large Numbers and Central Limit Theorem)
for their shapes, provided that the representation character ratios and their
cumulants converge to zero at some prescribed speed. Our class of examples
includes uniformly random shifted standard tableaux with prescribed shape as
well as shifted tableaux generated by some natural combinatorial algorithms
(such as shifted Robinson--Schensted--Knuth correspondence) applied to a random
input.
\end{abstract}

\maketitle

\hspace{0.5\textwidth} \begin{minipage}{0.5\textwidth}
	\emph{Dedidated to \\
		Etsuko Hirai 
		and 
		Takeshi Hirai.}
\end{minipage}	

\bigskip

\section{Introduction}

\subsection{Outlook}
\label{sec:outlook}

This paper is arranged in the following way which
is intended to make the learning curve somewhat less steep.

\smallskip

We start with \crefrange{sec:notations-start}{sec:notations-end} of this
Introduction where we present basic notations related to \emph{strict
	partitions}. We continue in \crefrange{sec:examples-start}{sec:examples-end}
with two concrete example models which we aim to treat,
namely
\emph{random strict tableaux of prescribed shape} and 
\emph{asymptotics of shifted Robinson--Schensted--Knuth correspondence}.

\smallskip

These two example models turn out to be special cases of a more general and
more abstract theory which we present later in the paper in
\cref{sec:representations-and-measures}. 
This general theory concerns \emph{random shifted Young diagrams} related to
\emph{reducible spin representations of the symmetric groups}. We also state
there the main results of the current paper, \cref{theo:mainLLN} and
\cref{theo:mainCLT}.

These general results are direct analogues of their classical counterparts for
\emph{linear} representations and \emph{usual} (non-shifted) Young diagrams. Our
strategy will be twofold: to revisit the ideas from the work of \sniady
\cite{Sniady2006} which concern the \emph{linear} representations of the symmetric
groups, as well as to use the link between the linear and the spin setup which
we explored only recently \cite{Matsumoto2018c}.

\smallskip

Our reuse of the notion of the \emph{approximate factorization property} for
the character ratios (which appears in the assumptions of \cref{theo:mainLLN}
and \cref{theo:mainCLT}) has some novelty: 
in \cref{sec:afc} we discuss this established notion with a new, more abstract and hopefully more elegant viewpoint of
the category theory. This new approach seems applicable in a quite wide
spectrum of setups, including the classical one \cite{Sniady2006}.

\smallskip

The remaining part of the paper is tailored specifically for the needs of the setup of the asymptotic spin representation theory of the symmetric groups.

In \cref{sec:ko} we present our key technical tools: Kerov--Olshanski algebra
and its spin analogue. We prove the main technical difficulty of the current
paper, \cref{theo:afp-for-spin}.

In \cref{sec:key} we prove \cref{theo:key-tool} which provides several
equivalent, convenient characterizations for approximate factorization property
for character ratios. This result directly implies the  main results of the
current paper,  \cref{theo:mainLLN} and \cref{theo:mainCLT}.

\cref{sec:applications} is devoted to applications of the
aforementioned \cref{theo:key-tool}. We construct a large collection of examples
of sequences of representations with approximate factorization property for
which the results of the current paper are applicable. In particular, we explain
how \cref{theo:randomSYT}, \cref{theo:schur-weyl},
\cref{coro:schur-weyl-insertion}, \cref{coro:schur-weyl-recording} fit into the
general framework of the approximate factorization property.

In \cref{sec:recover-shape} we recall the methods for finding explicitly the
limit shape of (shifted) Young diagrams.

\smallskip

In \cref{sec:projective-representations} we recall some basic facts from the spin
representation theory of the symmetric groups.

\subsection{Strict partitions}
\label{sec:strict-partitions}
\label{sec:notations-start}

\emph{Random partitions} 
occur in mathematics and physics in a wide variety of contexts,
in particular in the Gromov--Witten and Seiberg--Witten theories \cite{Okounkov2003,Vershik1995}. 
In the current paper we focus attention on a special class,
namely on \emph{random strict partitions}. 

We recall that 
\begin{equation}
\label{eq:strict}
\xi=(\xi_1,\dots,\xi_\ell)
\end{equation} 
is a \emph{strict partition} of an integer $n$ if $\xi_1>\dots>\xi_\ell$ form a
\emph{strictly} decreasing sequence of positive integers such that
$n=|\xi|=\xi_1+\cdots+\xi_\ell$, cf.~\cite[page 9]{Macdonald1995}. 
We denote by $\SP_n$ the set of strict partitions of a given integer $n$ and by
$\SP=\bigcup_{n\geq 0} \SP_n$ the set of all strict partitions.

It is convenient to represent graphically a strict partition $\xi$ by a
\emph{shifted Young diagram}, as it is shown in \cref{fig:strict}, which is a
collection of boxes
\begin{equation} 
\label{eq:boxes}
\left\{ (x,y) : x,y\in\N, \medspace
1\leq y < x \leq y + \xi_y  \right\}
 \end{equation}
on the plane. We use the French notation for drawing diagrams as well as the
usual Cartesian coordinate system $OXY$ on the plane; in~particular the variable
$x$ indexes the columns and the variable $y$ indexes the rows. The rows and
columns are indexed by natural numbers $\mathbb{N}=\{1,2,\dots\}$.

Above and in the following it is convenient to view a strict partition
\eqref{eq:strict} as an \emph{infinite} sequence of non-negative integers
\[\xi=(\xi_1,\dots,\xi_\ell,0,0,\dots)\]
by padding zeros at the end.

\subsection{Strict partitions: motivations and applications}
\label{sec:strict-motivations}

Strict partitions occur naturally in the context of \emph{spin representations}
of the symmetric groups $\Sym{n}$, cf.~\cref{sec:spin} later on, and in this
article we shall concentrate on this link. Nevertheless, it is worth pointing
out that they also appear in the theory of partially ordered sets as order
filters in the root poset of type $B_n$, as well as they form an interesting
infinite family of $d$-complete posets, which in turn is connected to fully
commutative elements of some Coxeter groups~\cite{Stembridge1996,Proctor1999}.
Also, strict partitions are in a bijective correspondence with permutations
which avoid patterns $132$ and $312$ \cite{Davis2017}.

\begin{figure}[t]
	\centerline{
		\begin{tikzpicture}[xscale=0.6,yscale=0.6]
		\begin{scope}
		\clip (0,0) -- (0,1) -- (1,1) -- (1,2) -- (2,2) -- (2,3) -- (3,3) -- (4,3) -- (4,2) -- (6,2) -- (6,1) -- (6,0);
		\draw[gray] (0,0) grid (8,3);
		\end{scope}
		\draw[blue,->] (-1,0) -- (8,0) node[anchor=west]{\textcolor{blue}{$x$}};
		\draw[blue,->] (-1,0) -- (-1,4) node[anchor=south]{\textcolor{blue}{$y$}};	   
		\fill[blue,opacity=0.1] (0,0) -- (0,1) -- (1,1) -- (1,2) -- (2,2) -- (2,3) -- (3,3) -- (4,3) -- (4,2) -- (6,2) -- (6,1) -- (6,0) -- cycle;
		\draw[ultra thick] (0,0) -- (0,1) -- (1,1) -- (1,2) -- (2,2) -- (2,3) -- (3,3) -- (4,3) -- (4,2) -- (6,2) -- (6,1) -- (6,0) -- cycle;
		\draw[line width=0.2cm, opacity=0.4,blue,line cap=round] (0.5,0.5) -- (5.5,0.5);
		\draw[line width=0.2cm, opacity=0.4,blue,line cap=round] (1.5,1.5) -- (5.5,1.5);
		\draw[line width=0.2cm, opacity=0.4,blue,line cap=round] (2.5,2.5) -- (3.5,2.5);	   
		\foreach \y in {1, 2,3,4,5, 6, 7, 8 }
			\draw[blue,thick] (\y cm, 0) ++(-1cm,5pt) -- ++(0,-10pt) node[anchor=north] {\tiny \textcolor{blue}{$\y$}};
		\foreach \y in {1, 2,3}
			\draw[blue,thick] (-1 cm,\y cm) ++(5pt,0) -- ++(-10pt,0) node[anchor=east] {\tiny \textcolor{blue}{$\y$}};
		\begin{scope}[xshift=11cm]
		    \draw[blue,->] (0,0) -- (10,0) node[anchor=west]{\textcolor{blue}{$x$}};
		    \draw[blue,->] (0,0) -- (0,8) node[anchor=south]{\textcolor{blue}{$y$}};	   
			\begin{scope}[xshift=1cm]
			\clip (0,0) -- (0,1) -- (1,1) -- (1,2) -- (2,2) -- (2,3) -- (3,3) -- (4,3) -- (4,2) -- (6,2) -- (6,1) -- (6,0);
			\draw[gray] (0,0) grid (8,3);
			\end{scope}
			\begin{scope}[rotate=90,yscale=-1]
			\clip (0,0) -- (0,1) -- (1,1) -- (1,2) -- (2,2) -- (2,3) -- (3,3) -- (4,3) -- (4,2) -- (6,2) -- (6,1) -- (6,0);
			\draw[gray] (0,0) grid (8,3);
			\end{scope}    
			\begin{scope}[xshift=1cm]
				\fill[blue,opacity=0.1] (0,0) -- (0,1) -- (1,1) -- (1,2) -- (2,2) -- (2,3) -- (3,3) -- (4,3) -- (4,2) -- (6,2) -- (6,1) -- (6,0) -- cycle;
				\draw[ultra thick] (0,0) -- (0,1) -- (1,1) -- (1,2) -- (2,2) -- (2,3) -- (3,3) -- (4,3) -- (4,2) -- (6,2) -- (6,1) -- (6,0) -- cycle;
				\draw[line width=0.2cm, opacity=0.4,blue,line cap=round] (0.5,0.5) -- (5.5,0.5);
				\draw[line width=0.2cm, opacity=0.4,blue,line cap=round] (1.5,1.5) -- (5.5,1.5);
				\draw[line width=0.2cm, opacity=0.4,blue,line cap=round] (2.5,2.5) -- (3.5,2.5);	 
				\end{scope}
				\begin{scope}[rotate=90,yscale=-1]
					\fill[OliveGreen,opacity=0.1] (0,0) -- (0,1) -- (1,1) -- (1,2) -- (2,2) -- (2,3) -- (3,3) -- (4,3) -- (4,2) -- (6,2) -- (6,1) -- (6,0) -- cycle;
					\draw[ultra thick] (0,0) -- (0,1) -- (1,1) -- (1,2) -- (2,2) -- (2,3) -- (3,3) -- (4,3) -- (4,2) -- (6,2) -- (6,1) -- (6,0) -- cycle;
					\draw[line width=0.2cm, opacity=0.4,OliveGreen,line cap=round] (0.5,0.5) -- (5.5,0.5);
					\draw[line width=0.2cm, opacity=0.4,OliveGreen,line cap=round] (1.5,1.5) -- (5.5,1.5);
					\draw[line width=0.2cm, opacity=0.4,OliveGreen,line cap=round] (2.5,2.5) -- (3.5,2.5);	
				\end{scope}    
				\draw[line width=3.2pt, dashed,red] (0,0) -- (1,0) -- (1,1) -- (2,1) -- (2,2) -- (3,2) -- (3,3) -- (4,3) -- (4,4) -- (5,4) -- (5,5) -- (6,5) -- (6,6);
				\foreach \y in {1,2,3,4,5,6,7,8,9}
				\draw[blue,thick] (\y cm, 0) ++(0,5pt) -- ++(0,-10pt) node[anchor=north] {\tiny \textcolor{blue}{$\y$}};
				\foreach \y in {1,2,3,4,5,6,7}
				\draw[blue,thick] (0,\y cm) ++(5pt,0) -- ++(-10pt,0) node[anchor=east] {\tiny \textcolor{blue}{$\y$}};
			\end{scope}
		\end{tikzpicture}
	}
	
	\caption{ Strict partition $\xi=(6,5,2)$ shown as a \emph{shifted Young
			diagram} and its double $D(\xi)=(7,7,5,3,2,2)$, cf.~\cref{sec:double}. } 
	\label{fig:strict}
	\label{fig:double}

\end{figure}
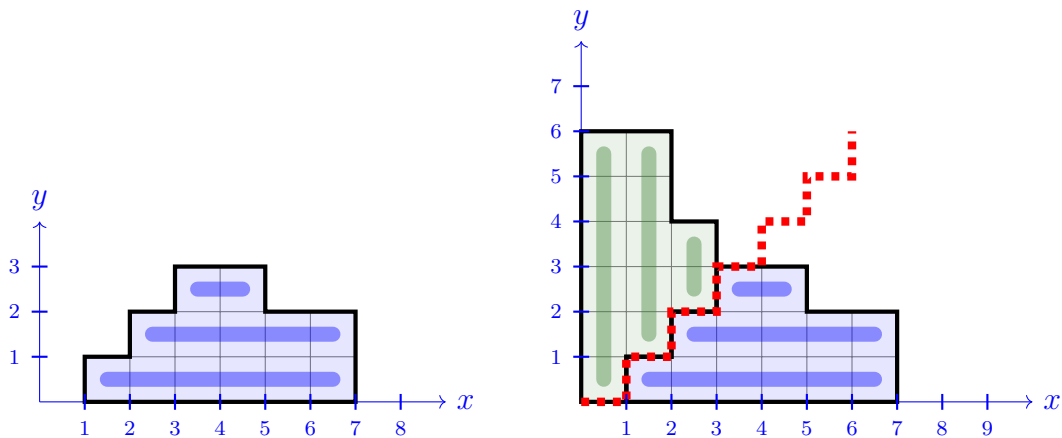

\subsection{Drawing (strict) partitions and Young diagrams for asymptotic
	problems} 

\label{sec:drawing}

For asymptotic problems we need some way of drawing large strict partitions
which would allow us to compare the shapes of such partitions with different
numbers of boxes. In this section we shall present a convenient solution to this
difficulty.

\begin{figure}
	
	\begin{tikzpicture}[scale=0.6]
	
	\begin{scope}
	\clip (-10.5,-1) rectangle (10.5,10.5);
	\draw[black!20] (-11,0) grid (11,11);
	\end{scope}
	
	\draw[dotted] (0,0) -- (-8.5,8.5);
	
	\begin{scope}[shift={(-0.5,-0.5)}]
	\clip  (7,7) -- (5,9) --(3,7) -- (2,8) -- (0,6) -- (1,5) --(0,4) -- (1,3) -- (0,2) -- (1,1);
	\fill[opacity=0.1,blue,rotate=45] (0,0) rectangle (15,15);
	\draw[blue,rotate=45,scale=sqrt(2)] (0,0) grid (11,11);
	\end{scope}

	\begin{scope}[shift={(0.5,-0.5)},xscale=-1]
	\draw	[dashed,ultra thick] (9,9) -- (7,7) -- (5,9) --(3,7) -- (2,8) -- (0.5,6.5);
	\end{scope}

	\begin{scope}[shift={(-0.5,-0.5)}]
	\draw [thick,blue,->] (0,0) -- (11,11) node[anchor=south west] {$x$};
	\draw [thick,blue,->] (0,0) -- (-8,8) node[anchor=south east] {$y$};

	\draw	[ultra thick] (9,9) -- (7,7) -- (5,9) --(3,7) -- (2,8) -- (0.5,6.5);
	
	\draw[line width=0.2cm, opacity=0.4,blue,line cap=round] (1,2) -- (6,7);
	\draw[line width=0.2cm, opacity=0.4,blue,line cap=round] (1,4) -- (5,8);
	\draw[line width=0.2cm, opacity=0.4,blue,line cap=round] (1,6) -- (2,7);
	
	\draw[blue,thick] (1 cm, 1 cm) ++(-5pt,5pt) -- ++(+10pt,-10pt);
	\foreach \y in {2,3,4,5, 6, 7, 8}
	\draw[blue,thick] (\y cm, \y cm) ++(-5pt,5pt) -- ++(+10pt,-10pt) node[anchor=north west] {\tiny \textcolor{blue}{$\y$}};
	
	\draw[blue,thick] (-1 cm, 1 cm) ++(5pt,5pt) -- ++(-10pt,-10pt);
	\foreach \y in {2,3,4,5,6}
	\draw[blue,thick] (-\y cm, \y cm) ++(5pt,5pt) -- ++(-10pt,-10pt) node[anchor=north east] {\tiny \textcolor{blue}{$\y$}};
	
	\end{scope}

	\draw [thick,->] (-10,0) -- (11,0) node[anchor=west] {$z$};
	\draw [thick,->] (0,-1) -- (0,11) node[anchor=south] {$t$}; 
	\draw (-1 cm,5pt) -- (-1 cm,-5pt); 
	\foreach \x in {1,2,3,4,5,6,7,8,9,10,-2,-3,-4,-5,-6,-7,-8,-9}
	\draw (\x cm,5pt) -- (\x cm,-5pt) node[anchor=north] {\tiny $\x$};
	\foreach \y in {1,2,3,4,5,6,7,8,9,10}
	\draw (5pt,\y cm) -- (-5pt,\y cm) node[anchor=east] {\tiny $\y$};
	
	\end{tikzpicture}
	\caption{The strict partition $\xi$ from \cref{fig:double} shown in the shifted Russian
		convention. The upper envelope of the boxes (the thick zig-zag line) is the
		graph of the \emph{profile} $\omega_{\xi}\colon \R_+ \to \R_+$. If necessary,
		the domain of the profile $\omega_{\xi}\colon \R \to \R_+$ can be extended to
		the whole real line (the thick dashed line).} 
	\label{fig:RussianDario}
	
	\vspace{3ex}
	
	\begin{tikzpicture}[scale=0.6]
	\begin{scope}
	\clip (-10.5,-1) rectangle (10.5,10.5);
	\draw[black!20] (-11,0) grid (11,11);
	\end{scope}
	\draw [thick,blue,->] (0,0) -- (11,11) node[anchor=south west] {$x$};
	\draw [thick,blue,->] (0,0) -- (-10,10) node[anchor=south east] {$y$};
	\begin{scope}
	\clip (0,0) -- (7,7) -- (5,9) --(3,7) -- (2,8) -- (0,6) -- (-1,7) -- (-2,6) -- (-4,8) -- (-6,6);
	\fill[blue,opacity=0.1]  (7,7) -- (5,9) --(3,7) -- (2,8) -- (0,6) -- (1,5) -- (0,4) -- (1,3) -- (0,2) -- (1,1);
	\fill[OliveGreen,opacity=0.1] (0,6) -- (-1,7) -- (-2,6) -- (-4,8) -- (-6,6) -- (0,0) -- (1,1) -- (0,2) -- (1,3) -- (0,4) -- (1,5);
	\draw[blue,rotate=45,scale=sqrt(2)] (0,0) grid (11,11);
	\end{scope}
	\draw	[ultra thick] (9,9) -- (7,7) -- (5,9) --(3,7) -- (2,8) -- (0,6) -- (-1,7) -- (-2,6) -- (-4,8) -- (-6,6) -- (-9,9);
	\draw [thick,->] (-11,0) -- (11,0) node[anchor=west] {$z$};
	\draw [thick,->] (0,-1) -- (0,11) node[anchor=south] {$t$};  
	\foreach \x in {1,2,3,4,5,6,7,8,9,10,-1,-2,-3,-4,-5,-6,-7,-8,-9,-10}
	\draw (\x cm,5pt) -- (\x cm,-5pt) node[anchor=north] {\tiny $\x$};
	\foreach \y in {1,2,3,4,5,6,7,8,9,10}
	\draw (5pt,\y cm) -- (-5pt,\y cm) node[anchor=east] {\tiny $\y$};
	\draw[line width=0.2cm, opacity=0.4,blue,line cap=round] (1,2) -- (6,7);
	\draw[line width=0.2cm, opacity=0.4,blue,line cap=round] (1,4) -- (5,8);
	\draw[line width=0.2cm, opacity=0.4,blue,line cap=round] (1,6) -- (2,7);
	\draw[line width=0.2cm, opacity=0.4,OliveGreen,line cap=round] (0,1) -- (-5,6);
	\draw[line width=0.2cm, opacity=0.4,OliveGreen,line cap=round] (0,3) -- (-4,7);
	\draw[line width=0.2cm, opacity=0.4,OliveGreen,line cap=round] (0,5) -- (-1,6);
	\foreach \y in {1,2,3,4,5, 6, 7, 8}
	\draw[blue,thick] (\y cm, \y cm) ++(-5pt,5pt) -- ++(+10pt,-10pt) node[anchor=north west] {\tiny \textcolor{blue}{$\y$}};
	
	\foreach \y in {1,2,3,4,5,6,7,8,9}
	\draw[blue,thick] (-\y cm, \y cm) ++(5pt,5pt) -- ++(-10pt,-10pt) node[anchor=north east] {\tiny \textcolor{blue}{$\y$}};
	\end{tikzpicture}
	\caption{The Young diagram $\lambda=D(\xi)$ from \cref{fig:double} shown in the
		Russian convention. The thick zig-zag line is the graph of the \emph{profile}
		$\omega_{D(\xi)}\colon \R\to [0,\infty)$.} \label{fig:Russian}
\end{figure}

\subsubsection{Shifted Russian convention for drawing shifted Young diagrams}
\label{sec:shiftedRussian} 

In the French convention we draw shifted Young diagrams on the plane using the
Cartesian coordinate system $OXY$, cf.~\cref{fig:strict}, but for asymptotic
questions it is convenient to draw them using the \emph{shifted Russian convention}
\cite[Section 4.2.6]{DeStavolaThesis} which
corresponds to a new coordinate system $OZT$ on the plane given by
\begin{equation}
\label{eq:italian} 
z=x-y-\frac{1}{2},   \qquad t=x+y - \frac{1}{2},
\end{equation}
see \cref{fig:RussianDario}. In this convention, the boundary of a shifted Young
diagram~$\xi$ (shown on \cref{fig:Russian} by the thick zigzag line), called its
\emph{profile}, is a graph of a function $\omega_{\xi}\colon \R_+ \to \R_+$ on
the positive half-line. If necessary, the domain of the profile can be extended
to the whole real line by declaring that $\omega_{\xi}(-x)=\omega_{\xi}(x)$ for
any $x\geq 0$. The graph of such a profile $\omega_\xi\colon \R \to \R_+$ is
shown on \cref{fig:RussianDario} as the union of the thick solid and the thick
dashed lines.

\subsubsection{Russian convention for drawing Young diagrams}
\label{sec:Russian} 

In the French convention we draw the usual (non-shifted) Young diagrams on the
plane using the Cartesian coordinate system $OXY$, but for asymptotic questions
it is convenient to draw them using the \emph{Russian convention}
(cf.~\cref{fig:Russian}) which corresponds to a new coordinate system $OZT$ on
the plane given by
\[ z=x-y, \qquad t=x+y.\]
In this convention, the boundary of a Young diagram $\lambda$ (shown on
\cref{fig:Russian} by the thick zigzag line), called its \emph{profile} is a
graph of a function $\omega_{\lambda}\colon \R \to \R_+$.

\subsubsection{Continual Young diagrams. Dilations of (shifted) Young diagrams}

We will say that a function $\omega \colon \R \to \R_+$ is a \emph{continual Young diagram}
\cite{Kerov1993a,Kerov1998} if
\begin{itemize}
	\item $|\omega(z_1)-\omega(z_2)| \leq |z_1-z_2|$ for any $z_1,z_2\in\R$,
	\item $\omega(z)=|z|$ for sufficiently big values of $|z|$.
\end{itemize}

\medskip

For a real number $r>0$ we can draw the boxes of the Young diagram $\lambda$ as
squares of side $r$. Such an object  --- denoted $r \lambda$ and called
\emph{dilated Young diagram} --- is usually no longer a Young diagram, but its
profile $\omega_{r \lambda}$ is still well defined and is a continual Young
diagram; note that
\begin{equation}
\label{eq:dilation} 
\omega_{r \lambda}(z) = r \  \omega_{\lambda}\left(\frac{z}{r}\right)
\qquad\text{for any } z\in\R
\end{equation}
which can be viewed as an alternative definition of $\omega_{r \lambda}$.

For a \emph{shifted} Young diagram $\xi$ the analogous operation of drawing
boxes as squares of side $r>0$ and then looking on the profile of the resulting
object is more delicate because we would have to adjust the additive terms in
the linear transformations \eqref{eq:italian} to the new size of the boxes. 
For this reason we take the
following\ analogue of \eqref{eq:dilation}:
\[ \omega_{r \xi}(z) = r \  \omega_{\xi}\left(\frac{z}{r}\right)
\qquad\text{for any } z\in\R\]
as the \emph{definition} of the profile of the dilated diagram $r \xi$.

\subsection{Shifted tableaux}
\label{sec:notations-end}

We recall that a \emph{shifted tableau}  is a filling of the boxes of a given
shifted Young diagram $\xi$ which is weakly increasing along the rows and
strictly increasing along the columns. Such a tableau is \emph{standard} if each
of the numbers $1,2,\dots,n$ appears as an entry exactly once, where $n=|\xi|$
is the number of the boxes, cf.~top of \cref{fig:SYT} for an example. 
For a shifted tableau $T$ of shape $\xi$ we denote by $T_{x,y}$ its entry in
$x$-th column and $y$-th row for integers $x,y$ such that $1\leq y< x\leq
y+\xi_y$. For any integer $0\leq i\leq n$ we denote by $T_{\leq
	i}=(\zeta_1,\zeta_2,\dots) \in \SP_i$ the strict partition which corresponds to
the set of boxes of $T$ occupied by numbers which are $\leq i$; in other words
\[ \zeta_y = \# \big\{ x :  y< x \leq y + \xi_y \text{ and } T_{x,y}\leq i \big\}.\]

\medskip

In the context of the representation theory of the symmetric groups,
\emph{shifted} tableaux play an analogous role (for \emph{spin} representations)
as the usual (\emph{``non-shifted''}) tableaux (for \emph{linear}
representations) and several classical combinatorial algorithms for tableaux
have their shifted counterparts
\cite{Worley1984,Sagan1987,Stembridge1989,Serrano2010}.

\medskip

A special role is played by the \emph{staircase strict partition}
\begin{equation} 
\label{eq:staircase}
\Delta_k=(k,\dots,3,2,1)
\end{equation}
and by shifted tableaux with this shape; we call the latter \emph{staircase
shifted tableaux}. Such staircase shifted tableaux --- apart from the
aforementioned general context of the representation theory --- appear in the
combinatorics of the Coxeter groups; more specifically they are in a bijective
correspondence with a natural class of objects which can be described in several
equivalent ways:
\begin{itemize}
	\item maximum length chains in the Tamari lattice  \cite{Fishel2014},

	\item maximal chains in weak Bruhat order on $312$-avoiding permutations in
$\Sym{n}$ \cite{Fishel2014} which are also known under the name of
\emph{$312$-avoiding sorting networks},
	
	\item \emph{both $132$- and $312$-avoiding sorting networks} 
	\cite[Proposition 3.12]{Linusson2018a},
	 
	\item the \emph{commutation class} of the word 
	\[\mathbf{w}_0=(s_1s_2\cdots s_{k-1})(s_1 s_2\cdots s_{k-2})\cdots(s_1s_2)(s_1)\] 
	in the symmetric group $\Sym{k}$ \cite{Schilling2017}.
	
\end{itemize}

\subsection{The first example: random shifted standard tableaux with prescribed shape}
\label{sec:examples-start}

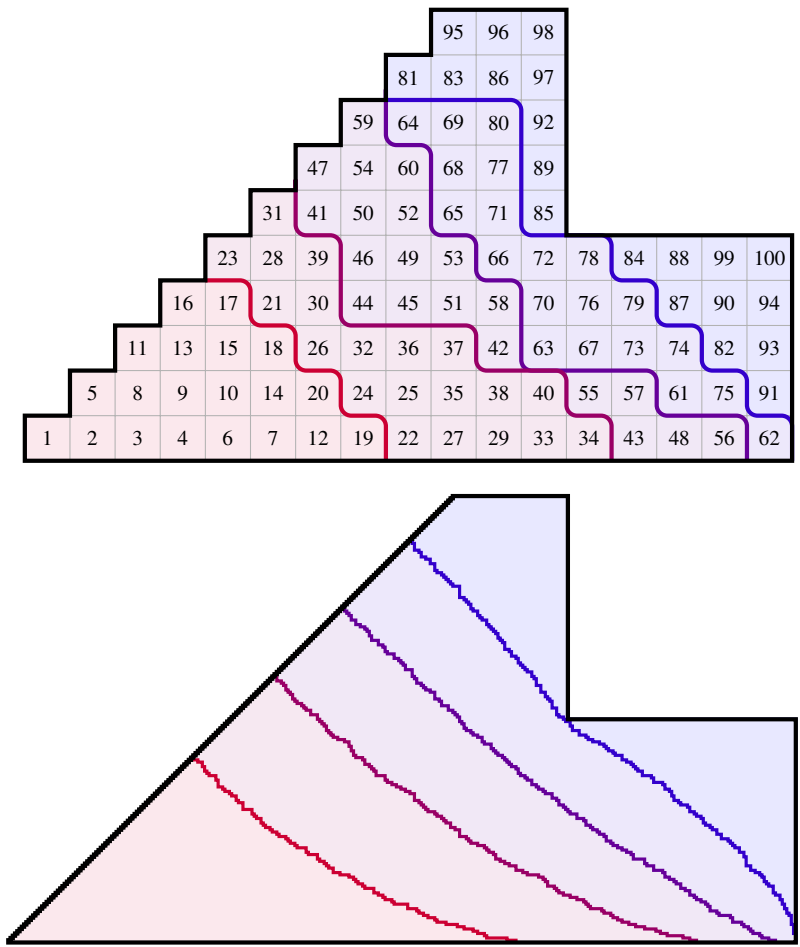
\begin{figure}
	\centering
	\begin{tikzpicture}[scale=0.6]

\draw(-0.5,0); 
 \draw[draw=none,fill={rgb:blue!100!red,1;white,10}](9,10) -- (9,9) -- (8,9) -- (8,8) -- (7,8) -- (7,7) -- (6,7) -- (6,6) -- (5,6) -- (5,5) -- (4,5) -- (4,4) -- (3,4) -- (3,3) -- (2,3) -- (2,2) -- (1,2) -- (1,1) -- (0,1) -- (0,0) -- (17,0) -- (17,5) -- (12,5) -- (12,10) -- cycle;
\draw[draw=none,fill={rgb:blue!80!red,1;white,10}](7,8) -- (7,7) -- (6,7) -- (6,6) -- (5,6) -- (5,5) -- (4,5) -- (4,4) -- (3,4) -- (3,3) -- (2,3) -- (2,2) -- (1,2) -- (1,1) -- (0,1) -- (0,0) -- (17,0) -- (17,1) -- (16,1) -- (16,2) -- (15,2) -- (15,3) -- (14,3) -- (14,4) -- (13,4) -- (13,5) -- (11,5) -- (11,8) -- cycle;
\draw[draw=none,fill={rgb:blue!60!red,1;white,10}](7,8) -- (7,7) -- (6,7) -- (6,6) -- (5,6) -- (5,5) -- (4,5) -- (4,4) -- (3,4) -- (3,3) -- (2,3) -- (2,2) -- (1,2) -- (1,1) -- (0,1) -- (0,0) -- (16,0) -- (16,1) -- (14,1) -- (14,2) -- (11,2) -- (11,4) -- (10,4) -- (10,5) -- (9,5) -- (9,7) -- (8,7) -- (8,8) -- cycle;
\draw[draw=none,fill={rgb:blue!40!red,1;white,10}](5,6) -- (5,5) -- (4,5) -- (4,4) -- (3,4) -- (3,3) -- (2,3) -- (2,2) -- (1,2) -- (1,1) -- (0,1) -- (0,0) -- (13,0) -- (13,1) -- (12,1) -- (12,2) -- (10,2) -- (10,3) -- (7,3) -- (7,5) -- (6,5) -- (6,6) -- cycle;
\draw[draw=none,fill={rgb:blue!20!red,1;white,10}](3,4) -- (3,3) -- (2,3) -- (2,2) -- (1,2) -- (1,1) -- (0,1) -- (0,0) -- (8,0) -- (8,1) -- (7,1) -- (7,2) -- (6,2) -- (6,3) -- (5,3) -- (5,4) -- cycle;
\begin{scope}
\clip(9,10) -- (9,9) -- (8,9) -- (8,8) -- (7,8) -- (7,7) -- (6,7) -- (6,6) -- (5,6) -- (5,5) -- (4,5) -- (4,4) -- (3,4) -- (3,3) -- (2,3) -- (2,2) -- (1,2) -- (1,1) -- (0,1) -- (0,0) -- (17,0) -- (17,5) -- (12,5) -- (12,10);
\draw[black!30] (0,0) grid (500,500);
 \end{scope}
\draw[ultra thick,draw=blue!80!red,rounded corners](17,0) -- (17,1) -- (16,1) -- (16,2) -- (15,2) -- (15,3) -- (14,3) -- (14,4) -- (13,4) -- (13,5) -- (11,5) -- (11,8) -- (8,8);
\draw[ultra thick,draw=blue!60!red,rounded corners](16,0) -- (16,1) -- (14,1) -- (14,2) -- (11,2) -- (11,4) -- (10,4) -- (10,5) -- (9,5) -- (9,7) -- (8,7) -- (8,8) -- (8,8);
\draw[ultra thick,draw=blue!40!red,rounded corners](13,0) -- (13,1) -- (12,1) -- (12,2) -- (10,2) -- (10,3) -- (7,3) -- (7,5) -- (6,5) -- (6,6) -- (6,6);
\draw[ultra thick,draw=blue!20!red,rounded corners](8,0) -- (8,1) -- (7,1) -- (7,2) -- (6,2) -- (6,3) -- (5,3) -- (5,4) -- (4,4);
\draw[ultra thick](9,10) -- (9,9) -- (8,9) -- (8,8) -- (7,8) -- (7,7) -- (6,7) -- (6,6) -- (5,6) -- (5,5) -- (4,5) -- (4,4) -- (3,4) -- (3,3) -- (2,3) -- (2,2) -- (1,2) -- (1,1) -- (0,1) -- (0,0) -- (17,0) -- (17,5) -- (12,5) -- (12,10) -- cycle;
\tiny
\node at (0.5,0.5) {1}; 
\node at (1.5,0.5) {2}; 
\node at (2.5,0.5) {3}; 
\node at (3.5,0.5) {4}; 
\node at (4.5,0.5) {6}; 
\node at (5.5,0.5) {7}; 
\node at (6.5,0.5) {12}; 
\node at (7.5,0.5) {19}; 
\node at (8.5,0.5) {22}; 
\node at (9.5,0.5) {27}; 
\node at (10.5,0.5) {29}; 
\node at (11.5,0.5) {33}; 
\node at (12.5,0.5) {34}; 
\node at (13.5,0.5) {43}; 
\node at (14.5,0.5) {48}; 
\node at (15.5,0.5) {56}; 
\node at (16.5,0.5) {62}; 
\node at (1.5,1.5) {5}; 
\node at (2.5,1.5) {8}; 
\node at (3.5,1.5) {9}; 
\node at (4.5,1.5) {10}; 
\node at (5.5,1.5) {14}; 
\node at (6.5,1.5) {20}; 
\node at (7.5,1.5) {24}; 
\node at (8.5,1.5) {25}; 
\node at (9.5,1.5) {35}; 
\node at (10.5,1.5) {38}; 
\node at (11.5,1.5) {40}; 
\node at (12.5,1.5) {55}; 
\node at (13.5,1.5) {57}; 
\node at (14.5,1.5) {61}; 
\node at (15.5,1.5) {75}; 
\node at (16.5,1.5) {91}; 
\node at (2.5,2.5) {11}; 
\node at (3.5,2.5) {13}; 
\node at (4.5,2.5) {15}; 
\node at (5.5,2.5) {18}; 
\node at (6.5,2.5) {26}; 
\node at (7.5,2.5) {32}; 
\node at (8.5,2.5) {36}; 
\node at (9.5,2.5) {37}; 
\node at (10.5,2.5) {42}; 
\node at (11.5,2.5) {63}; 
\node at (12.5,2.5) {67}; 
\node at (13.5,2.5) {73}; 
\node at (14.5,2.5) {74}; 
\node at (15.5,2.5) {82}; 
\node at (16.5,2.5) {93}; 
\node at (3.5,3.5) {16}; 
\node at (4.5,3.5) {17}; 
\node at (5.5,3.5) {21}; 
\node at (6.5,3.5) {30}; 
\node at (7.5,3.5) {44}; 
\node at (8.5,3.5) {45}; 
\node at (9.5,3.5) {51}; 
\node at (10.5,3.5) {58}; 
\node at (11.5,3.5) {70}; 
\node at (12.5,3.5) {76}; 
\node at (13.5,3.5) {79}; 
\node at (14.5,3.5) {87}; 
\node at (15.5,3.5) {90}; 
\node at (16.5,3.5) {94}; 
\node at (4.5,4.5) {23}; 
\node at (5.5,4.5) {28}; 
\node at (6.5,4.5) {39}; 
\node at (7.5,4.5) {46}; 
\node at (8.5,4.5) {49}; 
\node at (9.5,4.5) {53}; 
\node at (10.5,4.5) {66}; 
\node at (11.5,4.5) {72}; 
\node at (12.5,4.5) {78}; 
\node at (13.5,4.5) {84}; 
\node at (14.5,4.5) {88}; 
\node at (15.5,4.5) {99}; 
\node at (16.5,4.5) {100}; 
\node at (5.5,5.5) {31}; 
\node at (6.5,5.5) {41}; 
\node at (7.5,5.5) {50}; 
\node at (8.5,5.5) {52}; 
\node at (9.5,5.5) {65}; 
\node at (10.5,5.5) {71}; 
\node at (11.5,5.5) {85}; 
\node at (6.5,6.5) {47}; 
\node at (7.5,6.5) {54}; 
\node at (8.5,6.5) {60}; 
\node at (9.5,6.5) {68}; 
\node at (10.5,6.5) {77}; 
\node at (11.5,6.5) {89}; 
\node at (7.5,7.5) {59}; 
\node at (8.5,7.5) {64}; 
\node at (9.5,7.5) {69}; 
\node at (10.5,7.5) {80}; 
\node at (11.5,7.5) {92}; 
\node at (8.5,8.5) {81}; 
\node at (9.5,8.5) {83}; 
\node at (10.5,8.5) {86}; 
\node at (11.5,8.5) {97}; 
\node at (9.5,9.5) {95}; 
\node at (10.5,9.5) {96}; 
\node at (11.5,9.5) {98}; 
 
\end{tikzpicture}
	
	\vspace{2ex}
	
	\subfile{FIGURES/level_curves-100.tex} \caption{\emph{Above:} random shifted
	standard Young tableau with $n=100$ boxes, sampled with the uniform
	distribution on the set of shifted standard tableaux with fixed shape
	$\xi=(17,16,\dots,13,\ 7,6,\dots,3)$.  The coloured  \emph{level curves} indicate
	positions of $20\%, 40\%, 60\%, 80\%$ of the boxes with the smallest numbers.
	\emph{Below:} analogous random tableau with $n=39600$ boxes. Individual boxes 
	and the numbers filling the tableau are not shown.} \label{fig:SYT}
\end{figure}

For $\xi\in\SP$ we denote by $\tableaux_\xi$ the set of shifted standard
tableaux with the presecribed shape $\xi$ and by $\PP_\xi$ the uniform
probability measure on~$\tableaux_\xi$. Such a uniformly random shifted tableau
can be generated by the shifted hook walk algorithm \cite{Sagan1980} and is an
important tool in some proofs of the hook length formula for the number
$|\tableaux_\xi|$ of shifted tableaux \cite{Sagan1980}.

\subsubsection{Limit shape for random shifted tableaux}

\label{sec:stacks-of-cubes}

Following the ideas of Pittel and Romik \cite{Pittel2007}, a shifted tableau $T$
with shape $\xi\in\SP_n$ can be regarded as a three-dimensional stack of cubes
over the two-dimensional shifted Young diagram $\xi$, with $T_{x,y}$ cubes
stacked over the square $[x-1,x]\times [y-1,y] \times \{0\}$. Alternatively, the
function $(x,y)\mapsto T_{\lceil x\rceil ,\lceil y\rceil}$ can be interpreted as the graph of the
(non-continuous) surface of the upper envelope of this stack.

It is convenient to rescale the unit boxes on the plane by the factor
$\frac{1}{\sqrt{2n}}$ in such a way that the area of $\xi$ becomes equal to
$\frac{1}{2}$, and to rescale the height of the cubes by the factor
$\frac{1}{n}$ in such a way that the heights of stacks of cubes are all between
$0$ and $1$. In this way we may ask asymptotic questions about large random
shifted tableaux using the language of random surfaces.

\medskip

Before reading the exact form of the following result we recommend to consult
the almost self-explanatory \cref{fig:SYT}.

The following result states a kind of \emph{Law of Large Numbers} result that if
the (scaled down) shapes of the strict partitions $\big( \xi^{(k)} \big)$
converge to some limit shape $\Lambda$ then the aforementioned random surface
which corresponds to a uniformly random shifted tableau
$T^{(k)}\in\tableaux_{\xi^{(k)}}$ \emph{converges in probability} towards some
deterministic surface $F\colon \Lambda\to [0,1]$ \emph{in the sense of level curves}.
The latter sense of convergence means that the (scaled down) region on the plane
occupied by the boxes with (scaled) height bounded from above by any fixed real
number $\alpha$ --- in the physical geography the boundary of such a region is a
curve called \emph{the contour curve} or \emph{the level curve} --- converges in
probability to the region where the surface $F$ takes values which are bounded
from above by the same level $\alpha$.

\begin{theorem}[Law of Large Numbers for random shifted tableaux]
\label{theo:randomSYT} 

\ \\ For each $k\geq 1$ let $\xi^{(k)}=(\xi^{(k)}_1,\xi^{(k)}_2,\dots)
\in\SP_{n_k}$ for some sequence $(n_k)$ of positive integers which tends to
infinity. We assume that the sequence of rescaled profiles converges to some
limit
\[ \Omega_{1}:= \lim_{k\to\infty} 
\omega_{\frac{1}{\sqrt{2 n_k}} \xi^{(k)}} \]
in the sense of pointwise convergence of functions on $\R_+$.
We denote by
\[ \Lambda = \big\{ (x,y) : 0\leq y\leq x \text{ and }
x+y < \Omega_1(x-y)\big\}
\] 
this limit shape drawn in the French coordinate system.

We also assume that the sequence $(\xi^{(k)})$ is \emph{$C$-balanced}
\cite{Biane1998}, i.e.~the length of the first row
satisfies the bound 
\[ \xi^{(k)}_1 < C \sqrt{n_k}\] 
for some constant $C>0$ and all integers $k\geq 1$.

\bigskip

Then there exists a function $F \colon \Lambda \to [0,1]$ and the corresponding
a family of level curves (drawn in the Russian convention) indexed by
$0<\alpha<1$, defined for $z\geq 0$ by
\[ \Omega_{\alpha}(z) = \sup \big\{ x+y: (x,y)\in\Lambda\text{ and } 
                               x-y=z \text{ and } F(x,y)\leq \alpha \big\}; \]
we use the convention that if the supremum above is taken over the empty set
then $\Omega_{\alpha}(z) = |z|$.

We denote by $T^{(k)}$ a random standard Young tableau, sampled with the uniform
distribution on $\tableaux_{\xi^{(k)}}$. Then for each $0<\alpha<1$ the
(rescaled by the factor $\frac{1}{\sqrt{2 n_k}}$) profile of the shifted Young
diagram $T^{(k)}_{\leq \alpha n_k }$ 
converges in probability to $\Omega_{\alpha}$. In other words, for each
$\epsilon>0$
\begin{equation}
\label{eq:LLN} 
\lim_{k\to\infty} \PP_{\xi^{(k)}}\Bigg\{
T^{(k)}\in\tableaux_{\xi^{(k)}} :  
\sup_{x\geq 0}
\left| \omega_{\frac{1}{\sqrt{2 n_k}} T^{(k)}_{\leq \alpha n_k  }}(x) - 
    \Omega_{\alpha}(x) \right| 
>\epsilon \Bigg\}  =0. 
\end{equation}

\end{theorem}

The proof and the exact form of the limit surface $F$ is postponed to
\cref{sec:proof-randomSYT}. The results of the current paper can be also used to
show that the fluctuations of the random surfaces $T^{(k)}$ around the limit
shape are Gaussian.

\begin{remark}
With some minor effort, an analogous result for the usual (non-shifted) tableaux can
be extracted from the work of Biane \cite[Theorem 1.5.1]{Biane1998}. This
non-shifted analogue was known to Biane; in particular \sniady
witnessed a presentation of Biane in Spring 2008 in which this kind of result was
stated by referring to computer simulations similar to the one from
\cref{fig:SYT}, see also \cite{SniadyHabilitation} in the context of
Gaussianity of fluctuations. Nevertheless, it seems that this non-shifted
version was never explicitly stated in the existing literature and for this
reason it was overlooked by the scientific community. For example, the authors
of \cite{Pittel2007} cite the work of Biane but do not seem to be aware of a partial
overlap of their result with \cite{Biane1998}.
\end{remark}

\subsubsection{Example: random staircase tableaux} The assumptions of the above
\cref{theo:randomSYT} are fulfilled for the sequence
\[
\xi^{(k)}=\Delta_k \in \SP_{n_k} 
\]
of staircase strict partitions, cf.~Equation \eqref{eq:staircase},
with $n_k=\binom{k+1}{2}$ and the limit shape
\[ \Omega_1(x) = 
\begin{cases}
2-|x|         & \text{for } |x| \leq  1 \\
          |x| & \text{for } |x|  > 1.
\end{cases}
\]
\cref{theo:randomSYT} is applicable and, as we shall see in
\cref{sec:example-LPS}, in this case the limit surface
\[ F(x,y)=L(x,y) \qquad 
                            \text{for } 0\leq y\leq x\leq 1 \]
is the restriction of the function $L$ described by Pittel and Romik
\cite[Section 1.1]{Pittel2007}. In this way we recover a part of the result of
Linusson, Potka, Sulzberger \cite[Theorem 3.8]{Linusson2018} who proved a
stronger version of this result (the authors of \cite{Linusson2018} used
convergence with respect to a stronger topology given by the pointwise convergence
of the entries of tableaux towards the limit surface $F$). Note that the latter
paper contains also large deviations results which do not seem to be accessible
by out methods.

\subsection{The second example: asymptotics of shifted Schur--Weyl measures}
\label{sec:examples-end}

\subsubsection{Shifted RSK correspondence}

\label{sec:shifted-RSK}
	
Let us fix some positive integers $n$ and $d$. 
We consider the ordered set \
\[ \Alphabet_d:= \left\{ \circled{1}<1<\circled{2}<2<\cdots<\circled{d}<d\right\}.\]
In the following we will use \emph{shifted RSK correspondence}
\cite{Worley1984,Sagan1987,Hoffman1992} in a very specific context (with the
notations of \cite[Theorem 8.1]{Sagan1987} this corresponds to the case when the
circled matrix $(a_{ij})_{i\in [n], j\in [d]}$ is such that $\sum_j a_{ij} =1 $
for any $i\in [n]$).
In this context shifted RSK is a bijection between:
\begin{itemize}[itemsep=2ex]
	\item the set 
 \[ \Omega_{n,d}:=\left\{ \mathbf{w}=(w_1,\dots,w_n) :
 w_1,\dots,w_n \in \Alphabet_d \right\} \]
of words of length $n$,  and 

\item pairs $(P,Q)$ of (generalized) tableaux of the same shape $\xi\in\SP_n$
which fulfil the following conditions.

The \emph{insertion tableau} $P$ is a \emph{generalized shifted tableau} which
means that it is a filling of the boxes of $\xi$ with the elements of
$\Alphabet_d$ which is weakly increasing along the rows and along the columns;
furthermore each circled symbol appears in each row at most once, and each
non-circled symbol appears in each column at most once.

The \emph{recording tableau} $Q$ is a filling of the boxes of
$\xi$ with the elements of the set $[n]:=\{1,\dots,n\}$ with the property that
each element of $[n]$ appears exactly once; furthermore, each row and each column
is increasing. Additionally, each non-diagonal entry of the tableau can be
circled or not; each diagonal entries is non-circled.
\end{itemize}

\medskip

We consider the discrete probability space $\Omega_{n,d}$ equipped with the
uniform measure. We are interested in the probability distribution of the random
variable $\xi=\xi(\mathbf{w})$. Since shifted RSK correspondence is a bijection,
this probability distribution on $\SP_n$ ---  called \emph{shifted Schur--Weyl
	measure} --- is explicitly given by
\begin{multline} 
\label{eq:SW-measure}
\SW{n}{d}(\xi) = \\
 \frac{\text{\#(generalized shifted tableaux of shape $\xi$ with entries in $\Alphabet_d$)} 
 	    \cdot \ g^\xi }{2^{\ell(\xi)}\  d^n },
\end{multline}
where $\ell(\xi)$ denotes the number of non-zero parts of the partition $\xi$
and $g^\xi:=|\tableaux_\xi|$ is the number of shifted standard tableaux of shape
$\xi$. This probability distribution also has a natural representation-theoretic
interpretation which we will be discussed later in \cref{sec:example-SW}.

\subsubsection{Asymptotics of shifted Schur--Weyl measures} 
\label{sec:asymptotic-SW}

Usually we draw boxes which constitute a shifted Young diagram $\xi\in\SP_n$ as
unit squares. However, as we already mentioned, for asymptotic problems it might
be beneficial to draw them as squares of side $\frac{1}{\sqrt{2n}}$ so that the
total area occupied by the boxes is equal to $\frac{1}{2}$. 

The following result states that random strict partitions distributed according
to shifted Schur--Weyl measures with carefully chosen parameters converge (after the
rescaling of boxes described above) in probability towards some explicit limit
shapes. The analogue of this result for non-shifted Schur--Weyl measures is due
to Biane \cite{Biane2001}.

\begin{theorem}[Law of large numbers for shifted Schur--Weyl measures]

\label{theo:schur-weyl}

  \ 
	
	Let $(d_n)$ be a sequence of positive integers with the property that the limit
	\[ c:= \lim_{n\to\infty} \frac{\sqrt{n}}{d_n} \]
	exists.
	
	Then there exists a function $\Omega^{\SWW}_c \colon \R_+ \to \R_+$ with the property
that for each $\epsilon>0$
	\[ \lim_{n\to\infty} \SW{n}{d_n}\left\{ \xi\in \SP_n : 
	        \sup_{x\geq 0} 
	        \left| \omega_{\frac{1}{\sqrt{2n}} \xi}(x) - \Omega^{\SWW}_{c}(x) \right|
	            > \epsilon \right\} = 0. \]

\end{theorem}

The proof is postponed to \cref{sec:example-SW}; the exact form of the limit
curve will be discussed in \cref{sec:SW-measure-shape}.

This result is illustrated on \cref{fig:SW-theoretic,fig:SW-theoretic-2}. 

\begin{figure}
	\begin{tikzpicture}[scale=2]
	
	\draw [thick,->] (0,0) -- (3.5,0) node[anchor=west] {$z$};
	\draw [thick,->] (0,0) -- (0,3.5) node[anchor=south] {$t$};
	
	\foreach \x in {1,2,3}
	\draw (\x cm,1pt) -- (\x cm,-1pt) node[anchor=north] {\tiny $\x$};
	\foreach \y in {1,2,3}
	\draw (1pt,\y cm) -- (-1pt,\y cm) node[anchor=east] {\tiny $\y$};

	\draw[dotted] (0,0) -- (3,3);

	\begin{scope}[rotate=45,scale=sqrt(2)/40]
	\begin{scope}[xshift=-0.05cm]
	\draw[blue,fill=blue!10](23.00,23.10) -- (23.00,23.00) -- (22.90,23.00) --
	(22.90,22.90) -- (22.80,22.90) -- (22.80,22.80) -- (22.70,22.80) --
	(22.70,22.70) -- (22.60,22.70) -- (22.60,22.60) -- (22.50,22.60) --
	(22.50,22.50) -- (22.40,22.50) -- (22.40,22.40) -- (22.30,22.40) --
	(22.30,22.30) -- (22.20,22.30) -- (22.20,22.20) -- (22.10,22.20) --
	(22.10,22.10) -- (22.00,22.10) -- (22.00,22.00) -- (21.90,22.00) --
	(21.90,21.90) -- (21.80,21.90) -- (21.80,21.80) -- (21.70,21.80) --
	(21.70,21.70) -- (21.60,21.70) -- (21.60,21.60) -- (21.50,21.60) --
	(21.50,21.50) -- (21.40,21.50) -- (21.40,21.40) -- (21.30,21.40) --
	(21.30,21.30) -- (21.20,21.30) -- (21.20,21.20) -- (21.10,21.20) --
	(21.10,21.10) -- (21.00,21.10) -- (21.00,21.00) -- (20.90,21.00) --
	(20.90,20.90) -- (20.80,20.90) -- (20.80,20.80) -- (20.70,20.80) --
	(20.70,20.70) -- (20.60,20.70) -- (20.60,20.60) -- (20.50,20.60) --
	(20.50,20.50) -- (20.40,20.50) -- (20.40,20.40) -- (20.30,20.40) --
	(20.30,20.30) -- (20.20,20.30) -- (20.20,20.20) -- (20.10,20.20) --
	(20.10,20.10) -- (20.00,20.10) -- (20.00,20.00) -- (19.90,20.00) --
	(19.90,19.90) -- (19.80,19.90) -- (19.80,19.80) -- (19.70,19.80) --
	(19.70,19.70) -- (19.60,19.70) -- (19.60,19.60) -- (19.50,19.60) --
	(19.50,19.50) -- (19.40,19.50) -- (19.40,19.40) -- (19.30,19.40) --
	(19.30,19.30) -- (19.20,19.30) -- (19.20,19.20) -- (19.10,19.20) --
	(19.10,19.10) -- (19.00,19.10) -- (19.00,19.00) -- (18.90,19.00) --
	(18.90,18.90) -- (18.80,18.90) -- (18.80,18.80) -- (18.70,18.80) --
	(18.70,18.70) -- (18.60,18.70) -- (18.60,18.60) -- (18.50,18.60) --
	(18.50,18.50) -- (18.40,18.50) -- (18.40,18.40) -- (18.30,18.40) --
	(18.30,18.30) -- (18.20,18.30) -- (18.20,18.20) -- (18.10,18.20) --
	(18.10,18.10) -- (18.00,18.10) -- (18.00,18.00) -- (17.90,18.00) --
	(17.90,17.90) -- (17.80,17.90) -- (17.80,17.80) -- (17.70,17.80) --
	(17.70,17.70) -- (17.60,17.70) -- (17.60,17.60) -- (17.50,17.60) --
	(17.50,17.50) -- (17.40,17.50) -- (17.40,17.40) -- (17.30,17.40) --
	(17.30,17.30) -- (17.20,17.30) -- (17.20,17.20) -- (17.10,17.20) --
	(17.10,17.10) -- (17.00,17.10) -- (17.00,17.00) -- (16.90,17.00) --
	(16.90,16.90) -- (16.80,16.90) -- (16.80,16.80) -- (16.70,16.80) --
	(16.70,16.70) -- (16.60,16.70) -- (16.60,16.60) -- (16.50,16.60) --
	(16.50,16.50) -- (16.40,16.50) -- (16.40,16.40) -- (16.30,16.40) --
	(16.30,16.30) -- (16.20,16.30) -- (16.20,16.20) -- (16.10,16.20) --
	(16.10,16.10) -- (16.00,16.10) -- (16.00,16.00) -- (15.90,16.00) --
	(15.90,15.90) -- (15.80,15.90) -- (15.80,15.80) -- (15.70,15.80) --
	(15.70,15.70) -- (15.60,15.70) -- (15.60,15.60) -- (15.50,15.60) --
	(15.50,15.50) -- (15.40,15.50) -- (15.40,15.40) -- (15.30,15.40) --
	(15.30,15.30) -- (15.20,15.30) -- (15.20,15.20) -- (15.10,15.20) --
	(15.10,15.10) -- (15.00,15.10) -- (15.00,15.00) -- (14.90,15.00) --
	(14.90,14.90) -- (14.80,14.90) -- (14.80,14.80) -- (14.70,14.80) --
	(14.70,14.70) -- (14.60,14.70) -- (14.60,14.60) -- (14.50,14.60) --
	(14.50,14.50) -- (14.40,14.50) -- (14.40,14.40) -- (14.30,14.40) --
	(14.30,14.30) -- (14.20,14.30) -- (14.20,14.20) -- (14.10,14.20) --
	(14.10,14.10) -- (14.00,14.10) -- (14.00,14.00) -- (13.90,14.00) --
	(13.90,13.90) -- (13.80,13.90) -- (13.80,13.80) -- (13.70,13.80) --
	(13.70,13.70) -- (13.60,13.70) -- (13.60,13.60) -- (13.50,13.60) --
	(13.50,13.50) -- (13.40,13.50) -- (13.40,13.40) -- (13.30,13.40) --
	(13.30,13.30) -- (13.20,13.30) -- (13.20,13.20) -- (13.10,13.20) --
	(13.10,13.10) -- (13.00,13.10) -- (13.00,13.00) -- (12.90,13.00) --
	(12.90,12.90) -- (12.80,12.90) -- (12.80,12.80) -- (12.70,12.80) --
	(12.70,12.70) -- (12.60,12.70) -- (12.60,12.60) -- (12.50,12.60) --
	(12.50,12.50) -- (12.40,12.50) -- (12.40,12.40) -- (12.30,12.40) --
	(12.30,12.30) -- (12.20,12.30) -- (12.20,12.20) -- (12.10,12.20) --
	(12.10,12.10) -- (12.00,12.10) -- (12.00,12.00) -- (11.90,12.00) --
	(11.90,11.90) -- (11.80,11.90) -- (11.80,11.80) -- (11.70,11.80) --
	(11.70,11.70) -- (11.60,11.70) -- (11.60,11.60) -- (11.50,11.60) --
	(11.50,11.50) -- (11.40,11.50) -- (11.40,11.40) -- (11.30,11.40) --
	(11.30,11.30) -- (11.20,11.30) -- (11.20,11.20) -- (11.10,11.20) --
	(11.10,11.10) -- (11.00,11.10) -- (11.00,11.00) -- (10.90,11.00) --
	(10.90,10.90) -- (10.80,10.90) -- (10.80,10.80) -- (10.70,10.80) --
	(10.70,10.70) -- (10.60,10.70) -- (10.60,10.60) -- (10.50,10.60) --
	(10.50,10.50) -- (10.40,10.50) -- (10.40,10.40) -- (10.30,10.40) --
	(10.30,10.30) -- (10.20,10.30) -- (10.20,10.20) -- (10.10,10.20) --
	(10.10,10.10) -- (10.00,10.10) -- (10.00,10.00) -- (9.90,10.00) -- (9.90,9.90)
	-- (9.80,9.90) -- (9.80,9.80) -- (9.70,9.80) -- (9.70,9.70) -- (9.60,9.70) --
	(9.60,9.60) -- (9.50,9.60) -- (9.50,9.50) -- (9.40,9.50) -- (9.40,9.40) --
	(9.30,9.40) -- (9.30,9.30) -- (9.20,9.30) -- (9.20,9.20) -- (9.10,9.20) --
	(9.10,9.10) -- (9.00,9.10) -- (9.00,9.00) -- (8.90,9.00) -- (8.90,8.90) --
	(8.80,8.90) -- (8.80,8.80) -- (8.70,8.80) -- (8.70,8.70) -- (8.60,8.70) --
	(8.60,8.60) -- (8.50,8.60) -- (8.50,8.50) -- (8.40,8.50) -- (8.40,8.40) --
	(8.30,8.40) -- (8.30,8.30) -- (8.20,8.30) -- (8.20,8.20) -- (8.10,8.20) --
	(8.10,8.10) -- (8.00,8.10) -- (8.00,8.00) -- (7.90,8.00) -- (7.90,7.90) --
	(7.80,7.90) -- (7.80,7.80) -- (7.70,7.80) -- (7.70,7.70) -- (7.60,7.70) --
	(7.60,7.60) -- (7.50,7.60) -- (7.50,7.50) -- (7.40,7.50) -- (7.40,7.40) --
	(7.30,7.40) -- (7.30,7.30) -- (7.20,7.30) -- (7.20,7.20) -- (7.10,7.20) --
	(7.10,7.10) -- (7.00,7.10) -- (7.00,7.00) -- (6.90,7.00) -- (6.90,6.90) --
	(6.80,6.90) -- (6.80,6.80) -- (6.70,6.80) -- (6.70,6.70) -- (6.60,6.70) --
	(6.60,6.60) -- (6.50,6.60) -- (6.50,6.50) -- (6.40,6.50) -- (6.40,6.40) --
	(6.30,6.40) -- (6.30,6.30) -- (6.20,6.30) -- (6.20,6.20) -- (6.10,6.20) --
	(6.10,6.10) -- (6.00,6.10) -- (6.00,6.00) -- (5.90,6.00) -- (5.90,5.90) --
	(5.80,5.90) -- (5.80,5.80) -- (5.70,5.80) -- (5.70,5.70) -- (5.60,5.70) --
	(5.60,5.60) -- (5.50,5.60) -- (5.50,5.50) -- (5.40,5.50) -- (5.40,5.40) --
	(5.30,5.40) -- (5.30,5.30) -- (5.20,5.30) -- (5.20,5.20) -- (5.10,5.20) --
	(5.10,5.10) -- (5.00,5.10) -- (5.00,5.00) -- (4.90,5.00) -- (4.90,4.90) --
	(4.80,4.90) -- (4.80,4.80) -- (4.70,4.80) -- (4.70,4.70) -- (4.60,4.70) --
	(4.60,4.60) -- (4.50,4.60) -- (4.50,4.50) -- (4.40,4.50) -- (4.40,4.40) --
	(4.30,4.40) -- (4.30,4.30) -- (4.20,4.30) -- (4.20,4.20) -- (4.10,4.20) --
	(4.10,4.10) -- (4.00,4.10) -- (4.00,4.00) -- (3.90,4.00) -- (3.90,3.90) --
	(3.80,3.90) -- (3.80,3.80) -- (3.70,3.80) -- (3.70,3.70) -- (3.60,3.70) --
	(3.60,3.60) -- (3.50,3.60) -- (3.50,3.50) -- (3.40,3.50) -- (3.40,3.40) --
	(3.30,3.40) -- (3.30,3.30) -- (3.20,3.30) -- (3.20,3.20) -- (3.10,3.20) --
	(3.10,3.10) -- (3.00,3.10) -- (3.00,3.00) -- (2.90,3.00) -- (2.90,2.90) --
	(2.80,2.90) -- (2.80,2.80) -- (2.70,2.80) -- (2.70,2.70) -- (2.60,2.70) --
	(2.60,2.60) -- (2.50,2.60) -- (2.50,2.50) -- (2.40,2.50) -- (2.40,2.40) --
	(2.30,2.40) -- (2.30,2.30) -- (2.20,2.30) -- (2.20,2.20) -- (2.10,2.20) --
	(2.10,2.10) -- (2.00,2.10) -- (2.00,2.00) -- (1.90,2.00) -- (1.90,1.90) --
	(1.80,1.90) -- (1.80,1.80) -- (1.70,1.80) -- (1.70,1.70) -- (1.60,1.70) --
	(1.60,1.60) -- (1.50,1.60) -- (1.50,1.50) -- (1.40,1.50) -- (1.40,1.40) --
	(1.30,1.40) -- (1.30,1.30) -- (1.20,1.30) -- (1.20,1.20) -- (1.10,1.20) --
	(1.10,1.10) -- (1.00,1.10) -- (1.00,1.00) -- (0.90,1.00) -- (0.90,0.90) --
	(0.80,0.90) -- (0.80,0.80) -- (0.70,0.80) -- (0.70,0.70) -- (0.60,0.70) --
	(0.60,0.60) -- (0.50,0.60) -- (0.50,0.50) -- (0.40,0.50) -- (0.40,0.40) --
	(0.30,0.40) -- (0.30,0.30) -- (0.20,0.30) -- (0.20,0.20) -- (0.10,0.20) --
	(0.10,0.10) -- (0.00,0.10) -- (0.00,0.00) -- (93.10,0.00) -- (93.10,0.10) --
	(90.50,0.10) -- (90.50,0.20) -- (88.80,0.20) -- (88.80,0.30) -- (88.00,0.30) --
	(88.00,0.40) -- (87.50,0.40) -- (87.50,0.50) -- (85.70,0.50) -- (85.70,0.60) --
	(84.80,0.60) -- (84.80,0.70) -- (84.30,0.70) -- (84.30,0.80) -- (83.50,0.80) --
	(83.50,0.90) -- (83.20,0.90) -- (83.20,1.00) -- (82.60,1.00) -- (82.60,1.10) --
	(81.00,1.10) -- (81.00,1.20) -- (80.50,1.20) -- (80.50,1.30) -- (80.00,1.30) --
	(80.00,1.40) -- (78.60,1.40) -- (78.60,1.50) -- (78.50,1.50) -- (78.50,1.60) --
	(77.70,1.60) -- (77.70,1.70) -- (76.80,1.70) -- (76.80,1.80) -- (76.00,1.80) --
	(76.00,1.90) -- (75.40,1.90) -- (75.40,2.00) -- (74.40,2.00) -- (74.40,2.20) --
	(74.10,2.20) -- (74.10,2.30) -- (72.80,2.30) -- (72.80,2.40) -- (72.50,2.40) --
	(72.50,2.50) -- (72.00,2.50) -- (72.00,2.60) -- (71.90,2.60) -- (71.90,2.70) --
	(71.30,2.70) -- (71.30,2.80) -- (70.90,2.80) -- (70.90,2.90) -- (70.40,2.90) --
	(70.40,3.00) -- (69.70,3.00) -- (69.70,3.10) -- (69.40,3.10) -- (69.40,3.20) --
	(69.30,3.20) -- (69.30,3.30) -- (68.50,3.30) -- (68.50,3.40) -- (68.20,3.40) --
	(68.20,3.50) -- (67.90,3.50) -- (67.90,3.60) -- (67.20,3.60) -- (67.20,3.70) --
	(66.80,3.70) -- (66.80,3.80) -- (66.40,3.80) -- (66.40,3.90) -- (65.90,3.90) --
	(65.90,4.00) -- (65.70,4.00) -- (65.70,4.10) -- (65.10,4.10) -- (65.10,4.30) --
	(64.40,4.30) -- (64.40,4.40) -- (64.00,4.40) -- (64.00,4.50) -- (63.70,4.50) --
	(63.70,4.60) -- (63.60,4.60) -- (63.60,4.70) -- (63.00,4.70) -- (63.00,4.80) --
	(62.70,4.80) -- (62.70,4.90) -- (62.20,4.90) -- (62.20,5.00) -- (61.80,5.00) --
	(61.80,5.10) -- (61.40,5.10) -- (61.40,5.20) -- (61.10,5.20) -- (61.10,5.30) --
	(60.80,5.30) -- (60.80,5.40) -- (60.30,5.40) -- (60.30,5.50) -- (59.60,5.50) --
	(59.60,5.70) -- (59.40,5.70) -- (59.40,5.80) -- (58.90,5.80) -- (58.90,5.90) --
	(58.50,5.90) -- (58.50,6.00) -- (58.10,6.00) -- (58.10,6.10) -- (57.70,6.10) --
	(57.70,6.20) -- (57.30,6.20) -- (57.30,6.30) -- (56.50,6.30) -- (56.50,6.40) --
	(56.40,6.40) -- (56.40,6.50) -- (56.20,6.50) -- (56.20,6.60) -- (55.70,6.60) --
	(55.70,6.70) -- (55.40,6.70) -- (55.40,6.80) -- (55.00,6.80) -- (55.00,6.90) --
	(54.70,6.90) -- (54.70,7.00) -- (54.30,7.00) -- (54.30,7.20) -- (53.90,7.20) --
	(53.90,7.30) -- (53.70,7.30) -- (53.70,7.40) -- (53.30,7.40) -- (53.30,7.50) --
	(53.10,7.50) -- (53.10,7.60) -- (52.40,7.60) -- (52.40,7.80) -- (52.10,7.80) --
	(52.10,7.90) -- (51.90,7.90) -- (51.90,8.00) -- (51.40,8.00) -- (51.40,8.10) --
	(51.10,8.10) -- (51.10,8.20) -- (51.00,8.20) -- (51.00,8.30) -- (50.40,8.30) --
	(50.40,8.40) -- (49.90,8.40) -- (49.90,8.60) -- (49.60,8.60) -- (49.60,8.70) --
	(49.10,8.70) -- (49.10,8.80) -- (48.90,8.80) -- (48.90,8.90) -- (48.10,8.90) --
	(48.10,9.00) -- (47.90,9.00) -- (47.90,9.10) -- (47.80,9.10) -- (47.80,9.20) --
	(47.50,9.20) -- (47.50,9.30) -- (47.20,9.30) -- (47.20,9.40) -- (47.00,9.40) --
	(47.00,9.50) -- (46.90,9.50) -- (46.90,9.60) -- (46.60,9.60) -- (46.60,9.70) --
	(46.30,9.70) -- (46.30,9.80) -- (46.10,9.80) -- (46.10,10.00) -- (45.80,10.00)
	-- (45.80,10.10) -- (45.60,10.10) -- (45.60,10.20) -- (45.30,10.20) --
	(45.30,10.30) -- (45.20,10.30) -- (45.20,10.40) -- (44.80,10.40) --
	(44.80,10.50) -- (44.40,10.50) -- (44.40,10.70) -- (44.20,10.70) --
	(44.20,10.80) -- (43.60,10.80) -- (43.60,10.90) -- (43.30,10.90) --
	(43.30,11.00) -- (43.00,11.00) -- (43.00,11.10) -- (42.90,11.10) --
	(42.90,11.20) -- (42.50,11.20) -- (42.50,11.30) -- (42.20,11.30) --
	(42.20,11.40) -- (42.10,11.40) -- (42.10,11.50) -- (42.00,11.50) --
	(42.00,11.60) -- (41.70,11.60) -- (41.70,11.70) -- (41.60,11.70) --
	(41.60,11.80) -- (41.40,11.80) -- (41.40,11.90) -- (41.20,11.90) --
	(41.20,12.00) -- (40.70,12.00) -- (40.70,12.20) -- (40.50,12.20) --
	(40.50,12.30) -- (40.10,12.30) -- (40.10,12.40) -- (39.80,12.40) --
	(39.80,12.50) -- (39.60,12.50) -- (39.60,12.60) -- (39.40,12.60) --
	(39.40,12.70) -- (39.30,12.70) -- (39.30,12.80) -- (39.10,12.80) --
	(39.10,12.90) -- (38.90,12.90) -- (38.90,13.00) -- (38.40,13.00) --
	(38.40,13.10) -- (38.20,13.10) -- (38.20,13.20) -- (38.00,13.20) --
	(38.00,13.30) -- (37.80,13.30) -- (37.80,13.40) -- (37.70,13.40) --
	(37.70,13.50) -- (37.50,13.50) -- (37.50,13.60) -- (37.30,13.60) --
	(37.30,13.70) -- (37.00,13.70) -- (37.00,13.90) -- (36.60,13.90) --
	(36.60,14.00) -- (36.40,14.00) -- (36.40,14.10) -- (36.30,14.10) --
	(36.30,14.20) -- (35.90,14.20) -- (35.90,14.40) -- (35.70,14.40) --
	(35.70,14.50) -- (35.40,14.50) -- (35.40,14.60) -- (35.30,14.60) --
	(35.30,14.70) -- (35.10,14.70) -- (35.10,14.80) -- (34.80,14.80) --
	(34.80,14.90) -- (34.70,14.90) -- (34.70,15.00) -- (34.50,15.00) --
	(34.50,15.10) -- (34.30,15.10) -- (34.30,15.20) -- (34.20,15.20) --
	(34.20,15.30) -- (34.00,15.30) -- (34.00,15.40) -- (33.80,15.40) --
	(33.80,15.50) -- (33.50,15.50) -- (33.50,15.70) -- (33.20,15.70) --
	(33.20,15.80) -- (33.10,15.80) -- (33.10,16.00) -- (32.60,16.00) --
	(32.60,16.10) -- (32.40,16.10) -- (32.40,16.20) -- (32.20,16.20) --
	(32.20,16.30) -- (32.10,16.30) -- (32.10,16.50) -- (31.80,16.50) --
	(31.80,16.60) -- (31.70,16.60) -- (31.70,16.70) -- (31.50,16.70) --
	(31.50,16.90) -- (31.40,16.90) -- (31.40,17.00) -- (31.20,17.00) --
	(31.20,17.10) -- (31.10,17.10) -- (31.10,17.20) -- (30.90,17.20) --
	(30.90,17.30) -- (30.70,17.30) -- (30.70,17.40) -- (30.50,17.40) --
	(30.50,17.50) -- (30.30,17.50) -- (30.30,17.60) -- (30.20,17.60) --
	(30.20,17.70) -- (30.10,17.70) -- (30.10,17.80) -- (29.90,17.80) --
	(29.90,17.90) -- (29.80,17.90) -- (29.80,18.00) -- (29.50,18.00) --
	(29.50,18.10) -- (29.30,18.10) -- (29.30,18.30) -- (29.20,18.30) --
	(29.20,18.40) -- (29.00,18.40) -- (29.00,18.50) -- (28.70,18.50) --
	(28.70,18.60) -- (28.50,18.60) -- (28.50,18.70) -- (28.40,18.70) --
	(28.40,18.80) -- (28.20,18.80) -- (28.20,18.90) -- (28.10,18.90) --
	(28.10,19.00) -- (27.90,19.00) -- (27.90,19.10) -- (27.80,19.10) --
	(27.80,19.20) -- (27.70,19.20) -- (27.70,19.30) -- (27.50,19.30) --
	(27.50,19.40) -- (27.40,19.40) -- (27.40,19.50) -- (27.30,19.50) --
	(27.30,19.60) -- (27.20,19.60) -- (27.20,19.70) -- (27.10,19.70) --
	(27.10,19.80) -- (26.90,19.80) -- (26.90,20.00) -- (26.50,20.00) --
	(26.50,20.30) -- (26.40,20.30) -- (26.40,20.40) -- (26.20,20.40) --
	(26.20,20.50) -- (26.10,20.50) -- (26.10,20.60) -- (26.00,20.60) --
	(26.00,20.70) -- (25.90,20.70) -- (25.90,20.80) -- (25.70,20.80) --
	(25.70,20.90) -- (25.60,20.90) -- (25.60,21.00) -- (25.50,21.00) --
	(25.50,21.10) -- (25.40,21.10) -- (25.40,21.20) -- (25.20,21.20) --
	(25.20,21.30) -- (25.10,21.30) -- (25.10,21.40) -- (25.00,21.40) --
	(25.00,21.50) -- (24.80,21.50) -- (24.80,21.60) -- (24.70,21.60) --
	(24.70,21.80) -- (24.60,21.80) -- (24.60,22.00) -- (24.30,22.00) --
	(24.30,22.10) -- (24.10,22.10) -- (24.10,22.30) -- (23.90,22.30) --
	(23.90,22.50) -- (23.80,22.50) -- (23.80,22.70) -- (23.60,22.70) --
	(23.60,22.80) -- (23.50,22.80) -- (23.50,22.90) -- (23.40,22.90) --
	(23.40,23.00) -- (23.10,23.00) -- (23.10,23.10) -- cycle;
	\end{scope}
	\end{scope}

	\draw[red,ultra thick,opacity=0.4] plot[smooth] file {FIGURES/SchurWeyl-c=1.txt};
	
	\end{tikzpicture}

	\caption{The thick red line: the shape $\Omega^{\SWW}_{c}$ for the special case
	$c=1$ obtained from \eqref{eq:omega-concrete} by numerical integration. The blue
	area: scaled down random shifted partition (shown in the shifted Russian
	convention) sampled for the shifted Schur--Weyl measure $\SW{n}{d}$ for
	$n=80000$, $d=283$ and $c=\frac{\sqrt{n}}{d}\approx 1$.}
	
	\label{fig:SW-theoretic}	
\end{figure}
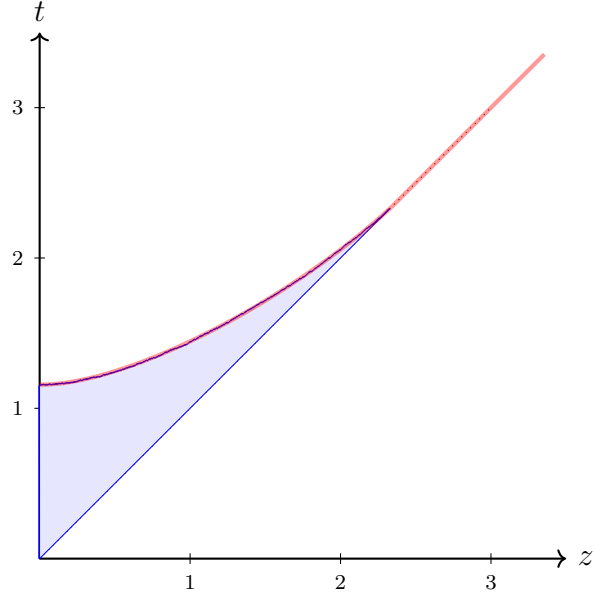

\begin{figure}
	\begin{tikzpicture}[scale=2]
	
	\draw [thick,->] (0,0) -- (3.5,0) node[anchor=west] {$z$};
	\draw [thick,->] (0,0) -- (0,3.5) node[anchor=south] {$t$};
	
	\foreach \x in {1,2,3}
	\draw (\x cm,1pt) -- (\x cm,-1pt) node[anchor=north] {\tiny $\x$};
	\foreach \y in {1,2,3}
	\draw (1pt,\y cm) -- (-1pt,\y cm) node[anchor=east] {\tiny $\y$};

	\draw[dotted] (0,0) -- (3,3);

	\begin{scope}[rotate=45,scale=sqrt(2)/40]
	\begin{scope}[xshift=-0.05cm]
	\draw[blue,fill=blue!10]
(14.00,14.10) -- (14.00,14.00) -- (13.90,14.00) -- (13.90,13.90) -- (13.80,13.90) -- (13.80,13.80) -- (13.70,13.80) -- (13.70,13.70) -- (13.60,13.70) -- (13.60,13.60) -- (13.50,13.60) -- (13.50,13.50) -- (13.40,13.50) -- (13.40,13.40) -- (13.30,13.40) -- (13.30,13.30) -- (13.20,13.30) -- (13.20,13.20) -- (13.10,13.20) -- (13.10,13.10) -- (13.00,13.10) -- (13.00,13.00) -- (12.90,13.00) -- (12.90,12.90) -- (12.80,12.90) -- (12.80,12.80) -- (12.70,12.80) -- (12.70,12.70) -- (12.60,12.70) -- (12.60,12.60) -- (12.50,12.60) -- (12.50,12.50) -- (12.40,12.50) -- (12.40,12.40) -- (12.30,12.40) -- (12.30,12.30) -- (12.20,12.30) -- (12.20,12.20) -- (12.10,12.20) -- (12.10,12.10) -- (12.00,12.10) -- (12.00,12.00) -- (11.90,12.00) -- (11.90,11.90) -- (11.80,11.90) -- (11.80,11.80) -- (11.70,11.80) -- (11.70,11.70) -- (11.60,11.70) -- (11.60,11.60) -- (11.50,11.60) -- (11.50,11.50) -- (11.40,11.50) -- (11.40,11.40) -- (11.30,11.40) -- (11.30,11.30) -- (11.20,11.30) -- (11.20,11.20) -- (11.10,11.20) -- (11.10,11.10) -- (11.00,11.10) -- (11.00,11.00) -- (10.90,11.00) -- (10.90,10.90) -- (10.80,10.90) -- (10.80,10.80) -- (10.70,10.80) -- (10.70,10.70) -- (10.60,10.70) -- (10.60,10.60) -- (10.50,10.60) -- (10.50,10.50) -- (10.40,10.50) -- (10.40,10.40) -- (10.30,10.40) -- (10.30,10.30) -- (10.20,10.30) -- (10.20,10.20) -- (10.10,10.20) -- (10.10,10.10) -- (10.00,10.10) -- (10.00,10.00) -- (9.90,10.00) -- (9.90,9.90) -- (9.80,9.90) -- (9.80,9.80) -- (9.70,9.80) -- (9.70,9.70) -- (9.60,9.70) -- (9.60,9.60) -- (9.50,9.60) -- (9.50,9.50) -- (9.40,9.50) -- (9.40,9.40) -- (9.30,9.40) -- (9.30,9.30) -- (9.20,9.30) -- (9.20,9.20) -- (9.10,9.20) -- (9.10,9.10) -- (9.00,9.10) -- (9.00,9.00) -- (8.90,9.00) -- (8.90,8.90) -- (8.80,8.90) -- (8.80,8.80) -- (8.70,8.80) -- (8.70,8.70) -- (8.60,8.70) -- (8.60,8.60) -- (8.50,8.60) -- (8.50,8.50) -- (8.40,8.50) -- (8.40,8.40) -- (8.30,8.40) -- (8.30,8.30) -- (8.20,8.30) -- (8.20,8.20) -- (8.10,8.20) -- (8.10,8.10) -- (8.00,8.10) -- (8.00,8.00) -- (7.90,8.00) -- (7.90,7.90) -- (7.80,7.90) -- (7.80,7.80) -- (7.70,7.80) -- (7.70,7.70) -- (7.60,7.70) -- (7.60,7.60) -- (7.50,7.60) -- (7.50,7.50) -- (7.40,7.50) -- (7.40,7.40) -- (7.30,7.40) -- (7.30,7.30) -- (7.20,7.30) -- (7.20,7.20) -- (7.10,7.20) -- (7.10,7.10) -- (7.00,7.10) -- (7.00,7.00) -- (6.90,7.00) -- (6.90,6.90) -- (6.80,6.90) -- (6.80,6.80) -- (6.70,6.80) -- (6.70,6.70) -- (6.60,6.70) -- (6.60,6.60) -- (6.50,6.60) -- (6.50,6.50) -- (6.40,6.50) -- (6.40,6.40) -- (6.30,6.40) -- (6.30,6.30) -- (6.20,6.30) -- (6.20,6.20) -- (6.10,6.20) -- (6.10,6.10) -- (6.00,6.10) -- (6.00,6.00) -- (5.90,6.00) -- (5.90,5.90) -- (5.80,5.90) -- (5.80,5.80) -- (5.70,5.80) -- (5.70,5.70) -- (5.60,5.70) -- (5.60,5.60) -- (5.50,5.60) -- (5.50,5.50) -- (5.40,5.50) -- (5.40,5.40) -- (5.30,5.40) -- (5.30,5.30) -- (5.20,5.30) -- (5.20,5.20) -- (5.10,5.20) -- (5.10,5.10) -- (5.00,5.10) -- (5.00,5.00) -- (4.90,5.00) -- (4.90,4.90) -- (4.80,4.90) -- (4.80,4.80) -- (4.70,4.80) -- (4.70,4.70) -- (4.60,4.70) -- (4.60,4.60) -- (4.50,4.60) -- (4.50,4.50) -- (4.40,4.50) -- (4.40,4.40) -- (4.30,4.40) -- (4.30,4.30) -- (4.20,4.30) -- (4.20,4.20) -- (4.10,4.20) -- (4.10,4.10) -- (4.00,4.10) -- (4.00,4.00) -- (3.90,4.00) -- (3.90,3.90) -- (3.80,3.90) -- (3.80,3.80) -- (3.70,3.80) -- (3.70,3.70) -- (3.60,3.70) -- (3.60,3.60) -- (3.50,3.60) -- (3.50,3.50) -- (3.40,3.50) -- (3.40,3.40) -- (3.30,3.40) -- (3.30,3.30) -- (3.20,3.30) -- (3.20,3.20) -- (3.10,3.20) -- (3.10,3.10) -- (3.00,3.10) -- (3.00,3.00) -- (2.90,3.00) -- (2.90,2.90) -- (2.80,2.90) -- (2.80,2.80) -- (2.70,2.80) -- (2.70,2.70) -- (2.60,2.70) -- (2.60,2.60) -- (2.50,2.60) -- (2.50,2.50) -- (2.40,2.50) -- (2.40,2.40) -- (2.30,2.40) -- (2.30,2.30) -- (2.20,2.30) -- (2.20,2.20) -- (2.10,2.20) -- (2.10,2.10) -- (2.00,2.10) -- (2.00,2.00) -- (1.90,2.00) -- (1.90,1.90) -- (1.80,1.90) -- (1.80,1.80) -- (1.70,1.80) -- (1.70,1.70) -- (1.60,1.70) -- (1.60,1.60) -- (1.50,1.60) -- (1.50,1.50) -- (1.40,1.50) -- (1.40,1.40) -- (1.30,1.40) -- (1.30,1.30) -- (1.20,1.30) -- (1.20,1.20) -- (1.10,1.20) -- (1.10,1.10) -- (1.00,1.10) -- (1.00,1.00) -- (0.90,1.00) -- (0.90,0.90) -- (0.80,0.90) -- (0.80,0.80) -- (0.70,0.80) -- (0.70,0.70) -- (0.60,0.70) -- (0.60,0.60) -- (0.50,0.60) -- (0.50,0.50) -- (0.40,0.50) -- (0.40,0.40) -- (0.30,0.40) -- (0.30,0.30) -- (0.20,0.30) -- (0.20,0.20) -- (0.10,0.20) -- (0.10,0.10) -- (0.00,0.10) -- (0.00,0.00) -- (115.80,0.00) -- (115.80,0.10) -- (114.10,0.10) -- (114.10,0.20) -- (113.20,0.20) -- (113.20,0.30) -- (111.60,0.30) -- (111.60,0.40) -- (108.80,0.40) -- (108.80,0.50) -- (107.90,0.50) -- (107.90,0.60) -- (106.50,0.60) -- (106.50,0.70) -- (105.90,0.70) -- (105.90,0.80) -- (104.20,0.80) -- (104.20,0.90) -- (101.80,0.90) -- (101.80,1.00) -- (100.80,1.00) -- (100.80,1.10) -- (100.40,1.10) -- (100.40,1.20) -- (99.10,1.20) -- (99.10,1.30) -- (98.00,1.30) -- (98.00,1.40) -- (97.70,1.40) -- (97.70,1.50) -- (96.30,1.50) -- (96.30,1.60) -- (95.30,1.60) -- (95.30,1.70) -- (94.70,1.70) -- (94.70,1.80) -- (94.10,1.80) -- (94.10,1.90) -- (92.70,1.90) -- (92.70,2.00) -- (92.30,2.00) -- (92.30,2.10) -- (91.50,2.10) -- (91.50,2.20) -- (90.70,2.20) -- (90.70,2.30) -- (90.30,2.30) -- (90.30,2.40) -- (89.10,2.40) -- (89.10,2.50) -- (88.80,2.50) -- (88.80,2.60) -- (88.00,2.60) -- (88.00,2.80) -- (86.80,2.80) -- (86.80,2.90) -- (85.80,2.90) -- (85.80,3.00) -- (85.10,3.00) -- (85.10,3.10) -- (84.10,3.10) -- (84.10,3.20) -- (83.60,3.20) -- (83.60,3.30) -- (82.90,3.30) -- (82.90,3.40) -- (82.70,3.40) -- (82.70,3.50) -- (81.70,3.50) -- (81.70,3.60) -- (81.30,3.60) -- (81.30,3.70) -- (80.20,3.70) -- (80.20,3.80) -- (79.80,3.80) -- (79.80,3.90) -- (79.40,3.90) -- (79.40,4.00) -- (78.90,4.00) -- (78.90,4.10) -- (77.90,4.10) -- (77.90,4.20) -- (76.90,4.20) -- (76.90,4.30) -- (76.10,4.30) -- (76.10,4.40) -- (76.00,4.40) -- (76.00,4.50) -- (75.30,4.50) -- (75.30,4.60) -- (74.90,4.60) -- (74.90,4.70) -- (73.90,4.70) -- (73.90,4.80) -- (73.80,4.80) -- (73.80,4.90) -- (72.60,4.90) -- (72.60,5.00) -- (71.90,5.00) -- (71.90,5.10) -- (71.40,5.10) -- (71.40,5.20) -- (70.70,5.20) -- (70.70,5.30) -- (70.60,5.30) -- (70.60,5.40) -- (69.70,5.40) -- (69.70,5.50) -- (69.40,5.50) -- (69.40,5.60) -- (68.80,5.60) -- (68.80,5.70) -- (67.90,5.70) -- (67.90,5.80) -- (67.50,5.80) -- (67.50,5.90) -- (67.20,5.90) -- (67.20,6.00) -- (66.60,6.00) -- (66.60,6.10) -- (65.90,6.10) -- (65.90,6.20) -- (65.20,6.20) -- (65.20,6.30) -- (64.90,6.30) -- (64.90,6.40) -- (64.10,6.40) -- (64.10,6.50) -- (63.50,6.50) -- (63.50,6.60) -- (63.10,6.60) -- (63.10,6.70) -- (62.70,6.70) -- (62.70,6.80) -- (62.10,6.80) -- (62.10,6.90) -- (61.80,6.90) -- (61.80,7.00) -- (61.40,7.00) -- (61.40,7.10) -- (60.60,7.10) -- (60.60,7.20) -- (60.40,7.20) -- (60.40,7.30) -- (60.20,7.30) -- (60.20,7.40) -- (59.50,7.40) -- (59.50,7.50) -- (59.20,7.50) -- (59.20,7.60) -- (58.60,7.60) -- (58.60,7.70) -- (58.30,7.70) -- (58.30,7.80) -- (57.50,7.80) -- (57.50,7.90) -- (57.10,7.90) -- (57.10,8.00) -- (56.10,8.00) -- (56.10,8.10) -- (55.70,8.10) -- (55.70,8.20) -- (55.20,8.20) -- (55.20,8.30) -- (54.80,8.30) -- (54.80,8.50) -- (53.80,8.50) -- (53.80,8.60) -- (53.30,8.60) -- (53.30,8.70) -- (53.10,8.70) -- (53.10,8.80) -- (52.30,8.80) -- (52.30,8.90) -- (51.90,8.90) -- (51.90,9.00) -- (51.30,9.00) -- (51.30,9.10) -- (51.00,9.10) -- (51.00,9.20) -- (50.20,9.20) -- (50.20,9.30) -- (50.10,9.30) -- (50.10,9.40) -- (49.40,9.40) -- (49.40,9.50) -- (48.90,9.50) -- (48.90,9.60) -- (48.80,9.60) -- (48.80,9.70) -- (48.10,9.70) -- (48.10,9.80) -- (47.90,9.80) -- (47.90,9.90) -- (47.20,9.90) -- (47.20,10.00) -- (46.50,10.00) -- (46.50,10.10) -- (46.40,10.10) -- (46.40,10.20) -- (45.30,10.20) -- (45.30,10.30) -- (44.90,10.30) -- (44.90,10.40) -- (44.60,10.40) -- (44.60,10.50) -- (44.30,10.50) -- (44.30,10.60) -- (43.60,10.60) -- (43.60,10.70) -- (43.10,10.70) -- (43.10,10.80) -- (42.90,10.80) -- (42.90,10.90) -- (41.80,10.90) -- (41.80,11.00) -- (41.70,11.00) -- (41.70,11.10) -- (41.10,11.10) -- (41.10,11.20) -- (40.90,11.20) -- (40.90,11.30) -- (39.90,11.30) -- (39.90,11.40) -- (39.80,11.40) -- (39.80,11.50) -- (39.30,11.50) -- (39.30,11.60) -- (38.40,11.60) -- (38.40,11.70) -- (38.20,11.70) -- (38.20,11.80) -- (37.70,11.80) -- (37.70,11.90) -- (37.30,11.90) -- (37.30,12.00) -- (36.70,12.00) -- (36.70,12.10) -- (36.00,12.10) -- (36.00,12.20) -- (35.80,12.20) -- (35.80,12.30) -- (35.60,12.30) -- (35.60,12.40) -- (35.30,12.40) -- (35.30,12.50) -- (34.20,12.50) -- (34.20,12.60) -- (34.10,12.60) -- (34.10,12.70) -- (33.10,12.70) -- (33.10,12.80) -- (32.60,12.80) -- (32.60,12.90) -- (31.90,12.90) -- (31.90,13.00) -- (31.20,13.00) -- (31.20,13.10) -- (30.50,13.10) -- (30.50,13.20) -- (30.30,13.20) -- (30.30,13.30) -- (29.70,13.30) -- (29.70,13.40) -- (29.40,13.40) -- (29.40,13.50) -- (29.00,13.50) -- (29.00,13.60) -- (28.20,13.60) -- (28.20,13.70) -- (27.80,13.70) -- (27.80,13.80) -- (27.20,13.80) -- (27.20,13.90) -- (26.00,13.90) -- (26.00,14.00) -- (25.70,14.00) -- (25.70,14.10) -- cycle;
	\end{scope}
	\end{scope}

	\draw[red,ultra thick,opacity=0.4] plot[smooth] file {FIGURES/Schur-Weyl-c=2.txt};
	
	\end{tikzpicture}

	\caption{The thick red line: the shape $\Omega^{\SWW}_{c}$ for the special case
		$c=2$ obtained from \eqref{eq:omega-concrete} by numerical integration. The blue
		area: scaled down random shifted partition (shown in the shifted Russian
		convention) sampled for the shifted Schur--Weyl measure $\SW{n}{d}$
		for $n=80000$, $d=141$ and $c=\frac{\sqrt{n}}{d}\approx
		2$.}
	
	\label{fig:SW-theoretic-2}	
\end{figure}
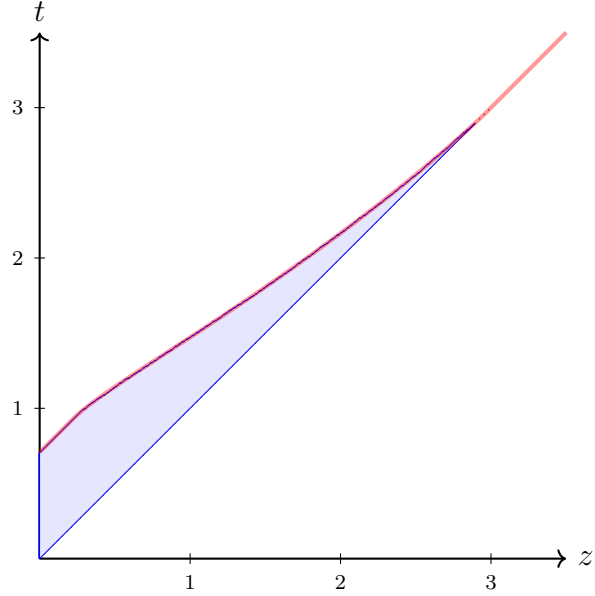

\subsubsection{Asymptotics of the insertion tableaux $P$}

\begin{figure}
	\centering \subfile{FIGURES/SchurWeyl-300-45000.tex} \caption{Insertion tableau
	obtained by applying shifted version of RSK algorithm to a random word of
	length $n=45000$ in the alphabet $\Alphabet_d$ with $d=300$. The boxes were
	drawn as squares of side $\frac{d}{n}=\frac{1}{150}$.
	In the context of	\cref{coro:schur-weyl-insertion} the Young diagram corresponds
	to the boxes with the rescaled height at most $t=\frac{d^2}{n} = 2$. 
	The level curves indicate positions of the boxes with rescaled height at most
	$t$ with: $\bullet$ \emph{the blue curve}: $t=1$, 
	$\bullet$ \emph{the burgundy curve}: $t=\frac{1}{2}$, 
	$\bullet$ \emph{the red curve}: $t=\frac{1}{4}$. 
} \label{fig:SW}
\end{figure} 

We continue the discussion of shifted RSK correspondence from
\cref{sec:shifted-RSK}. If we ignore that some of the entries of the insertion
tableau $P=P(\mathbf{w})$ are circled, we may represent $P$ as a stack of cubes,
just like we did it in \cref{sec:stacks-of-cubes}. This time, however, we
rescale all dimensions of the unit cubes by the factor $\frac{d}{n}$.

The following result states that such rescaled random surfaces \emph{converge in
probability} to some universal surface~$\LimitP$ \emph{in the sense of level curves}.
This result is
illustrated by a computer simulation on \cref{fig:SW}.

\begin{corollary}[Law of Large Numbers for insertion tableaux]
\label{coro:schur-weyl-insertion}
There exists a function 	
\[ \LimitP \colon \Big\{ (X,Y) : 0\leq Y \leq X \Big\}  \to \R_+ \] 
and the corresponding a family of level curves (drawn in the Russian convention)
indexed by $\alpha>0$, defined for $z\geq 0$ by
\[ \Omega^{\LimitP}_{\alpha}(z) = \sup \big\{ x+y: 0\leq y\leq x \text{ and }
x-y=z \text{ and } \LimitP(x,y)\leq \alpha \big\} \]
with the following property.

For any sequence $(d_n)$ of positive integers such that 
\[ 
\lim_{n\to\infty} \frac{d_n}{n} = 0\]
we have that
\[
\lim_{n\to\infty} \PP\Bigg\{ \mathbf{w}\in \Omega_{n,d} : 
\sup_{x\geq 0}
\left| \omega_{\frac{d}{n} \big( P(\mathbf{w}) \big)_{\leq \alpha \frac{n}{d} }}(x) 
      - \Omega^{\LimitP}_{\alpha}(x) \right| 
>\epsilon \Bigg\}  =0
\]
(for legibility we write $d=d_n$) holds true for any $\epsilon>0$ and any level $\alpha$ such that
\[ 0<\alpha<\liminf_{n\to\infty} \frac{d_n^2}{n}.\] 
\end{corollary}	

The proof is postponed to \cref{sec:example-SW}.

\subsubsection{Asymptotics of recording tableaux $Q$}

\begin{figure}
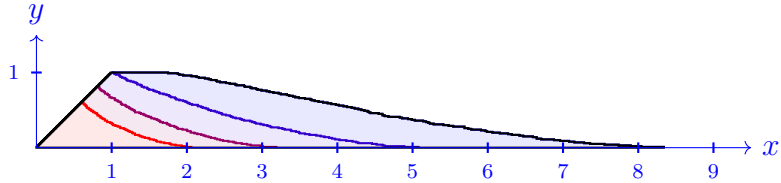

	\centering \subfile{FIGURES/recording_level_curves.tex} 
	
	\caption{Recording tableau obtained by applying shifted version of RSK
	algorithm to a random word of length $n=40000$ in the alphabet $\Alphabet_d$
	with $d=100$. The boxes were drawn as squares of side
	$\frac{1}{d}=\frac{1}{100}$. 
	In the context of \cref{coro:schur-weyl-recording} the Young diagram
	corresponds to the boxes with the rescaled height at most $t=\frac{n}{d^2} =
	2$. 
	The level curves indicate positions of the boxes with rescaled height at most
	$t$ with: $\bullet$ \emph{blue curve}: $t=1$, 
	$\bullet$ \emph{burgundy curve}: $t=\frac{1}{2}$ 
	$\bullet$ \emph{red curve}: $t=\frac{1}{4}$. 
} \label{fig:SW-recording}
\end{figure} 

Again, if we ignore that some of the entries of the recording tableau
$Q=Q(\mathbf{w})$ are circled, we may represent $Q$ as a stack of cubes, just
like we did it in \cref{sec:stacks-of-cubes}. This time, we rescale the unit
boxes on the plane by the factor $\frac{1}{d}$, and rescale the height of the
cubes by the factor $\frac{1}{d^2}$.

The following result states that such rescaled random surfaces converge \emph{in
probability} to some universal surface~$\LimitQ$ \emph{in the sense of the level curves}.
This result is
illustrated by a computer simulation on \cref{fig:SW-recording}.

\begin{corollary}[Law of Large Numbers for recording tableaux]
	\label{coro:schur-weyl-recording}
	There exists a function 	
	\[ \LimitQ \colon \Big\{ (X,Y) : 0\leq Y \leq 1, \ Y \leq X 
	\Big\}  \to \R_+ \] 
	and the corresponding a family of level curves (drawn in the Russian convention)
	indexed by $\alpha>0$, defined for $z\geq 0$ by
	\[ \Omega^{\LimitQ}_{\alpha}(z) = \sup \big\{ x+y: 0\leq y\leq x \text{ and } y\leq 1 \text{ and }
	x-y=z \text{ and } \LimitQ(x,y)\leq \alpha \big\} \]
	with the following property.

	For any sequence $(d_n)$ of positive integers such that 
	\[ \lim_{n\to\infty}  d_n = \infty\]
we have that
\[
\lim_{n\to\infty} \PP\Bigg\{ \mathbf{w}\in \Omega_{n,d} : 
\sup_{x\geq 0}
\left| \omega_{\frac{1}{d} \big( Q(\mathbf{w}) \big)_{\leq \alpha d^2 }}(x) - \Omega^{\LimitQ}_{\alpha}(x) \right| 
>\epsilon \Bigg\}  =0
\]
(for legibility we write $d=d_n$) holds true for any $\epsilon>0$ and any level $\alpha$ such that
\[ 0<\alpha<\liminf_{n\to\infty} \frac{n}{d_n^2}.\] 
\end{corollary}	

The proof is postponed to \cref{sec:example-SW}.

\begin{remark}
\cref{coro:schur-weyl-insertion} and \cref{coro:schur-weyl-recording} have
non-shifted analogues which concern the asymptotic shapes of the insertion and
the recording tableaux when the usual (non-shifted) RSK correspondence is
applied to a random sequence of length $n$ with the entries selected from the
finite set~$[d]$. These analogues follow from the results of Biane \cite[Theorem
3]{Biane2001} in a way similar to the one presented in the proofs of
\cref{coro:schur-weyl-insertion} and \cref{coro:schur-weyl-recording}. Somewhat
surprisingly, it seems that these non-shifted analogues were not stated
explicitly in the existing literature.
\end{remark}

The results of the current paper can be also used to show that the fluctuations
of the random surfaces corresponding to the insertion and the recording tableaux
around the limit shapes are Gaussian.

\section{Spin representations and random strict partitions}
\label{sec:representations-and-measures}

We are ready now to present the abstract general theory which provides the link
between asymptotics of spin characters and the corresponding random strict
partitions.

\subsection{Character ratios}

\label{sec:spin-characters-spin-group-overview} 

For Reader's convenience we collected some very basic facts and notations from the spin
representation theory of the symmetric groups in
\cref{sec:projective-representations}. Nevertheless, most of these facts will
not be used in the following. The bare minimal necessary knowledge is condensed
in \cref{appendix:conclusion}.

\medskip

If $\irrepSp\colon \SGA{n}\to\GL(V) $ is a spin representation and $\pi\in\OP_n$
is an odd partition, we define the corresponding \emph{character ratio}
\begin{equation} 
\label{eq:character-ratio}
\chi^\irrepSp(\pi):=  \tr_V \irrepSp(c^\pi) =
\frac{\Tr \irrepSp(c^\pi)}{\dimm V} 
\end{equation}
as the (normalized) character of $\irrepSp$ evaluated on any element $c^\pi\in
C_\pi^+$ which belongs to the conjugacy class which corresponds to $\pi$,
cf.~\cref{appendix:conclusion}. Above
\begin{equation}
\label{eq:normalized-trace}
\tr_V=\frac{1}{\dimm V} \Tr
\end{equation} 
denotes the 
\emph{normalized trace}.

\medskip

For $\xi\in\SP_n$ and $\pi\in\OP_n$ we denote by
\[ \chi^{\xi}(\pi) =  \frac{\phi^\xi(\pi)}{\phi^\xi(1^n)}\] 
the character ratio \eqref{eq:character-ratio} which corresponds to (any)
irreducible spin representation given by $\xi$, cf.~\cref{appendix:conclusion}.

\subsection{Random strict partitions and reducible representations}
\label{sec:random-strict-reducible}

Let $\irrepSp\colon G\to \GL(V)$ be a representation of a finite group $G$ and let
\[ V = \bigoplus_{\xi\in\widehat{G}} n_\xi V^\xi \]
be its decomposition into irreducible components. Above,
$n_\xi\in\{0,1,2,\dots\}$ denotes the multiplicity of the irreducible
component $V^\xi$. We define a probability measure on the set $\widehat{G}$ of
the 
irreducible representations of $G$ given by
\begin{equation} 
\label{eq:probability-and-dimensions}
\PP^V (\xi):=\frac{n_\xi \dimm V^\xi}{\dimm V} 
\end{equation}
which can be interpreted as the probability distribution of a 
\emph{random irreducible component of $V$}.

\medskip

The main results of the current paper 
concern the random shifted Young diagrams given by the above construction in the
special case when $G=\Spin{n}$ is the spin group and $V$ is its reducible spin
representation. Equivalently, in order to ensure we deal with a \emph{spin}
representation, we may consider the analogue of the above construction in which
instead of a \emph{group representation} of $\Spin{n}$ we use an \emph{algebra
	representation} of $\SGA{n}$. 

Due to the correspondence between the irreducible
representations of $\SGA{n}$ and strict partitions of $n$, the measure $\PP^V$
can be equivalently viewed as a probability measure on $\SP_n$, see
\cref{sec:measure-more-precise} for some technical details.

\begin{example}[Strict Plancherel measure]
	\label{example:strict-plancherel} 
	
	The vector space $V:=\SGA{n}$ admits a
	natural action of the spin group algebra $\SGA{n}$ by left multiplication which
	can be regarded as an analogue of the left-regular representation of a group.
	
	The corresponding probability measure, called \emph{strict (or shifted)
	Plan\-che\-rel measure} \cite{Borodin1997,Ivanov1999a,Ivanov2004}, 
	is given by
	\begin{equation}
	\label{eq:Plancherel} 
	\PP^{\textrm{Plancherel}}_n(\xi) =
	\PP^{\SGA{n}}(\xi) = 
	     \frac{2^{n-\ell(\xi)}\ (g^\xi)^2}{n!} \qquad \text{for } \xi\in\SP_n,
	\end{equation}
	where $g^\xi=|\tableaux_\xi|$ denotes the number of shifted standard tableaux
	with shape given by the shifted Young diagram $\xi$.
	
Equivalently, \eqref{eq:Plancherel} is the probability distribution of the
common shape of the two shifted tableaux associated via shifted
Robinson--Schensted correspondence \cite{Worley1984,Sagan1987} to a uniformly
random circled permutation in $n$ letters. 
\end{example}

\subsection{Random strict partitions, revisited}
\label{sec:measure-more-precise}

Formally speaking, the construction from \cref{sec:random-strict-reducible}
associates to a given (reducible) representation $V$ of $\SGA{n}$ a probability
measure $\PP^V$ on the set of irreducible representations (or irreducible
characters) of $\SGA{n}$ which can be identified with the set
\begin{multline} 
\label{eq:all-characters}
\SP_n^+ \sqcup \SP_n^- \sqcup \SP_n^- = \\
\{ \phi^\xi : \xi \in \SP_n^+\} 
\cup 
\{ \phi_+^\xi : \xi \in \SP_n^-\} 
\cup
\{ \phi_-^\xi : \xi \in \SP_n^-\} 
\end{multline}
in which each element of $\SP_n^-$ is counted \emph{twice}, thus it is  \emph{not}
equal to $\SP_n$. However, if we identify the two copies of $\SP^-_n$ by
identifying the characters $\phi^\xi_{\pm}$ for $\xi\in\SP^-$ then $\PP^V$
becomes, as required, a probability measure on $\SP_n$ given by
\begin{equation} 
\label{eq:decomposition}
\PP^V(\xi) = 
\begin{cases}
\PP^V(\phi^{\xi}) & \text{if } \xi\in\SP^+, \\
\PP^V(\phi^{\xi}_+)+\PP^V(\phi^{\xi}_-) & \text{if } \xi\in\SP^-.
\end{cases} 
\end{equation}

An alternative solution to the above minor difficulty is to start with a reducible
\emph{superrepresentation} of $\SGA{n}$ and then to decompose it into
irreducible \emph{superrepresentations} which directly gives rise to a
probability distribution on strict partitions.

\subsection{Random strict partitions, alternative viewpoint}
\label{sec:only-ratio-matters}

If $V$ is a representation of of $\SGA{n}$ then (keeping in mind
\eqref{eq:decomposition}) the following equality between functions on $\OP_n$
holds true:
\begin{equation}
\label{eq:linear-combination} 
\chi^V = \sum_{\xi\in\SP_n}   \mathbb{P}^V(\xi)\ \chi^\xi. 
\end{equation}
Thanks to \cref{lem:linear-bases} below it follows that the coefficients $\left(
\mathbb{P}^V(\xi) \right)$ of this expansion are uniquely determined by
\eqref{eq:linear-combination}. 

\begin{lemma}
	\label{lem:linear-bases}
	The family of character ratios 
	\[ \left\{ \chi^\xi : \xi\in\SP_n \right\} \] 
	forms a linear basis of the space of (complex-valued) functions on $\OP_n$.
\end{lemma}
\begin{proof}
	The right-hand side of \eqref{eq:all-characters} is the complete collection of
	the characters of $\SGA{n}$ 
	hence it forms a
	linear basis of the space of complex-valued functions on the set of conjugacy
	classes $C_\pi^+$ over $\pi\in\OP_n \cup \SP_n^-$.
	
	It follows that 
	\begin{equation} 
	\label{eq:small-characters}
	\left\{ \phi^\xi : \xi\in\SP_n \right\} 
	\end{equation}
	spans the space of complex-valued functions on the set of conjugacy classes
	$C_\pi^+$ over $\pi\in\OP_n$. By the identity $|\SP_n|=|\OP_n|$, its
	cardinality matches the dimension hence \eqref{eq:small-characters} is a linear
	basis of the latter space.
	
	Since 
	\[\chi^\xi=\frac{1}{\phi^\xi(1^n)} \phi^\xi,\] 
	the claim follows immediately.
\end{proof}

\medskip

\emph{Equality \eqref{eq:linear-combination} can be viewed as an alternative
	definition of the probabilities $\mathbb{P}^V(\xi)$ as coefficients of the
	expansion of $\chi^V$ in the linear basis $(\chi^\xi)$.} This viewpoint has
interesting consequences. 

\medskip

Firstly, in order to state the results of the current paper we do not need the
representation $V$ and it is enough to speak about the corresponding character
ratio $\chi^V$. The property of approximate factorization
(\cref{def:afp-representations}) is, in fact, not a property of a sequence of
representations but of the corresponding sequence of character ratios.

\medskip

Secondly, for a superrepresentation of $\SGA{n}$, the construction from
\cref{sec:random-strict-reducible} gives a probability measure directly
on $\SP_n$, without the difficulties discussed in
\cref{sec:measure-more-precise}. 
Passage from the framework of representations to the framework of
superrepresentations implies that we should replace the family of characters
\eqref{eq:all-characters} by the characters of the irreducible superrepresentations
\[ \{ \phi^\xi : \xi \in \SP_n^+\} 
\cup 
\{ \phi_+^\xi + \phi_-^\xi : \xi \in \SP_n^-\}.\]
This change does not create any difficulties because the corresponding family of
character ratios (viewed as functions on $\OP_n$) remains the same.
By revisiting Equation \eqref{eq:linear-combination}
in the new context of superrepresentations we see that it still remains valid;
this shows that that the probability measure on $\SP_n$ associated to a
(super)representation $V$ remains the same, no matter if we regard $V$ as a
superrepresentation or as a representation.

\medskip

Thirdly, the (super)representation theory of the spin symmetric groups $\Sym{n}$
is known to be Morita equivalent to the (super)representation theory of
Hecke--Clifford algebra $\mathcal{H}_n=\mathcal{C}l_n \rtimes \C\Sym{n}$. In the
context of the (super)representations of $\mathcal{H}_n$ it still makes sense to
speak about the character ratios $\chi$ as functions on $\OP_n$; in this way the
results of the current paper can be reformulated in the language of the
(super)representations of $\mathcal{H}_n$ and the corresponding probability
measures.

\subsection{Hypothesis of the main results: approximate factorization of
	characters}
\label{sec:hypothesis-afc}

In this section we will present the hypothesis of the main results
of the current paper, \cref{theo:mainLLN,theo:mainCLT}.

\subsubsection{Cumulants of characters} 
\label{sec:cumulants-of-characters} 

We define a product of two odd partitions as their concatenation, followed by
arranging the entries in a weakly decreasing manner. In this way the set $\OP$
of odd partitions becomes a commutative monoid with the unit given by the empty
partition $\emptyset$. We consider the \emph{algebra of odd partitions}
$\C[\OP]$ which --- as a vector space --- is defined as the set of formal linear
combinations of odd partitions; the product corresponds to the above monoid
structure via distributivity of multiplication. Any function
$\chi\colon\OP\to\C$ on odd partitions extends by linearity to a linear map
$\chi\colon\C[\OP]\to\C$ on the odd partition algebra.

A convenient way to encode the information about a function $\chi\colon\OP\to\C$
with the property that $\chi(\emptyset)=1$ is to use \emph{cumulants}. More
specifically, for partitions $\pi^1,\dots,\pi^\ell\in\OP$ we define their
\emph{cumulant} (with respect to $\chi$)
\begin{multline}
\label{eq:cumulant}
\kumu_\ell^\chi(\pi^1,\dots,\pi^\ell):=
\left. \frac{\partial^\ell}{\partial t_1 \cdots \partial t_\ell} 
          \log \chi \left( e^{t_1 \pi^1+\cdots+t_\ell \pi^\ell} \right)
\right|_{t_1=\cdots=t_\ell=0}=\\
[t_1 \cdots t_\ell] \log \chi \left( e^{t_1 \pi^1+\cdots+t_\ell \pi^\ell} \right) 
\end{multline}
as a coefficient in the Taylor series of an analogue of the logarithm of the
(multidimensional) Laplace transform. The operations on the right-hand side
should be understood in the sense of formal power series with values in the odd
partitions algebra $\C[\OP]$. 
For example,
\begin{align*} 
\kumu_1(\pi^1) &= \chi(\pi^1), \\
\kumu_2(\pi^1,\pi^2) &= \chi(\pi^1\pi^2)- \chi(\pi^1)\ \chi(\pi^2).
\end{align*}
Informally speaking, the cumulants $\kumu_\ell^\chi$ (for $\ell\geq 2$) quantify the
extent to which $\chi\colon \OP\to\C$ fails to be a semigroup homomorphism (with
respect to the multiplication).

\medskip

We denote by $\OP_{\leq n}:=\bigcup_{0\leq m\leq n} \OP_m$ the set of odd
partitions of size smaller or equal than $n$. We will apply the above
construction of cumulants to the special case when $V$ is a (reducible) spin representation
of $\Spin{n}$ and $\chi=\chi^V$ is the character ratio
defined on $\OP_{\leq n}$ by extending the domain of \eqref{eq:character-ratio}
by padding the partition $\pi$ with additional ones:
\begin{equation} 
\label{eq:extension}
\chi^V(\pi):=  \tr_V \irrepSp(\pi,1^{n-|\pi|}) =
\frac{\Tr \irrepSp(\pi,1^{n-|\pi|})}{\dimm V} \qquad \text{for }\pi\in\OP_{\leq n}.
\end{equation}
Note that so defined $\chi^V$ is well-defined only on the set 
$\OP_{\leq n}$; in this way the cumulant
\eqref{eq:cumulant} is well-defined
as long as $|\pi^1|+\cdots+|\pi^\ell|\leq n$.

\subsubsection{Approximate factorization of characters}

For a partition $\pi=(\pi_1,\dots,\pi_\ell)$ with
$\pi_1,\dots,\pi_\ell\geq 1$ we define its \emph{length}
\[ \|\pi\|:=|\pi|-\ell = (\pi_1-1)+\cdots+(\pi_\ell-1) \]
as the difference of its size and its number of parts.

\begin{definition}
\label{def:afp-representations}
Assume that for each integer $n\geq 1$ 
we are given a spin representation
$\irrepSp^{(n)}\colon\Spin{n}\to \GL(V^{(n)})$.
We say that the sequence $( V^{(n)} )$ has
\emph{approximate factorization property}
if for each $l\geq 1$ and all $\pi^1,\dots,\pi^\ell\in\OP$ such that
each $\pi^i=(2k_i+1)$ is an odd partition which consists of exactly one part,
we have that
\begin{equation}
\label{eq:decay-cumualts-characters}  
\kumu_\ell^{V^{(n)} }(\pi^1,\dots,\pi^\ell) 
=O\left( n^{- \frac{\|\pi^1\|+\cdots+\|\pi^\ell\| + 2(\ell-1) }{2}} \right) \qquad \text{for } n\to\infty.
\end{equation}
\end{definition}

\begin{example}
	\label{example:Plancherel}
	We continue \cref{example:strict-plancherel}. The vector space $\SGA{n}$ is the
image of the left-regular representation $\C\Spin{n}$ under the projection
$\frac{1-z}{2}$. Since the character of the left-regular representation
vanishes on all group elements (except for the identity $1$), it follows that
the character of  $\SGA{n}$ vanishes on all group elements, except for $1$ and
$z$, and hence the corresponding character ratio is given by
\begin{align} 
\nonumber
\chi^{\SGA{n}}(\pi)  &= 
\begin{cases} 
1 & \text{if } \pi = (1^n), \\
0 & \text{otherwise},
\end{cases} \\
\intertext{for any $\pi\in \OP_n$. 
We extend the domain of the character ratio 
by \eqref{eq:extension} and obtain}
\label{eq:left-regular-character}
\chi^{\SGA{n}}(\pi) &= 
\begin{cases} 
1 & \text{if $\pi = (1^k)$ for some $0\leq k\leq n$}, \\
0 & \text{otherwise},
\end{cases}
\end{align}
for any $\pi\in \OP_{\leq n}$.

Since $\kumu_1^{\SGA{n}}(\pi)= \chi^{\SGA{n}}(\pi)$, we just calculated the
first cumulant as well.

\smallskip

Since the map \eqref{eq:left-regular-character} is a homomorphism (in the
somewhat restricted sense that $\chi^{\SGA{n}}(\pi^1 \pi^2)=
\chi^{\SGA{n}}(\pi^1) \chi^{\SGA{n}}(\pi^2)$ for all $\pi^1,\pi^2\in\OP$ such
that $|\pi^1|+|\pi^2|\leq n$), it follows immediately that all higher cumulants
$\kumu_\ell^{\SGA{n}}$ (for $\ell\geq 2$) vanish.

\smallskip

Now it is easy to check that the sequence of representations
$(\SGA{n})$ indeed has approximate factorization property.
\end{example}

\medskip

We will construct a whole class of examples later in
\cref{sec:example-SW,sec:proof-randomSYT}.

\subsection{Free cumulants}
\label{sec:free-cumulants}

It was noticed by Biane \cite{Biane1998,Biane2001} that for asymptotic problems
it is convenient to parametrize the set of Young diagrams by \emph{free
	cumulants}, quantities which originate in the random matrix theory and
Voiculescu's free probability \cite{Mingo2017}.  We review these quantities in
the following.

\medskip

For a continual Young diagram $\omega$ we consider a function
$\sigma_\omega\colon\R \to \R_+$ given by
\begin{equation}
\label{eq:sigma-measure}
 \sigma_\omega(z):= \frac{\omega(z)-|z|}{2}  
 \end{equation}
which can be viewed as the density of a measure on $\R$.

If $\omega=\omega_{r \lambda}$ is the profile of a rescaled Young diagram for a
Young diagram $\lambda$ and $r>0$ then the total weight of this measure 
\[ \int_\R \sigma_{r \lambda}(z) \dif z = r^2 |\lambda| \]
is equal to the area of the rescaled Young diagram $r \lambda$ (there are $|\lambda|$
boxes, each is a square of side $r$).

Similarly, if $\omega=\omega_{r \xi} \colon \R\to\R_+$ is the 
profile of a rescaled shifted Young diagram $\xi$ then
\[ \int_\R \sigma_{r \xi}(z) \dif z = 2 r^2 |\xi| \]
is the \emph{double} of the area of the rescaled shifted Young diagram $r \xi$
(the additional factor $2$ appears because we extended the domain of the profile
to the whole real line).

\bigskip

For a given
continual Young diagram $\omega$ and an integer $n\geq 2$ we define the rescaled moment of
the measure $\sigma_\omega$:
\begin{align}
\label{eq:s-as-moment}
S_n & = S_n(\omega)  = (n-1) \int_\R  z^{n-2} \sigma_\omega(z) \dif z.  \\
\intertext{Then the sequence of free cumulants $R_2,R_3,\dots$ is defined by}
\label{eq:r-in-s} 
R_n &= R_n(\omega)  = \sum_{l\geq 1} \frac{1}{l!} (-n+1)^{l-1}
\sum_{\substack{k_1,\dots,k_l\geq 2,\\ k_1+\cdots+k_l=n}} S_{k_1} \cdots
S_{k_l}. 
\\ 
\intertext{Conversely, the sequence of free cumulants determines
	uniquely the corresponding sequence of moments $S_2,S_3,\dots$ by the identity}
\label{eq:s-in-r}
S_n & = \sum_{l\geq 1} \frac{1}{l!} (n-1)^{\downarrow l-1}
\sum_{\substack{k_1,\dots,k_l\geq 2,\\ k_1+\cdots+k_l=n}} R_{k_1} \cdots
R_{k_l},
\end{align}
where $m^{\downarrow k}=m (m-1) \cdots (m-k+1)$ denotes the falling power, see
\cite[Section 3.2 and Proposition 2.2]{DolegaFeraySniady2008}.

\subsection{The first main result: random strict partitions concentrate around
	some limit shape}

In the following, in order to keep the notation lightweight, we will write
$\chi^{(n)}$, $\kumu_\ell^{(n)}$, $\mathbb{P}^{(n)}$ instead of $\chi^{V^{(n)}}$,
$\kumu_\ell^{V^{(n)} }$, $\mathbb{P}^{V^{(n)}}$, etc.

\begin{theorem}
	\label{theo:mainLLN} Assume that for each integer $n\geq 1$ we are given a spin
representation $\irrepSp^{(n)}\colon\Spin{n}\to \GL(V^{(n)})$ and assume that
the sequence $(V^{(n)})$ fulfils the approximate factorization property
(\cref{def:afp-representations}).

Additionally, we require that for all odd numbers $k,l\geq 3$ the following limits exist:
\begin{align}
\label{eq:freecumulants-afp} 
\free_{k+1} & := \lim_{n\to\infty} n^{\frac{k-1}{2}} \chi^{(n) }(k), \\
\label{eq:covariancedisjoint-afp}
\covarianceDisjoint_{k+1,l+1} & := 
\lim_{n\to\infty} n^{\frac{k+l}{2}} \cdot 2 \cdot \kumu_2^{(n) }\big( (k),(l) \big)
\\ 
& \nonumber = 
\lim_{n\to\infty} n^{\frac{k+l}{2}} \cdot 2 \cdot
\left(  \chi^{(n) }(k,l) - \chi^{(n) }(k)\ \chi^{(n) }(l)  \right)
\end{align}
(note that \cref{def:afp-representations} implies already that the expressions
under the limits are $O(1)$); and that and that the sequence $\free_2, \free_4,
\dots$ grows at most like a geometric sequence:
\[ \limsup_{k\to\infty} \sqrt[k]{\left| \free_k \right|} < \infty.\]

We denote by $\xi^{(n)}$ the random shifted Young diagram with the distribution
given by $\PP^{(n)}$. 

\medskip

	Then the sequence of rescaled shifted Young diagrams
	$\frac{1}{\sqrt{n}} \xi^{(n)}$ converges in probability towards some limit $\Omega$.
	In other words: there exists a unique continual Young diagram
$\Omega\colon\R\to [0,\infty)$ (\emph{``the limit shape''}) with the
property that for each $\epsilon>0$
	\[\lim_{n\to\infty} \PP^{(n)}\left( \xi^{(n)} \in \SP_n :
	 \left\| \omega_{\frac{1}{\sqrt{n}} \xi^{(n)} } - \Omega  \right\| > \epsilon \right) =0, \]
	 where $\| \cdot \|$ denotes the supremum norm.
	
	\smallskip
	
	This limit shape
	$\Omega \colon\R\to\R_+$ is uniquely determined by its free cumulants
	\begin{equation} 
	\label{eq:our-beloved-free-cumulants}
	R_k\left(\Omega\right) = 
	\begin{cases} \free_k
	& \text{if $k$ is even},
	\\
	0 & \text{if $k$ is odd}.
	\end{cases} 
	\end{equation}
\end{theorem}

The proof of this result is postponed to \cref{sec:proof-of-LLN}. This theorem
is analogous to a result of Biane \cite[Corollary 1]{Biane2001} who considered \emph{linear}
representations of the symmetric groups and the corresponding random
(non-shifted) Young diagrams. The assumptions of \cref{theo:mainLLN} can be
weakened to match the assumptions of the analogous result of Biane; for
simplicity we decided to have the same assumptions for \cref{theo:mainLLN} and
\cref{theo:mainCLT}.

In \cref{sec:plancherel} we will prove that in the special case considered in
\cref{example:Plancherel} when $V^{(n)}=\SGA{n}$ is the spin part of the
left-regular representation and the corresponding probability measure
$\mathbb{P}^{(n)}$ is the shifted Plancherel measure, the limit curve $\Omega$
coincides with the Logan--Shepp \& Vershik--Kerov curve
\cite{Logan1977,Vershik1977} which describes the limit shape of (non-shifted)
random Young diagram distributed to the (non-shifted) Plancherel measure. In
this case the proof is due to De Stavola \cite[Section 4.5]{DeStavolaThesis};
this result was conjectured earlier by the authors of \cite{Bernstein2007}.

\subsection{The second main result: Gaussian fluctuations}
\label{sec:gaussian-fluctuations}

\begin{theorem}
	\label{theo:mainCLT} 
We keep the notations and the assumptions from \cref{theo:mainLLN}.

	\begin{enumerate}[topsep=2ex,itemsep=2ex]
		\item \emph{(Gaussian fluctuations of characters.)} Then the joint
distribution of (any finite collection of) the centred random variables
\[ n^{\frac{k}{2}} \left( \chi^{\xi^{(n)}}(k) -\E \chi^{\xi^{(n)}}(k)\right), 
                                                \qquad  k\in\{3,5,7,9,\dots\}   \]
		converges in distribution to a Gaussian distribution, where
		\[ \chi^{\xi^{(n)}}(k) = \chi^{\xi^{(n)}}\big( (k)\big)= 
		   \chi^{\xi^{(n)}}\big( (k, \underbrace{1, \dots, 1}_{\text{$n-k$ times}})\big) \] 
		denotes the evaluation of the character ration on the odd partition $(k)$ which
consists of a single part.

		\item 
\emph{(Gaussian fluctuations of shapes.)}
Then the joint distribution of (any finite collection of) the  random variables
\[
\sqrt{n} \int_0^\infty x^{2k} \left(\omega_{\frac{1}{\sqrt{n}} D(\xi^{(n)}) }(x) 
    - \Omega(x)\right)  \mathrm{d}x, \qquad k\in\{1,2,\dots\}
\]
converges in distribution to a centered Gaussian distribution, where $\Omega$ is
the function provided by \cref{theo:mainLLN}.

	\end{enumerate}
\end{theorem}

The proof is postponed to \cref{sec:proof-theomainCLT}. The explicit form of the
covariance can be calculated thanks to \cref{theo:key-tool}. This result is
analogous to the central limit theorem for random (non-shifted) Young diagrams
proved by \sniady \cite{Sniady2006} which was an extension of Kerov's Central
Limit Theorem \cite{Kerov1993,Ivanov2002} to the non-Plancherel case. Note also
that the special case of the above result for the shifted Plancherel measure was
proved by Ivanov \cite{Ivanov2004} already in 2004.

\smallskip

The following sections are a preparation for the proofs of \cref{theo:mainLLN}
and \cref{theo:mainCLT}.

\section{The approximate factorization category} 
\label{sec:afc}

The notion of \emph{approximate factorization of characters} was introduced by
\sniady as a tool for proving Gaussianity of fluctuations of
random Young diagrams related to representation theory and special functions
\cite{Sniady2006,DolegaSniady2018}. We use this occasion to present this notion
in a more abstract and more transparent framework.

\subsection{Filtered algebras}

In the usual definition of a \emph{filtered algebra} $\A=\bigcup_{i\in\Z_{\geq
		0}} \F_i$ the family $(\F_i)$ is indexed by non-negative integers. For this
reason we will refer to such a filtered algebra as \emph{$\Z_{\geq 0}$-filtered
	algebra}. The following is a slight extension of this concept. Note that each
such a $\Z_{\geq 0}$-filtered algebra becomes a $\Z$-filtered algebra by setting
$\F_i:=\{0\}$ for all negative integers $i<0$.

\begin{definition}
By a \emph{$\Z$-filtered algebra} we will understand an algebra $\A$ together
with a family (indexed by integers) of linear subspaces $(\F_i)_{i\in\Z}$  which
is increasing: $\cdots\subseteq \F_{-1}\subseteq \F_0 \subseteq \F_1 \subseteq
\cdots \subseteq \A$, such that $\A=\bigcup_i \F_i$ and such that $\F_i \cdot \F_j
\subseteq \F_{i+j}$ holds true for all $i,j\in\Z$. We will always assume that
$\A$ has a unit and $1\in \F_0$.

For $x\in\A$ its \emph{degree} $\DEGREE_\A x$ is defined as the minimal value of
$i\in\Z$ such that $x\in \F_i$. 
\end{definition}

Often we do not need to distinguish between $\Z$- and $\Z_{\geq 0}$-filtered
algebras; in this case we will speak simply about \emph{filtered algebras}.

\subsection{Examples of filtered algebras} 

The following two examples will play an important role later on.

\subsubsection{The algebra $\Sequences$ of sequences with polynomial growth}

For $i\in\Z$ we define 
\[ \F_i:= \Big\{ (f_1,f_2,\dots): \sup_n \left| f_n\right| n^{-\frac{i}{2}} < \infty     \Big\}
\]
to be the linear space of (real valued) sequences with growth at most $O\left(
n^{\frac{i}{2}} \right)$. Then $\Sequences:=\bigcup_i \F_i$ is a unital
$\Z$-filtered commutative algebra with the multiplication given by the pointwise
product. The unit $1=(1,1,\dots)\in \F_0$ corresponds to the constant sequence.

\subsubsection{The algebra of odd partitions}

We revisit \cref{sec:cumulants-of-characters} and we equip the algebra
$\C[\OP]$ of odd partitions with a $\Z$-filtration by setting
\[ \HH_i = \lin \left\{ \pi \in\OP : - \| \pi \| \leq i \right\}\]
for any $i\in\Z$.

It should be stressed that this filtration has a peculiar property that the
degree of any element is a non-positive integer.

\subsection{The approximate factorization property}

\begin{definition}
	\label{def:cumulants}
	Suppose that $\A$ and $\B$ are unital commutative algebras and 
	$F\colon\A \to\B$ is a linear unital map. For
	$a_1,\dots,a_l\in\A$ we define their \emph{cumulant} 
	\begin{multline*}
	\kumu^F_\ell(a_1,\dots,a_\ell):=
	\left. \frac{\partial^\ell}{\partial t_1 \cdots \partial t_\ell} 
	     \log F \left( e^{t_1 a_1+\cdots+t_\ell a_\ell} \right)
	\right|_{t_1=\cdots=t_l=0}=\\
	[t_1 \cdots t_\ell] \log F \left( e^{t_1 a_1+\cdots+t_\ell a_\ell} \right) \in\B
	\end{multline*}
	where the operations on the right-hand side should be understood in the sense
of formal power series in the variables $t_1,\dots,t_\ell$,
cf.~\eqref{eq:cumulant}.
\end{definition}
So defined cumulant is a coefficient in the expansion of an multidimensional
Laplace transform, with the role of the expected value $\E$ played by the linear
map $F$. For example,
\begin{equation}
\label{eq:cumulants-in-terms-of-moments}
\left\{
\begin{aligned}
\kumu^F_1(a_1    ) &= F(a_1), \\
\kumu^F_2(a_1,a_2) &= F(a_1 a_2)- F(a_1) F(a_2),\\
\vdots
\end{aligned}
\right.
\end{equation}
correspond to the mean value and the covariance.	

\begin{definition}
	\label{def:afp-for-maps}
	We say that a linear unital map $F\colon\A\to\B$ between filtered commutative
algebras has \emph{the approximate factorization property} if for all $\ell\geq 1$ and
$a_1,\dots,a_\ell\in\A$
	\begin{equation} 
	\label{eq:afp}
	\DEGREE_\B \kumu_\ell^F (a_1,\dots,a_\ell) 
	\leq  \left(\DEGREE_\A a_1\right)+\cdots+\left( \DEGREE_\A a_\ell\right) - 2(\ell-1).
	\end{equation}
\end{definition}	

\begin{remark}
More generally,
for an arbitrary $\alpha>0$ one can consider the notion of
\emph{$\alpha$-approximate factorization property} in which \eqref{eq:afp} is
replaced by the condition
\[ 	\DEGREE_\B \kumu_\ell^F (a_1,\dots,a_\ell) 
\leq  \left(\DEGREE_\A a_1\right)+\cdots+\left( \DEGREE_\A a_\ell\right) - \alpha (\ell-1).
\] 	
Clearly, \cref{def:afp-for-maps} is a special case corresponding to $\alpha=2$.

Some of the abstract general results which we present in the following remain
true also for this more general notion (in particular, all results from
\cref{sec:af-category,sec:generatotrs-af}). On the other hand, the results
which concern specific examples (such as
\cref{obs:afp-revisited,prop:afp-gauss}) depend on the specific choice of $\alpha=2$.

In order to keep the notation lightweight we will not pursue further this more
general setup.
\end{remark}

\subsection{The approximate factorization category}
\label{sec:af-category}

\begin{lemma}
	\label{lem:category}

If $\A,\B,\mathcal{C}$ are filtered unital commutative
algebras and $F\colon \A\to\B$ and $G\colon\B\to\mathcal{C}$ have the approximate
factorization property then their composition $G\circ F\colon\A\to\mathcal{C}$
also has approximate factorization property.
\end{lemma}
This lemma appears in a rather concealed form in \cite[Section 4.7, proof of the
implication (13)$\implies$(14)]{Sniady2006}; the main idea of the proof is to
use the formula of Brillinger \cite{Brillinger} in order to express the
cumulants for the composition $G\circ F$ in terms of the cumulants for $G$ and
cumulants for $F$.

\medskip

\cref{lem:category} allows us to speak about
\emph{the approximate factorization category} which has filtered
unital commutative algebras as objects and 
unital maps with the approximate factorization property as morphisms.

\begin{lemma}
	
\label{lem:afp-inverse}

	Suppose that $F\colon\A\to\B$ has the approximate factorization property and additionally it preserves the degree, i.e.~$\DEGREE_\B F(a) = \DEGREE_\A a$ for each $a\in\A$,
	and that
	$F$ is invertible as a linear map.
	
	Then $F^{-1}\colon \B\to\A$ also has the approximate factorization property.
\end{lemma}
This lemma appears in a concealed form in 
\cite[Section 4.7, proof of the implication (14)$\implies$(13)]{Sniady2006}.

\subsection{Generators and approximate factorization}
\label{sec:generatotrs-af}

Since we consider the setup of \emph{filtered} algebras, the usual notion of
\emph{generators of an algebra} has to be adjusted accordingly.  

\begin{definition}
Let $\A$ be a filtered algebra and let $X\subseteq \A$ be its subset. We say
that $X$ generates $\A$ \emph{as a filtered algebra} if each $a\in\A$ is a
linear combination (with complex coefficients) of finite products of the form
$x_1\cdots x_\ell$ for some $\ell\geq 0$ and $x_1,\dots,x_\ell\in X$ such that 
\begin{equation}
\label{eq:generates-as-filtered-algebra}
\DEGREE x_1+\cdots+\DEGREE x_\ell\leq \DEGREE a.
\end{equation}
\end{definition}

The following simple result was proved by \sniady \cite[Corollary
19]{Sniady2006} in the specific setup of the Kerov--Olshanski algebra (with two
distinct multiplicative structures). The proof did not use any specific
properties of these filtered algebras and thus it remains valid also in this
general context. Note that the original paper mistakenly does not mention the
necessary condition \eqref{eq:generates-as-filtered-algebra}; see also
\cite[Lemma 3.4]{DolegaSniady2018} and the proceeding discussion.

\begin{lemma}
\label{lem:generators-afp-ok}
Let $F\colon\A\to\B$ be a linear unital	map between filtered commutative
algebras and let $X$ be a set which generates $\A$ as a filtered algebra.

If condition \eqref{eq:afp} holds true for all $\ell\geq 1$ and $a_1,\dots,a_\ell\in
X$ then it holds for arbitary $a_1,\dots,a_\ell\in\A$; in other words $F$ has
the approximate factorization property.
\end{lemma}

\subsection{Example: approximate factorization of characters revisited}
\label{sec:AFP-partitions-revisited}

In \cref{def:afp-representations} we defined \emph{the approximate
	factorization property} for a sequence of representations while above we used
the same name in \cref{def:afp-for-maps} in the context of maps between
filtered algebras. As we explain below, this is not a coincidence.

\smallskip

With the notations of \cref{def:afp-representations}, let $\big( V^{(n)} \big)$
be a fixed sequence of representations and let $\big(\chi^{(n)}\big)$ be the
corresponding sequence of the character ratios with $\chi^{(n)}\colon \OP_{\leq
	n} \to \R$. We extend the domain of $\chi^{(n)}$ in an arbitrary way so that
$\chi^{(n)}\colon \OP \to \R$; for example we may set $\chi^{(n)}(\pi)=0$ if
$|\pi|>n$. The information about this sequence of character ratios can be
encoded by a single map
\begin{equation}
\label{eq:chi-to-sequences}
\chi \colon \C[\OP] \to \Sequences
\end{equation} which is defined by
\[ \chi(\pi):= \big( \chi^{(1)}(\pi), \chi^{(2)}(\pi),\dots \big)\qquad \text{for } \pi\in \OP.\]

\begin{observation}
	\label{obs:afp-revisited}
	The map \eqref{eq:chi-to-sequences} has the
approximate factorization property (in the sense of \cref{def:afp-for-maps}) if
and only if the sequence of representations $\big( V^{(n)}\big)$ has the
approximate factorization property (in the sense of \cref{def:afp-representations}).
\end{observation}
\begin{proof}
We denote by $X\subset \OP$ the set of odd partitions which consist of exactly
one part. This set generates $\OP$ as a commutative monoid. We shall identify
any odd partition with the corresponding element of the odd partition algebra
$\C[\OP]$; it is easy to check that $X\subset \C[\OP]$ generates the odd
partition algebra $\C[\OP]$ as a filtered algebra.

\smallskip

Assume that $\big( V^{(n)}\big)$ has the approximate factorization property.
The condition \eqref{eq:decay-cumualts-characters} implies that the assumptions
of \cref{lem:generators-afp-ok} are fulfilled for the map
\eqref{eq:chi-to-sequences} and the aforementioned generating set $X$. It
follows that \eqref{eq:chi-to-sequences} indeed has the approximate
factorization property, as required.

\smallskip

The opposite implication is immediate.
\end{proof}

\subsection{The motivating example: Gaussian fluctuations} 

The following example will be our key tool for proving Gaussianity of various
random variables.

\medskip

For each $n\geq 1$ let $(\Omega_n,\mathfrak{F}_n,\mathbb{P}_n)$ be a
probability space. Assume that $\A$ is a $\Z$-filtered   commutative algebra
such that each element $\X\in\A$ is a sequence $\X=(X_1,X_2,\dots)$ where
$X_n\colon\Omega_n\to\R$ is a random variable on the appropriate probability space
with the property that its expected value $\E_n X_n$ is well-defined.

\medskip

We assume that the unital map $\E\colon \A\to \Sequences$ given by
the expected value:
\[ \E\colon (X_1,X_2,\dots) \mapsto (\E_1 X_1, \E_2 X_2,\dots) \]
is well-defined, i.e.~it indeed takes values in $\Sequences$.

\begin{proposition}
	\label{prop:afp-gauss}
	With the above notations, let us assume that $\E$ has the  approximate factorization
	property. Let $\{\X_1,\dots,\X_l\} \subset \A$ be a finite set; we denote
	$\X_i=(X_{i,1},X_{i,2},\dots)$ and define a centered random variable
	\[
	y_{i,n} := n^{\frac{1-\DEGREE \X_i}{2}} \cdot
	(X_{i,n} - \E_n X_{i,n}). 
	\]
	Assume that the limit of the covariance
	\[\lim_{n\to\infty} \Cov\left(y_{i_1,n},y_{i_2,n}\right)=
	\lim_{n\to\infty} \E_n \left( y_{i_1,n}\ y_{i_2,n} \right)
	\]
	exists for any $1\leq i_1,i_2\leq l$. 
	
	\medskip
	
	Then the joint distribution of the tuple 
	\begin{equation} 
	\label{eq:my-collection}
	\big( y_{1,n}, \dots, y_{l,n} \big) 
	\end{equation}
	of centered random variables converges to a Gaussian distribution in the limit
when $n\to\infty$, in the weak topology of probability measures.
\end{proposition}	
\begin{proof}
	We consider the cumulant of the random variables 
	$y_{i_1,n},\dots,y_{i_\ell,n}$
	\begin{equation}
	\label{eq:cumulant-concrete}
	\kumu_\ell (y_{i_1,n},\dots,y_{i_\ell,n}):=
	\left. \frac{\partial^\ell}{\partial t_1 \cdots \partial t_\ell} 
	       \log \E_n \left( e^{t_1 y_{i_1,n}+\cdots+t_\ell y_{i_\ell,n}} \right)
	\right|_{t_1=\cdots=t_\ell=0}.
	\end{equation}
	
	For $\ell=1$ this cumulant 
	\[ \kumu_1 (y_{i_1,n})= \E_n y_{i_1,n} = 0
	\]
	trivially vanishes by the centeredness.
	
	For $\ell\geq 2$ the cumulant is shift-invariant, thus
	the assumption of the approximate factorization property implies that
	\[ \kumu_\ell (y_{i_1,n},\dots,y_{i_\ell,n}) = 
	n^{\frac{\ell-\DEGREE \X_{i_1}-\cdots-\DEGREE \X_{i_\ell}}{2}} \kumu_\ell(X_{i_1,n},\dots,X_{i_\ell,n})=
	O\left( n^{\frac{2-\ell}{2}}\right).
	\]
	In particular, for $\ell\geq 3$ this cumulant clearly converges to zero.
	
	For $\ell=2$ the cumulant coincides with the covariance
	\[\kumu_2 (y_{i_1,n},y_{i_2,n})
	=\Cov\left(y_{i_1,n},y_{i_2,n}\right) \]
	and hence by it converges by assumption.
	
	\medskip
	
To summarize: we proved that each of the cumulants \eqref{eq:cumulant-concrete}
converges to a finite limit. Since each mixed moment of a collection of random
variables can be expressed as a polynomial in their cumulants (for example, via
the moment-cumulant formula \cite{Speed1983}), it follows that the tuple
\eqref{eq:my-collection} converges \emph{in moments} to some distribution. The
multidimensional Gaussian distribution can be characterized by the property
that its cumulants vanish (except for the mean value and the covariance) thus
this limit is a Gaussian distribution. The Gaussian distribution is uniquely
determined by its moments; it follows that convergence \emph{in moments}
implies \emph{weak convergence}, as required.
\end{proof}

\section{Kerov--Olshanski algebra and its spin analogue} 

\label{sec:ko}

The usual (linear) Kerov--Olshanski algebra $\KO$
\cite{Kerov1994,Hora2016} (also known under the less compact name \emph{algebra
	of polynomial functions on the set of Young diagrams};
in the monograph
\cite[Section 7]{Meliot2017} it is referred to as \emph{Ivavov--Kerov algebra}) is an important tool in
the (linear) asymptotic representation theory of the symmetric groups. One of
its advantages comes from the fact that it can be characterized in several
equivalent ways (for example as the algebra $\Lambda^*$ of \emph{shifted
	symmetric functions}); it also has several convenient linear and algebraic
bases which are related to various viewpoints and aspects of the asymptotic
representation theory. In particular, it was an important ingredient in the
proof of Gaussianity of fluctuations for random (non-shifted) Young diagrams
\cite{Sniady2006}. 

For the purposes of the current paper we will need the spin
analogue $\Gamma$ of Kerov--Olshanski algebra \cite{Ivanov2004}.

\medskip

In the current section we shall present these two algebras, as well as their
modifications related to a different multiplicative structure (\emph{``disjoint
	product''}) and the links between them. The main result of the current section
is \cref{theo:afp-for-spin}. Our strategy of proof is to use the link between
the linear and the spin setup which we explored in \cite{Matsumoto2018c}. This
section is purely algebraic: all calculations are exact, there are no asymptotic
assumptions, there is no randomness, there are no representations and no random
Young diagrams.

\subsection{The linear setup}

\subsubsection{Normalized characters of the symmetric groups}

The usual way of viewing the characters of the symmetric groups is to fix the
irreducible representation $\lambda$ and to consider the character as a function
of the conjugacy class $\pi$. However, there is also another very successful
viewpoint due to Kerov and Olshanski \cite{Kerov1994}, called \emph{dual
	approach}, which suggests to do roughly the opposite. It turns out that the most
convenient way to pursue in this direction is to define --- for a fixed integer
partition $\pi$ --- \emph{the normalized character on the conjugacy class $\pi$}
as a function on the set of \emph{all} Young diagrams $\Ch_\pi\colon\Young
\to\C$ given by
\[ \Ch_\pi(\lambda):=\begin{cases}
n^{\downarrow k} 
\ \tr \irrepSn^\lambda(\pi\cup 1^{n-k}) & 
\text{if } n\geq k, \\
0 & \text{otherwise,}
\end{cases}\]
where $n=|\lambda|$ and $k=|\pi|$. Above, $\irrepSn^\lambda(\pi\cup 1^{n-k})$
denotes the irreducible representation of the symmetric group $\Sym{n}$
evaluated on any permutation from $\Sym{n}$ with the cycle decomposition given
by the partition $\pi\cup 1^{n-k}$; furthermore $\tr$ is the normalized trace
\eqref{eq:normalized-trace}; and $n^{\downarrow k}=n (n-1) \cdots (n-k+1)$
denotes the falling power.

\subsubsection{Kerov--Olshanski algebra} 

For the purposes of the current paper, \linebreak Kerov--Olshanski algebra 
\[ \KO := \lin\{ \Ch_\class : \class \in \Part \} \]
may be defined as the linear span of the normalized linear characters of the
symmetric groups. 
We equip if with a filtration $\F_0\subseteq \F_1 \subseteq \cdots \subseteq
\KO$ given by
\begin{equation} 
\label{eq:filtration-KO}
\F_k := \lin\{ \Ch_\class : \class \in \Part,\quad  \vertiii{\class}\leq k \},
\end{equation}
where 
\[ \vertiii{\class}=|\class|+\ell(\class).\]
This specific choice of the filtration is motivated by investigation of
asymptotics of (random) Young diagrams and tableaux in the scaling in which they
grow to infinity in such a way that they remain \emph{balanced}
\cite{Biane1998,Sniady2006}.

\subsubsection{Disjoint product and the algebra $\KO_\bullet$}
\label{sec:disjoint-product}

The characters $(\Ch_\pi : \pi \in \Part)$ turn out to form a linear basis of
$\KO$ 
which allows us to define a new multiplication on
$\KO$, the \emph{disjoint product}, by setting
\[ \Ch_{\pi^1} \bullet \Ch_{\pi^2}:= \Ch_{\pi^1 \pi^2},\]
where the product of two partitions $\pi^1\pi^2$ on the right-hand side should
be understood --- just like in \cref{sec:cumulants-of-characters} --- as their
concatenation.

The vector space $\KO$ equipped with the disjoint product $\bullet$ becomes a
unital, commutative, $\Z_{\geq 0}$-filtered (with respect to the usual
filtration \eqref{eq:filtration-KO}) algebra which will be denoted by
$\KO_\bullet$.

\medskip

To summarize: the vector space $\KO$ can be equipped with two distinct
multliplicative structures which correspond to the pointwise product and to the
disjoint product. Comparison of these two multiplicative structures turns out to
be a fruitful idea in asymptotic representation theory \cite{Sniady2006}. We
recall the key result of this flavour below in \cref{theo:afp-for-linear}.

\subsubsection{Approximate factorization property for linear characters}

The following result turned out to be essential for proving Gaussianity of
fluctuations for random Young diagrams in \cite{Sniady2006}.
Our goal in this section is to prove its spin analogue (\cref{theo:afp-for-spin}).

\begin{proposition}
	\label{theo:afp-for-linear}
	The identity map
	\begin{equation} 
	\label{eq:identityKO}
		\id_\KO \colon \KO_\bullet \longrightarrow \KO 
	\end{equation}
	has the approximate factorization property.
	
	Furthermore, the second cumulant of this map is given by
	\begin{multline} 
	\label{eq:cumulant-disjoint}
	\kumu^{\id_\KO}_2(\Ch_{k_1},\Ch_{k_2} )= \Ch_{k_1,k_2} - \Ch_{k_1} \Ch_{k_2}  = \\
	   (-1) \sum_{r\geq 1}
	   \sum_{\substack{a_1,\dots,a_r\geq 1 \\ a_1+\cdots+a_r=k_1}}
	   \sum_{\substack{b_1,\dots,b_r\geq 1 \\ b_1+\cdots+b_r=k_2}}
	   \frac{k_1 k_2}{r}\ \Ch_{(a_1+b_1-1), \dots, (a_r+b_r-1)}  +\\
	   + \text{terms of degree at most ($k_1+k_2-2$)}.
	 \end{multline}
\end{proposition}

The first proof of this result was found by \sniady
\cite[Theorem 15]{Sniady2006}. For a sketch of an alternative proof based on
Stanley character formula and some more historical context we refer to
\cite[Section 1.13]{SniadyJackUnpublished}. 
An extension of this result to the context of Jack symmetric functions and Jack
characters was proved by a yet another method in \cite[Theorem 2.3]{Sniady2019}.

\subsection{The spin setup}

\subsubsection{Normalized spin characters}

Following Ivanov \cite{Ivanov2004,Ivanov2006}
(see also \cite{Matsumoto2018c}),
for a fixed odd partition
$\class\in\OP$ the corresponding \emph{normalized spin character} is a function
on the set of \emph{all} strict partitions given by
\begin{equation}
\label{eq:projective-normalized-trololo}
\ChSpin_\class(\xi):=\begin{cases}
n^{\downarrow k}\ 2^{\frac{\|\pi\|}{2}}
\ \frac{ \PHIeasy{\xi}{\class\cup 1^{n-k}}}{ \PHIeasy{\xi}{1^{n}}}
& \text{if } n\geq k, \\
0 & \text{otherwise,}
\end{cases}
\end{equation}
where $n=|\xi|$ and $k=|\class|$. 

\subsubsection{Spin Kerov--Olshanski algebra}

We define the \emph{spin} Kerov--Olshanski algebra 
(maybe \emph{Ivanov algebra} would be an even better name)
\begin{equation}
\label{eq:Gamma}
\Gamma := \lin\{ \ChSpin_\class : \class \in \SP \} 
\end{equation}
as the linear span of spin characters \cite[Section 6]{Ivanov2004}.
Ivanov proved  
that the elements of $\Gamma$ can be identified with 
\emph{supersymmetric polynomials}, thus $\Gamma$ is a unital, commutative algebra. 
Following \cite[Section 1.6]{Matsumoto2018c}, we equip $\Gamma$ with a filtration 
$\G_0\subseteq \G_1 \subseteq \cdots \subseteq \Gamma$ defined by
\begin{equation} 
\label{eq:filtration-spin}
\G_k := \lin\{ \ChSpin_\class : \class \in \OP,\quad  \vertiii{\class}\leq k \}.
\end{equation}

\subsubsection{Disjoint product and the algebra $\Gamma_\bullet$}

Similarly as in \cref{sec:disjoint-product} we define the disjoint product
of spin characters
\[ \ChSpin_{\pi^1} \bullet \ChSpin_{\pi^2}:= \ChSpin_{\pi^1 \pi^2}\]
for arbitrary $\pi_1,\pi_2\in\SP$. We denote by $\Gamma_{\bullet}$
the filtered algebra which, as a vector space, coincides with $\Gamma$ and 
is equipped with a multiplication given by the disjoint product $\bullet$.

\subsubsection{Approximate factorization property for spin characters}

\begin{theorem}
	\label{theo:afp-for-spin}
	The identity map
	\[ \id_\Gamma \colon \Gamma_\bullet \longrightarrow \Gamma \]
	has the approximate factorization property.

	Furthermore, the second cumulant of this map is given --- for any odd integers
$k_1,k_2\geq 1$ --- by
		\begin{multline}
		\label{eq:second-cumulant} 
		\kumu^{\id_\Gamma}_2(\ChSpin_{k_1},\ChSpin_{k_2} )=
		 \ChSpin_{k_1,k_2}- \ChSpin_{k_1} \ChSpin_{k_2}  = \\
	(-1) \sum_{r\geq 1} \frac{2^{r-1}}{r}  k_1 k_2 
	\sum_{
       (a_i), (b_i)
        }
	\ChSpin_{(a_1+b_1-1), \dots, (a_r+b_r-1)}  +\\
	+ \text{terms of degree at most $(k_1+k_2-2)$},
	\end{multline}
	where the second sum runs over integers $a_1,\dots,a_r,b_1,\dots,b_r\geq 1$
such that $a_1+\cdots+a_r=k_1$, and $b_1+\cdots+b_r=k_2$ and for each $i\in[r]$
the sum $a_i+b_i$ is even.
\end{theorem}

The proof is postponed to \cref{sec:proof-of-afp-for-spin}.
Our strategy is to
explore the link between the linear and the spin setup.

\subsection{Double of a function. Kerov--Olshanski algebra: linear vs spin} 

\subsubsection{Double of a strict partition} 

\label{sec:double}

We denote by $\Part_n$ the set of partitions of a given integer $n\geq 0$. The
theory of partitions and Young diagrams is more developed than its shifted
counterpart. For this reason it is convenient to encode a given strict partition
$\xi\in\SP_n$ by its \emph{double} $D(\xi)\in \Part_{2n}$. Graphically, $D(\xi)$
corresponds to a Young diagram obtained by arranging the shifted Young diagram
$\xi$ and its `transpose' so that they nicely fit along the `diagonal',
cf.~\cref{fig:double}, see also \cite[page 9]{Macdonald1995}.

\subsubsection{Double of a function}

If $F\colon \Part\to \C$ is a function on the set of partitions, we define its
\emph{double} as the function $\double F \colon \SP \to \C$ on the set of
\emph{strict} partitions
\[ \left( \double F \right)(\xi) := F\big( D(\xi) \big) \qquad \text{for }\xi\in\SP\]
given by doubling of the argument.

\begin{proposition}[{\cite[Theorem 1.8, Theorem 1.9]{Matsumoto2018c}}]
	\label{prop:double-preserves-degree}
	The map $\double$ is an algebra homomorphism
which maps Kerov--Olshanski algebra to its spin counterpart:
	\begin{equation}
		\label{eq:double-maps}
		\double (\KO) = \Gamma 
	\end{equation}	
	and, additionally, preserves the filtration, i.e.
	\begin{equation}
		\label{eq:double-filtration}
		\G_k = \double(\F_k) \qquad \text{for any integer $k\geq 0$.}
	\end{equation}
\end{proposition}

In the following we well need to compute the images $\double \Ch_{\rho}$ of the
linear basis of $\KO$. The following two results provide the necessary
information.

\begin{proposition}[{\cite[Theorem 3.1]{Matsumoto2018c}}]
	\label{thm:general_characters} 
	In the case when $\rho\in\OP$ is an odd partition,
	\begin{equation}
	\label{eq:spin-vs-linear}
	\double \Ch_\rho = 
	\sum_{I \subseteq \{1,2,\dots,\ell(\rho)\}} 
	\ChSpin_{\rho(I)}\  \ChSpin_{\rho(I^c)} \in \Gamma
	\end{equation}
	where $\rho(I)=(\rho_{i_1}, \rho_{i_2},\dots, \rho_{i_r})$ for $I=\{i_1<i_2<
\cdots< i_r\}$ and $I^c=\{1,\dots,\ell(\rho)\}\setminus I$ denotes the
complement of $I$.
\end{proposition}

\begin{proposition}
	\label{prop:not-odd-partitions}
	In the case when $\rho\notin\OP$ is \emph{not} an odd partition,
	$\double \Ch_\rho\in\Gamma$ is of degree at most \begin{equation}
	\label{eq:double-degree-bound-even}
	\begin{cases}
		\vertiii{\rho}-1 & \text{if $\rho$ contains \emph{exactly one} part which is even}, \\
		\vertiii{\rho}-2 & \text{if $\rho$ contains at \emph{least two} parts which are even}.
	\end{cases}
	\end{equation}
\end{proposition}
\begin{proof}

\medskip

We start with the special case when $\rho\notin\OP$ contains exactly one part
which is even. In this case $\vertiii{\rho}$ is an odd integer.

Clearly $\Ch_{\rho}\in\F_{\vertiii{\rho}}$; Eq.~\eqref{eq:double-filtration}
implies therefore that $\double\Ch_{\rho}\in\G_{\vertiii{\rho}}$. We revisit
the definition \eqref{eq:filtration-spin} of the filtration $\G$. Note that
$\vertiii{\class}$ is always an \emph{even} integer for any $\class \in
\OP$; it follows therefore that $\G_{2k+1}=\G_{2k}$ for any integer $k\geq 0$.
In particular, $\G_{\vertiii{\rho}}=\G_{\vertiii{\rho}-1}$. In this way we
proved that $\double \Ch_{\rho}\in \G_{\vertiii{\rho}-1}$, as required.

\medskip

Let $\rho=(\rho_1,\dots,\rho_\ell)\notin\OP$ be now a general partition. We
consider the identity map $\id_\KO \colon \KO_{\bullet}\to\KO$ and the
corresponding cumulants. The system of equations (analogous to
\eqref{eq:cumulants-in-terms-of-moments}) which express the cumulants in terms
of the moments can be inverted. The resulting moment-cumulant formula
\cite{Speed1983} expresses any given moment
\begin{equation}
	\label{eq:moment-ch} 
	\Ch_\rho= \id_\KO \left( \Ch_{\rho_1} \bullet \cdots \bullet \Ch_{\rho_\ell} \right) 
\end{equation}
as a polynomial in terms of the cumulants of the individual factors
$\Ch_{\rho_1},\dots,\linebreak \Ch_{\rho_\ell} $. For example,
\begin{align*}
  \Ch_{\rho_1} = \id_\KO \left( \Ch_{\rho_1}  \right) &= 
  \kumu_1^{\id_\KO} (\Ch_{\rho_1} ) ,\\
  \Ch_{\rho_1,\rho_2} = \id_\KO \left( \Ch_{\rho_1} \bullet \Ch_{\rho_2} \right) &= 
  \kumu_2^{\id_\KO} (\Ch_{\rho_1}, \Ch_{\rho_2} )+ 
          \kumu_1^{\id_\KO} (\Ch_{\rho_1})  \kumu_1^{\id_\KO} (\Ch_{\rho_2}).  
\end{align*}
In the general case, the summands in such an expansion of \eqref{eq:moment-ch}
can be split into the following two classes:
\begin{enumerate}[label=\emph{(\alph*)}]
	\item 
	\label{item:unique}
	the unique summand 
	\begin{equation} 
	\label{eq:the-unique-summand} 
	\kumu_1^{\id_\KO}(\Ch_{\rho_1}) \cdots \kumu_1^{\id_\KO}(\Ch_{\rho_{\ell}})=
	\Ch_{\rho_1} \cdots \Ch_{\rho_{\ell}}; 
	\end{equation}

	\item \label{item:remaining} the remaining summands; each such a summand
contains at least one factor with a cumulant $\kumu_k^{\id_\KO}$ for some $k\geq 2$.
\end{enumerate}
We apply the map $\double$ to \eqref{eq:moment-ch} or, equivalently, to its
aforementioned expansion to the products of cumulants and we investigate the
resulting terms. 

In the case \ref{item:unique}, by applying $\double$ to
\eqref{eq:the-unique-summand} we get
\begin{equation} 
\label{eq:the-unique-summand-d} 
(\double\Ch_{\rho_1}) \cdots (\double\Ch_{\rho_{\ell}}). 
\end{equation}
The discussion from the very beginning of this proof (the special case when
$\rho$ has exactly one even part) it follows that for each of the factors we
have that $\double\Ch_{\rho_{i}}\in \G_{\rho_i+1}$ if $\rho_{i}$ is odd and
$\double\Ch_{\rho_{i}}\in\G_{\rho_i}$ if $\rho_{i}$ is even. Thus the degree of
\eqref{eq:the-unique-summand-d} is bounded from above by
\[ \vertiii{\rho}- \text{(number of parts of $\rho$ which are even)}. \]

In the case \ref{item:remaining}, by approximate factorization property each
cumulant $\kumu_k^{\id_\KO}$ causes a decrease of the degree by $2(k-1)$. It
follows that the image of the considered summand under the map $\double$ is of
degree at most $\vertiii{\rho}-2$.

This completes the proof.
\end{proof}

\begin{remark}
It would be very interesting to have some explicit closed formula (maybe in the
flavour of Eq.~\eqref{eq:spin-vs-linear}) for $\double \Ch_\rho$ in the case
when $\rho\notin\OP$. Such a formula would make the link between
Kerov--Olshanski algebra $\KO$ and its spin counterpart $\Gamma$ even more
explicit.

We conjecture that the degree bound \eqref{eq:double-degree-bound-even} is not
optimal and that $\double\Ch_{\rho}$ is of degree at most
\[ \vertiii{\rho}-\text{(number of parts in $\rho$ which are even)}\]
for an arbitrary partition $\rho$. 
\end{remark}	

\subsection{Abstract viewpoint on \cref{thm:general_characters}}

We define the vector space
\[ \KOodd:= 
\lin\{ \Ch_\pi : \pi \in \OP \} \subset \KO \]
which is spanned by the characters corresponding to \emph{odd} partitions. This
vector space, equipped with the disjoint product $\bullet$, is a unital,
commutative algebra.

\medskip

We consider the algebra homomorphism $\Psi\colon \KOodd \to \Gamma_{\bullet}
\otimes \Gamma_{\bullet} $ which is defined on the algebraic basis of $\KOodd$
by an analogue of the Leibniz rule
\[ \Psi(\Ch_k)= 
   \ChSpin_k \otimes 1 + 1 \otimes \ChSpin_k \qquad \text{for any odd integer $k\geq 1$}.
\]
It follows that for any $\rho\in\OP$
\[ \Psi(\Ch_\rho)= 	\sum_{I \subseteq \{1,2,\dots,\ell(\rho)\}} 
\ChSpin_{\rho(I)} \otimes \ChSpin_{\rho(I^c)}.\]

The pointwise product gives rise to a bilinear function $m\colon \Gamma \times
\Gamma\to\Gamma$ given by $m(F,G):=F G$ which can be lifted to the unique linear
map $m\colon \Gamma \otimes \Gamma \to\Gamma$ on the tensor product. Because of
the isomorphism of the vector spaces $\Gamma \cong \Gamma_\bullet$ we will view
$m$ as a map $m\colon \Gamma_\bullet \otimes \Gamma_\bullet \to\Gamma$.

\medskip

With these notations, the following result is a straightforward reformulation of
\cref{thm:general_characters}.
\begin{corollary}
	The following diagram commutes
\tikzexternaldisable
\begin{equation} 
\label{eq:commutative-diagram}
\begin{tikzcd}
\KOodd \arrow[d, "\Psi"] \arrow[r,"\id_\KO"]
& \KO \arrow[d,"\double"] \\
\Gamma_{\bullet} \otimes \Gamma_{\bullet} \arrow[r, "m"]
& \Gamma 
\end{tikzcd},
\end{equation}
\tikzexternalenable
where the upper horizontal arrow is an inclusion of vector spaces.
\end{corollary}

\subsection{Cumulants of characters}

The following result provides a direct link between the cumulants for
$\id_\KO\colon \KO_\bullet\to \KO$ and $\id_\Gamma\colon \Gamma_{\bullet}\to
\Gamma$.

\begin{theorem}
	For any odd integers $k_1,\dots,k_\ell\geq 1$
	\begin{equation}  
	\label{eq:link-between-cumulants}
	\kumu^{\id_\Gamma}_\ell\left( \ChSpin_{k_1} , \dots,\ChSpin_{k_\ell} 
	\right) = \frac{1}{2} \double\left(  \kumu^{\id_\KO}_{\ell}(\Ch_{k_1},\dots,\Ch_{k_\ell} )  \right), 
	\end{equation}
	where the cumulant on the left hand side concerns the identity map 
	\[ \id_\Gamma \colon \Gamma_\bullet \to \Gamma\] 
	while the cumulant on the right-hand side concerns the identity map 
	\[\id_\KO \colon \KO_\bullet\to \KO.\]
\end{theorem}
\begin{proof}
Since both vertical arrows in the commutative diagram
\eqref{eq:commutative-diagram} are unital homomorphisms of algebras, it follows
immediately that the cumulants for the two horizontal arrows are related by the
identity
\[ 
\double\left(  \kumu^{\id_\KO}_{\ell}(x_1,\dots,x_\ell )  \right) =
\kumu^m_\ell\left( \Psi(x_1),\dots,\Psi(x_\ell) \right) 
\]
for any $x_1,\dots,x_\ell\in \KOodd$. 

We will consider the special case of this
equality when each $x_i=\Ch_{k_i}$ is a character corresponding to a partition
$(k_i)\in\OP$ which consists of a single part which is odd.
Then
\begin{multline*}
\double\left(  \kumu^{\id_\KO}_{\ell}(\Ch_{k_1},\dots,\Ch_{k_\ell} )  \right) = \\
\kumu^m_\ell\left( \ChSpin_{k_1} \otimes 1+ 1\otimes \ChSpin_{k_1},
     \dots,\ChSpin_{k_\ell} \otimes 1+ 1\otimes \ChSpin_{k_\ell} \right).
\end{multline*}

Since the cumulant is linear with respect to each of its arguments, the
right-hand side can be expanded to $2^\ell$ summands. Each summand is the
cumulant $\kumu^m_\ell$ applied to an $\ell$-tuple of elements, with each
element either from the subalgebra $\Gamma_{\bullet}\otimes 1$ or from the
subalgebra $1\otimes \Gamma_{\bullet}$. Since these two subalgebras are
classically independent with respect to the expected value $m$, any mixed
cumulant vanishes. It follows that out of these $2^\ell$ summands there are only
two which are non-zero and
\begin{multline*}
\kumu^m_\ell\left( \ChSpin_{k_1} \otimes 1+ 1\otimes \ChSpin_{k_1},
\dots,\ChSpin_{k_\ell} \otimes 1+ 1\otimes \ChSpin_{k_\ell} \right) = \\
\kumu^m_\ell\left( \ChSpin_{k_1} \otimes 1, \dots,\ChSpin_{k_\ell} \otimes 1
\right)+ \kumu^m_\ell\left( 1\otimes \ChSpin_{k_1}, \dots,1\otimes
\ChSpin_{k_\ell} \right) = \\
2 \kumu^{\id_\Gamma}_\ell\left( \ChSpin_{k_1} , \dots,\ChSpin_{k_\ell} 
\right).
\end{multline*}
\end{proof}

\subsection{Proof of \cref{theo:afp-for-spin}}
\label{sec:proof-of-afp-for-spin}

\begin{proof}[Proof of \cref{theo:afp-for-spin}]
By \cref{lem:generators-afp-ok}, in order to prove the first part 
it is enough to show that the cumulant on the left-hand side of
\eqref{eq:link-between-cumulants} is an element of $\Gamma$ of degree at most
\[ (k_1+1)+\cdots+(k_\ell+1)- 2(\ell-1). \]
However, the approximate factorization property for the map $\id_\KO$
(\cref{theo:afp-for-linear}) combined with \cref{prop:double-preserves-degree}
show this degree bound for the right-hand side of
\eqref{eq:link-between-cumulants}, as required.

\medskip

For the second part we apply \eqref{eq:link-between-cumulants} in the special
case $\ell=2$. The second cumulant $\kumu_2^{\id_\KO}$ which contributes to the
right-hand side is explicitly given by \eqref{eq:cumulant-disjoint}. It remains
now to evaluate $\double \Ch_\rho$ for
\[ \rho=\big( (a_1+b_1-1), \dots, (a_r+b_r-1) \big). \]

A simple parity argument shows that there is an even number of even parts of
such a partition $\rho$.  In particular, if $\rho\notin\OP$ then the number of its even
parts is at least $2$. \cref{prop:not-odd-partitions} is then applicable and
shows that in this case $\double \Ch_{\rho}$ is of degree at most $k_1+k_2-2$.

In the opposite case when $\rho\in \OP$, the exact value of $\double \Ch_{\rho}$
is given by \cref{thm:general_characters}. This exact form can be simplified
thanks to the observation that by approximate factorization property for $\id_\Gamma$
\begin{multline*}
\kumu_2^{\id_\Gamma}\left( \ChSpin_{\rho(I)} , \ChSpin_{\rho(I^c)} \right) 
= \ChSpin_{\rho(I)} \bullet \ChSpin_{\rho(I^c)} -
\ChSpin_{\rho(I)}  \ChSpin_{\rho(I^c)} = \\
\ChSpin_{\rho} -\ChSpin_{\rho(I)}  \ChSpin_{\rho(I^c)} \in \G_{\vertiii{\rho}-2};
\end{multline*}
it follows that
\[ \double \Ch_\rho = 2^{\ell(\rho)} \ChSpin_{\rho} + 
\text{(summands of degree at most $\vertiii{\rho}-2$)}.
\]
which completes the proof.
\end{proof}

\subsection{Free cumulants revisited}

We revisit \cref{sec:free-cumulants}.
For an integer $n\geq 2$ and $\xi\in \SP$ we define
\begin{align*} 
S_n^{\spin}(\xi) &= S_n(\omega_\xi) = 
 (n-1) \int_\R  z^{n-2} \sigma_{\omega_\xi} (z) \dif z, \\
R_n^{\spin}(\xi) &= R_n(\omega_\xi).  
\end{align*}
By the symmetry of the profile $\omega_{\xi}\colon\R\to\R$ it follows that
$S^{\spin}_n=R^{\spin}_n=0$ if $n$ is an odd integer. In the following we view
$S^{\spin}_n,R^{\spin}_n \colon \SP\to\R$ as \emph{functions} on the set of
strict partitions.

Note that the free cumulants for strict partitions defined above and the ones
considered by Matsumoto \cite{Matsumoto2018} differ by a factor of $2$.

\begin{proposition}
\label{prop:SR-generate}
For each even integer $n\geq 2$ we have that $S^{\spin}_n,R^{\spin}_n\in\Gamma$
with $\DEGREE S^{\spin}_n=\DEGREE R^{\spin}_n=n$.
Furthermore, $(S^{\spin}_2,S^{\spin}_4,\dots)$ as well as 
$(R^{\spin}_2,R^{\spin}_4,\dots)$ generate $\Gamma$ as a filtered algebra.
\end{proposition}
\begin{proof}

Comparison of \cref{fig:RussianDario} and \cref{fig:Russian} shows that the
measures $\sigma_{\omega_{\xi}}$ and $\sigma_{\omega_{D(\xi)} }$ are equal, up
to a translation by $\frac{1}{2}$. 
In particular,
\begin{multline*} 
\frac{1}{n-1} (\double S_n) (\xi) = \frac{1}{n-1} S_n (D(\xi)) = 
\int_\R  z^{n-2} \sigma_{\omega_{D(\xi)}} (z) \dif z =\\
\int_\R  \left(z+\frac{1}{2}\right)^{n-2} \sigma_{\omega_\xi} (z) \dif z = 
\sum_{\substack{0\leq k\leq n-2,\\ \text{$k$ is even}}} 
         \binom{n-2}{k} \frac{1}{2^{n-k-2} (k+1)} S^{\spin}_{k+2}(\xi).
\end{multline*}
The collection of such equalities over even integers $n\in\{2,4,\dots,2k\}$
shows that the linear span (with rational coefficients) of the functions
$\double S_2, \linebreak \double S_4, \dots, \double S_{2k}$ is equal to the
linear span (also with rational coefficients) of the functions $S^{\spin}_2,
S^{\spin}_4,\dots, S^{\spin}_{2k}$. This has a twofold consequence. Firstly,
$S^{\spin}_{2n}\in \Gamma$ is of degree  $2n$, as required. Secondly,
$(S^{\spin}_2,S^{\spin}_4,\dots)$ generate $\Gamma$ as a filtered algebra, as
required.

The systems of equations \eqref{eq:r-in-s} and \eqref{eq:s-in-r} imply that the
analogous claims hold as well for the free cumulants
$R^{\spin}_2,R^{\spin}_4,\dots$.
\end{proof}

\section{The key tool}
\label{sec:key}

Suppose that we are given a (sequence of) spin representation(s) and the
corresponding (sequence of) character ratio(s).  As we already mentioned,
sometimes it is more convenient to pass from the character ratio to the
corresponding cumulants. Interestingly, there are \emph{three} distinct natural
types of such cumulants, each having its own advantages. In the current section
we will review them and prove the key tool of the current paper,
\cref{theo:key-tool} which provides a link between them.

\subsection{Three types of cumulants for the character ratios}

We will use the setup considered in \cref{sec:AFP-partitions-revisited},
i.e.~$\big( V^{(n)} \big)$ is a sequence of representations and
$\big(\chi^{(n)}\big)$ is the corresponding sequence of the character ratios.

\subsubsection{Cumulants of partitions} 

The first type of cumulants we will use are the ones which correspond to the
linear map $\chi \colon \C[\OP] \to \Sequences$, see
Equation~\eqref{eq:chi-to-sequences}. These cumulants 
$\kumu^\chi_\ell(\pi_1,\dots,\pi_\ell)$ are indexed by odd partitions
$\pi_1,\dots,\pi_\ell\in\OP$.

The advantage of these cumulants lies in the observation that in the
applications we are often given a representation in terms of its characters and
thus such cumulants can be often calculated explicitly without much effort.
Regretfully, these cumulants do not have a truly probabilistic interpretation.
This kind of cumulants will appear in \cref{theo:key-tool} within conditions
\ref{item:AFP-partitions-one} and \ref{item:AFP-partitions-multi}.

\subsubsection{Cumulants in $\Gamma$}
\label{sec:cumulants-probabilistic}

Let us fix for a moment an integer $n\geq 1$. We consider the discrete
probability space $\Omega_n=\SP_n$ equipped with the probability distribution
$\PP^{V^{(n)}}$; we denote by $\xi^{(n)}$ a random strict partition in $\SP_n$
with the same probability distribution $\PP^{V^{(n)}}$. By restricting the
domain of the functions, any element $X\in\Gamma$ can be viewed as a function on
the set $\SP_n$ or, equivalently, as a random variable on the probability space
$\Omega_n$. We are interested in its expected value $\E^{(n)} X= \E X\big(
\xi^{(n)} \big)$.

We define a unital map $\E_\Gamma \colon \Gamma \to \Sequences$ by setting
\[ \E_\Gamma  X = \big( \E^{(1)} X, \E^{(2)} X, \dots \big)\]
for any $X\in\Gamma$. In the case when $X$ is a normalized spin character, this
definition takes the following more concrete form
\[ \E_\Gamma (\ChSpin_\pi) = (x_1,x_2,\dots), \]
where 
\[ x_n=  \E\ \Ch_{\pi}(\xi^{(n)}) =
n^{\downarrow |\pi|}\ 2^{\frac{\|\pi\|}{2}}
\ \chi^{(n)}(\pi). \]
The last equality is a consequence of the definition
\eqref{eq:projective-normalized-trololo} of the normalized spin characters.

The cumulants $\kumu^{\E_\Gamma}_\ell(X_1,\dots,X_\ell)$ which correspond to this map
have a direct probabilistic meaning. This kind of cumulants will appear in
\cref{theo:key-tool} within conditions \ref{item:AFP-gamma-special} and
\ref{item:AFP-gamma-general}.

\subsubsection{Cumulants in $\Gamma_\bullet$}

The equality $\Gamma=\Gamma_{\bullet}$ between the vector spaces allows us to
view the aforementioned map $\E_{\Gamma}$ as a function on $\Gamma_{\bullet}$.
We will denote it by $\E_{\Gamma_\bullet} \colon \Gamma_\bullet \to \Sequences$.
The cumulants $\kumu^{\E_{\Gamma_\bullet}}_\ell(X_1,\dots,X_\ell)$ which correspond to
this map do not have a probabilistic meaning. This kind of cumulants will appear
in \cref{theo:key-tool} within conditions \ref{item:AFP-gamma-bullet-special}
and \ref{item:AFP-gamma-bullet-general}.

Since the algebras $\Gamma$ and $\Gamma_{\bullet}$ have different multiplicative
structures, the cumulants for the maps $\E_{\Gamma}$ and $\E_{\Gamma_{\bullet}}$
are also different.

\subsection{The key tool} 

The following theorem is a direct analogue of a result of \sniady \cite[Theorem
and Definition 1]{Sniady2006} which concerns the usual (non-projective)
representations of the symmetric groups, see also \cite[Theorem
2.3]{DolegaSniady2018} for a generalization to Jack characters.

In the following for an integer (or half-integer) $k$ we denote by
$n^k:=(1^k,2^k,\dots)\in\Sequences$ the sequence of powers of the integers.

\begin{theorem}
\label{theo:key-tool} 

Assume that for each integer $n\geq 1$ we are given a representation $V^{(n)}$
of $\SGA{n}$.

Then the following conditions are equivalent:
\begin{enumerate}[label=\emph{(\alph*)}]
	
\item 
\label{item:AFP-partitions-one}
for all odd partitions $\pi_1=(k_1), \dots, \pi_\ell=(k_\ell)\in\OP$
\emph{which consist of exactly one part}
\[ \kumu_\ell^{\chi}( \pi_1,\dots,\pi_\ell) 
   \  n^{\frac{\|\pi_1\|+\cdots+\|\pi_\ell\|+2(\ell-1) }{2}} = O(1), \]

\item 
\label{item:AFP-partitions-multi}
for all odd partitions $\pi_1, \dots,\pi_\ell\in\OP$
\[ \kumu_\ell^{\chi}( \pi_1,\dots,\pi_\ell) 
    \  n^{\frac{\|\pi_1\|+\cdots+\|\pi_\ell\|+2(\ell-1) }{2}} = O(1), \]

\item 
\label{item:AFP-gamma-special}
there exists a set $X\subseteq \Gamma$ which generates $\Gamma$ as a
filtered algebra with the property that for all $x_1,\dots,x_\ell\in X$
\[ \kumu_\ell^{\E_{\Gamma}}( x_1,\dots,x_\ell) 
    \  n^{-\frac{\DEGREE x_1+ \cdots +\DEGREE x_l- 2(\ell-1)}{2}} = O(1), \]

\item 
\label{item:AFP-gamma-general}
for all $x_1,\dots,x_\ell\in \Gamma$
\[ \kumu_\ell^{\E_{\Gamma}}( x_1,\dots,x_\ell) 
     \  n^{-\frac{\DEGREE x_1+ \cdots +\DEGREE x_\ell- 2(\ell-1)}{2}} = O(1), \]

\item 
\label{item:AFP-gamma-bullet-special}
there exists a set $X\subseteq \Gamma_\bullet$ which generates $\Gamma_\bullet$ as a
filtered algebra with the property that for all $x_1,\dots,x_\ell\in X$
\[ \kumu_\ell^{\E_{\Gamma_\bullet}}( x_1,\dots,x_\ell) 
    \  n^{-\frac{\DEGREE x_1+ \cdots +\DEGREE x_\ell- 2(\ell-1)}{2}} = O(1), \]

\item 
\label{item:AFP-gamma-bullet-general}
for all $x_1,\dots,x_\ell\in \Gamma$
\[ \kumu_\ell^{\E_{\Gamma_\bullet}}( x_1,\dots,x_\ell) 
    \  n^{-\frac{\DEGREE x_1+ \cdots +\DEGREE x_\ell- 2(\ell-1)}{2}} = O(1), \]

\end{enumerate}

Furthermore, if the left-hand side of any of the expressions above has a limit
for $n\to \infty$ for all $\ell\leq 2$ and all prescribed choices of
$\pi_1,\dots,\pi_\ell$ (respectively, $x_1,\dots,x_\ell$), then each of the the
left-hand sides of the above expressions has a limit for $n\to\infty$. These
limits for $\ell\leq 2$ are interrelated as follows: for any odd integers
$k,k_1,k_2\geq 1$ 
\begin{align*}
\free_{k+1} & := \lim_{n\to\infty} n^{\frac{k-1}{2}}\ \chi^{(n)}(k) = \\
& = \lim_{n\to\infty} n^{\frac{k-1}{2}}\ \E \chi^{\xi^{(n)}}(k) = 
\\
& = \lim_{n\to\infty} 
(2n)^{-\frac{k+1}{2}}\cdot 2 \cdot \E \Ch^{\spin}_{k}(\xi^{(n)})  = \\
& = \lim_{n\to\infty} 
(2n)^{-\frac{k+1}{2}}\ \E R^{\spin}_{k+1}(\xi^{(n)}),  \\[2ex]	
\covarianceDisjoint_{k_1+1,k_2+1} & := 
\lim_{n\to\infty} n^{\frac{k_1+k_2}{2}}\cdot 2\cdot 
\left(  \chi^{(n) }(k_1,k_2) - \chi^{(n) }(k_1)\ \chi^{(n) }(k_2)  \right), \\[2ex]
\covarianceProba_{k_1+1,k_2+1} & := 
\lim_{n\to\infty} n^{\frac{k_1+k_2}{2}}\cdot 
2\cdot \Cov\left(  \chi^{\xi^{(n)}}(k_1), \chi^{\xi^{(n)}}(k_2)  \right)  \\
 & = \lim_{n\to\infty} (2n)^{-\frac{k_1+k_2}{2}} 
 \Cov\left(  R^{\spin}_{k_1+1}(\xi^{(n)}), R^{\spin}_{k_2+1}(\xi^{(n)})  \right), \\
\intertext{where $\xi^{(n)}$ denotes a random strict partition 
	with the probability distribution $\PP^{V^{(n)}}$, }
\covarianceProba_{k_1+1,k_2+1}  &= 
\covarianceDisjoint_{k_1+1,k_2+1} -2 k_1 k_2\ \free_{k_1+1} \free_{k_2+1}+ \\ & + 2
\sum_{r\geq 1} \frac{k_1 k_2}{r} 
\sum_{(a_i), (b_i) }
\free_{a_1+b_1} \cdots \free_{a_r+b_r},  
\end{align*}
where the last sum runs over integers
$a_1,\dots,a_r,b_1,\dots,b_r\geq 1$ such that $a_1+\cdots+a_r=k_1$, and
$b_1+\cdots+b_r=k_2$ and for each $i\in[r]$ the sum $a_i+b_i$ is even.

\end{theorem}

\begin{proof}
The proof is fully analogous to \cite[Theorem and Definition 1, Theorem
3]{Sniady2006}; we just have to make sure that all ingredients of the old proof
have their spin counterparts. The list of new ingredients: the information about
disjoint cumulants is provided by \cref{theo:afp-for-spin}, the link between
non-normalized and normalized spin characters is provided by definition
\eqref{eq:projective-normalized-trololo}. The calculation is fully analogous to
the linear case, however one should carefully track the powers of $2$ which come
from two distinct sources: the choice of normalization used in
\eqref{eq:projective-normalized-trololo} as well as \eqref{eq:second-cumulant}.
\end{proof}

\subsection{Proof of \cref{theo:mainLLN}}
\label{sec:proof-of-LLN}

\begin{proof}[Proof of \cref{theo:mainLLN}]
	The proof of a very similar result presented in \cite[Section
5]{DolegaSniady2018} works also in our setup. In fact, we consider the simplest
possible case of the latter result (with the notations of
\cite{DolegaSniady2018} this corresponds to $\alpha=1$, $g=0$). We just have to
make sure that all ingredients of the proof have their spin counterparts and
indeed \cref{theo:key-tool} provides the necessary tools.

Alternatively, it was observed by De Stavola \cite[Section 4.5]{DeStavolaThesis}
that central limit theorem (a la \cref{theo:mainCLT}) for shifted Young diagrams implies also law of large numbers.
The proof of De Stavola can be adapted easily to our context.
\end{proof}

\subsection{Proof of \cref{theo:mainCLT} }
\label{sec:proof-theomainCLT}

\begin{proof}[Proof of \cref{theo:mainCLT}] 
This proof is fully analogous to the proof of \cite[Corollary 4]{Sniady2006}: in
order to show convergence to the Gaussian distribution it is enough to check
that the higher cumulants converge to zero.
\end{proof}

\section{Applications and examples}
\label{sec:applications}

\subsection{Example: shifted Plancherel measure}
\label{sec:plancherel}

We continue the discussion of the spin part of the left-regular representation
from \cref{example:Plancherel}. \cref{theo:mainLLN} is applicable in this
context with
\begin{align*} 
      \free_{k} &= \begin{cases}
					1 & \text{if $k=2$}, \\
					0 & \text{if $k\neq 2$},
					\end{cases} \\[1ex]
	 \covarianceDisjoint_{k,l} &= 0
\end{align*}				
for any even integers $k,l\geq 2$.

The limit shape $\Omega$ with the corresponding sequence of free cumulants given
by \eqref{eq:our-beloved-free-cumulants} is uniquely determined to the
Logan--Shepp \& Vershik--Kerov curve, see \cite[example on pages
133--134]{Biane1998}. In this way we recover the result of De Stavola \cite[Section 4.5]{DeStavolaThesis}.

\subsection{Example: Schur--Weyl measure}
\label{sec:example-SW}

\subsubsection{Proof of \cref{theo:schur-weyl}}
\label{sec:proof-theo-schur-weyl}

\begin{proof}[Proof of \cref{theo:schur-weyl}]
We consider the vector space
\[ V^{\SWW}_{d,n}= \left( \C^d \oplus \C^d \right)^{\otimes n} \]
which with the action of the spin symmetric group $\Spin{n}$
given by \cite[Section 4.3 combined with (2.5)]{Wan2012}
becomes a (reducible) spin superrepresentation. 

Its decomposition into irreducible components is provided by \emph{Sergeev
	duality} \cite[Theorem 4.8]{Wan2012}; in particular it follows that the
corresponding probability measure $\PP^{V^{\SWW}_{d,n}} $ on $\SP_n$ coincides
with the measure $\SW{n}{d}$ considered in \eqref{eq:SW-measure}.

The character of $V^{\SWW}_{d,n}$ is given by
\[ \chi^{\SWW}_{d,n}(\pi) = \frac{1}{(\sqrt{2}\ d)^{\| \pi \|}} \]
for any $\pi\in \OP_n$. In the same way as in \cref{example:Plancherel} we
calculate the corresponding cumulants; it follows that
\[ \kumu_\ell(\pi_1,\dots,\pi_\ell)= 0 \]
for all $\ell\geq 2$.  

We set $d=d_n$; it follows that \cref{theo:mainLLN} is applicable
with 
\begin{align*}
\free_{k+1} &= \left( \frac{c}{\sqrt{2}} \right)^{k-1}, \\ 
\covarianceDisjoint_{k_1+1,k_2+1} &= 0
\end{align*}
for any odd integers $k,k_1,k_2\geq 1$.

The explicit form of the limit shape $\Omega^{\SWW}_c$ will be found in
\cref{sec:SW-measure-shape}.
\end{proof}

\subsubsection{Proof of \cref{coro:schur-weyl-recording}}

\begin{proof}[Proof of \cref{coro:schur-weyl-recording}]

We are interested in the rescaled diagram $\omega_{\frac{1}{d_n} \xi}$, where $\xi:=
\big( Q(\mathbf{w}) \big)_{\leq
	\alpha d_n^2 }$ is the level curve of the recording tableau.
The distribution of the random shifted Young diagram $\xi$
 is given by Schur--Weyl measure $\SW{\lfloor \alpha
	d_n^2\rfloor}{d}$. It follows that \cref{theo:schur-weyl} is applicable with 
\[ c:= \lim_{n\to\infty} \frac{\sqrt{\lfloor \alpha
	d_n^2\rfloor}}{d_n}  = \sqrt{\alpha}
\]
and that the convergence 
\[ \omega_{ \frac{1}{\sqrt{2 \lfloor \alpha
		d_n^2\rfloor}} \xi} \longrightarrow \Omega^{\SWW}_{\sqrt{\alpha}} \]
in the usual sense holds true.

By comparing the normalization factors it follows that $\omega_{\frac{1}{d_n}
	\xi}$ is a dilation of $\omega_{ \frac{1}{\sqrt{2 \lfloor \alpha d_n^2\rfloor}}
	\xi}$ by the factor $\frac{\sqrt{2 \lfloor \alpha d_n^2\rfloor}}{d_n}\to \sqrt{2
	\alpha}$. It follows that the theorem holds for $\Omega^{\LimitQ}_{\alpha}$
equal to the dilation of $\Omega^{\SWW}_{\sqrt{\alpha}}$ by the factor $\sqrt{2
	\alpha}$:
\[ \Omega^{\LimitQ}_{\alpha}(z) = 
   \sqrt{2 \alpha}\ \Omega^{\SWW}_{\sqrt{\alpha}}\left( \frac{z}{\sqrt{2 \alpha}}\right).\]

The limit shape $\LimitQ$ can be then recovered from the family of curves
$\Omega^{\LimitQ}_{\alpha}$ by setting
\[ \LimitQ(x,y) = \sup\{ \alpha\geq 0 :
\Omega^{\LimitQ}_{\alpha}(x-y) \leq x+y \}.\]

\end{proof}

\subsubsection{Proof of \cref{coro:schur-weyl-insertion}}

\begin{proof}[Proof of \cref{coro:schur-weyl-insertion}]

The level curve $\xi=\big( P(\mathbf{w}) \big)_{\leq \alpha \frac{n}{d} }$ can
be obtained by the following two-step procedure: firstly, we remove from the
random word $\mathbf{w}$ all entries which are bigger than the threshold
$\alpha \frac{n}{d}$ and denote the resulting word by $\mathbf{w}'$; then we consider
the common shape of the tableaux associated by shifted RSK correspondence to
$\mathbf{w}'$.

We denote by $n'$ the random length of the word $\mathbf{w}'$. Since the
distribution of $n'$ is $\operatorname{Bin}(n,\alpha \frac{n}{d^2})$ it follows
that
\[ \frac{n'}{\alpha \frac{n^2}{d^2}}\longrightarrow 1 \]
holds true in probability.

The limit shape of $ \omega_{\frac{1}{\sqrt{2 n'}}\xi} $
is provided by \cref{theo:schur-weyl} with
\[ c = \lim_{n\to\infty} \frac{\sqrt{n'}}{ \lfloor \alpha \frac{n}{d}\rfloor } =\frac{1}{\sqrt{\alpha}}; \]
thus 
\[ \omega_{\frac{1}{\sqrt{2 n'}}\xi}  \longrightarrow \Omega^{\SWW}_{\frac{1}{\sqrt{\alpha}}}.\]

By comparison of the normalization factors it follows that $\omega_{\frac{d}{n}
	\xi}$ is a dilation of $\omega_{\frac{1}{\sqrt{2 n'}}\xi}$ by the factor
$\frac{d \sqrt{2 n'}}{n}\to \sqrt{2\alpha} $. It follows that the theorem holds
for $\Omega^{\LimitP}_{\alpha}$ equal to the dilation of
$\Omega^{\SWW}_{\frac{1}{\sqrt{\alpha}}}$ by the factor $\sqrt{2 \alpha}$:
\[ \Omega^{\LimitP}_{\alpha}(z) = 
  \sqrt{2 \alpha}\ \Omega^{\SWW}_{\frac{1}{\sqrt{\alpha}}}\left( \frac{z}{\sqrt{2 \alpha}}\right).\]
The limit shape $\LimitP$ can be then recovered from the family of curves
$\Omega^{\LimitP}_{\alpha}$ by setting
\[ \LimitP(x,y) = \sup\{ \alpha\geq 0 :
\Omega^{\LimitP}_{\alpha}(x-y) \leq x+y \}.\]

\end{proof}

\subsection{Random shifted tableaux with prescribed shape.  Proof of
	\cref{theo:randomSYT}.} 

\label{sec:proof-randomSYT}

The current section is devoted to the proof of \cref{theo:randomSYT}. For
Reader's convenience the proof is split into several subsections. Note that some
of the intermediate results might be interesting by themselves.

\subsubsection{Random shifted tableaux and the spin representation theory}

We start with the link between random \emph{shifted} standard tableaux and the
\emph{spin} representation theory of the symmetric groups. This link
(\cref{prop:Bratteli-equal}) is fully analogous to the classical link between
\emph{non-shifted} standard tableaux and the \emph{linear} representation theory
of the symmetric groups. Unfortunately, the proof of \cref{prop:Bratteli-equal}
is more complex than its classical counterpart because not all edges in the
Bratteli diagram are simple. We present the details below.

\subsubsection{Random shifted tableaux} 
\label{sec:random-young-tableaux} 

For a given shifted standard tableau $T$ with shape $\mu\in\SP_n$ and an integer
$0\leq m\leq n$ we denote by $T_{\leq m}$ the set of boxes of $T$ which are
smaller or equal than $m$. This set of boxes corresponds to some $\xi\in\SP_m$.
Thus $T_{\leq m}:=\xi$ can be regarded as a shifted partition.

We fix $\mu\in\SP_n$ and consider a uniformly random tableau
$T\in \tableaux_\mu$ with this prescribed shape.
\emph{We are interested in the probability distribution of 
the random strict partition $T_{\leq m}$.}

\subsubsection{Restriction of irreducible representations}
\label{sec:restriction} 

We fix $\mu\in\SP_n$ and consider the corresponding
irreducible representation $V^\mu$ of $\SGA{n}$. 
For an integer $0\leq m\leq n$ the restriction
\begin{equation}
\label{eq:restriction}
V^\mu\left\downarrow^{\SGA{n}}_{\SGA{m}}\right.
\end{equation}
is  a representation of the spin group algebra $\SGA{m}$.
\emph{We are interested in the corresponding probability measure
\[ \PP^{V^\mu \left\downarrow^{\SGA{n}}_{\SGA{m}}\right.}\]
on the set $\SP_m$ which gives the probability distribution
of a random irreducible component of \eqref{eq:restriction}.}

\subsubsection{The link}

The following result shows that the two probability measures on $\SP_m$ which we
considered in \cref{sec:random-young-tableaux,sec:restriction} are
equal.

\begin{proposition}
\label{prop:Bratteli-equal}
Let $\mu\in\SP_n$ and $0\leq m\leq n$. 
Then for any $\xi\in\SP_m$
\begin{equation} 
\label{eq:restriction-random-tableau}
\PP_{\tableaux_\mu}( T \in\tableaux_\mu :
T_{\leq m} = \xi) =
\PP^{V^\mu \left\downarrow^{\SGA{n}}_{\SGA{m}}\right.}(\xi).
 \end{equation}
\end{proposition}
\begin{proof}
We consider the Bratteli diagram for the sequence of superalgebras
$\SGA{0}\subseteq \SGA{1} \subseteq \cdots$. The vertices of this graph
correspond to irreducible superrepresentations of $\SGA{n}$ over $n\geq 0$ and
hence can be identified with the strict partitions. This kind of graph (with
different multiplicities of the edges) is known in the literature under the name
of \emph{Schur graph} \cite{Borodin1997,Ivanov2004}.

We connect a superrepresentation $\xi$ of $\SGA{n-1}$ by an oriented edge with a
superrepresentation $\zeta$ of $\SGA{n}$ if $\xi$ appears in the decomposition
of the restriction $\zeta\big\downarrow^{\SGA{n}}_{\SGA{n-1}}$ into irreducible
components. We declare the multiplicity of the edge to be equal to the
multiplicity of $\xi$ in $\zeta\big\downarrow^{\SGA{n}}_{\SGA{n-1}}$. In the
following we will describe the structure of this graph based on \cite[Theorem
22.3.4]{Kleshchev2005}.

For strict partitions $\xi,\zeta$ we write $\xi\nearrow\zeta$ if the shifted
Young diagram $\xi$ is obtained from $\zeta$ by removal of a single box. It
turns out that in the Bratteli diagram there is an oriented edge from $\xi$ to
$\zeta$ if and only if $\xi\nearrow\zeta$. If this is the case, the multiplicity
of this edge is given by
\[ m(\xi, \zeta) = 
\begin{cases}
2 & \text{if $\xi\in\SP^+$  and $\zeta\in\SP^-$}, \\
1 & \text{otherwise},
\end{cases}
\]
see \cref{fig-younggraph}.

\begin{figure}
	\begin{tikzpicture}[scale=0.25]
	\tikzset{
		odd/.style={blue,fill=blue!10}
	}
	
	\tikzset{
		even/.style={}
	}
	
	\coordinate (d0) at (-6,0);
	\draw (d0) node {$\emptyset$};

	\coordinate (d1) at (0,0);
	\draw[even] (d1)  
	++(-0.5,-0.5) rectangle +(1,1);

	\coordinate (d2) at (6,0);
	\draw[odd] 
	(d2) ++(-1,-0.5) 
	rectangle +(1,1) 
	{[current point is local] ++(1,0) rectangle +(1,1)};

	\coordinate (d3) at (12,-2.5);
	\draw[even] (d3) 
	++(-1.5,-0.5) 
	rectangle +(1,1) 
	{[current point is local] ++(1,0) rectangle +(1,1)}
	{[current point is local] ++(2,0) rectangle +(1,1)};

	\coordinate (d21) at (12,2.5);
	\draw[odd] (d21) 
	++(-1,-1) 
	rectangle +(1,1) 
	{[current point is local] ++(1,1) rectangle +(1,1)}
	{[current point is local] ++(1,0) rectangle +(1,1)};

	\coordinate (d4) at (20,-2.5);
	\draw[odd] (d4) ++(-2,-0.5) 
	rectangle +(1,1) 
	{[current point is local] ++(1,0) rectangle +(1,1)}
	{[current point is local] ++(2,0) rectangle +(1,1)}
	{[current point is local] ++(3,0) rectangle +(1,1)};

	\coordinate (d31) at (20,2.5);
	\draw[even] (d31) 
	++(-1.5,-1) 
	rectangle +(1,1) 
	{[current point is local] ++(1,1) rectangle +(1,1)}
	{[current point is local] ++(1,0) rectangle +(1,1)}
	{[current point is local] ++(2,0) rectangle +(1,1)};

	\coordinate (d5) at (28,-5);
	\draw[even] (d5) 
	++(-2.5,-0.5) 
	rectangle +(1,1) 
	{[current point is local] ++(1,0) rectangle +(1,1)}
	{[current point is local] ++(2,0) rectangle +(1,1)}
	{[current point is local] ++(3,0) rectangle +(1,1)}
	{[current point is local] ++(4,0) rectangle +(1,1)};
	
	\coordinate (d41) at (28,0);
	\draw[odd] (d41) 
	++(-2,-1) 
	rectangle +(1,1) 
	{[current point is local] ++(1,1) rectangle +(1,1)}
	{[current point is local] ++(1,0) rectangle +(1,1)}
	{[current point is local] ++(2,0) rectangle +(1,1)}
	{[current point is local] ++(3,0) rectangle +(1,1)};
	
	\coordinate (d32) at (28,5);
	\draw[odd] (d32) 
	++(-1.5,-1) 
	rectangle +(1,1) 
	{[current point is local] ++(1,0) rectangle +(1,1)}
	{[current point is local] ++(2,0) rectangle +(1,1)}
	{[current point is local] ++(1,1) rectangle +(1,1)}
	{[current point is local] ++(2,1) rectangle +(1,1)}
	;

	\coordinate (d6) at (36,-7.5);
	\draw[odd] (d6) 
	++(-3,-0.5) 
	rectangle +(1,1) 
	{[current point is local] ++(1,0) rectangle +(1,1)}
	{[current point is local] ++(2,0) rectangle +(1,1)}
	{[current point is local] ++(3,0) rectangle +(1,1)}
	{[current point is local] ++(4,0) rectangle +(1,1)}
	{[current point is local] ++(5,0) rectangle +(1,1)};
	
	\coordinate (d51) at (36,-2.5);
	\draw[even] (d51) 
	++(-2.5,-1) 
	rectangle +(1,1) 
	{[current point is local] ++(1,1) rectangle +(1,1)}
	{[current point is local] ++(1,0) rectangle +(1,1)}
	{[current point is local] ++(2,0) rectangle +(1,1)}
	{[current point is local] ++(3,0) rectangle +(1,1)}
	{[current point is local] ++(4,0) rectangle +(1,1)}
	;
	
	\coordinate (d42) at (36,2.5);
	\draw[even] (d42) ++(-2,-1) 
	rectangle +(1,1) 
	{[current point is local] ++(1,0) rectangle +(1,1)}
	{[current point is local] ++(2,0) rectangle +(1,1)}
	{[current point is local] ++(3,0) rectangle +(1,1)}
	{[current point is local] ++(1,1) rectangle +(1,1)}
	{[current point is local] ++(2,1) rectangle +(1,1)}
	;
	
	\coordinate (d321) at (36,7.5);
	\draw[odd] (d321) ++(-1.5,-1.5) 
	rectangle +(1,1) 
	{[current point is local] ++(1,0) rectangle +(1,1)}
	{[current point is local] ++(2,0) rectangle +(1,1)}
	{[current point is local] ++(1,1) rectangle +(1,1)}
	{[current point is local] ++(2,1) rectangle +(1,1)}
	{[current point is local] ++(2,2) rectangle +(1,1)}
	;

	\draw[->] ($ (d0)!1.5 cm!(d1) $) -- ($ (d1)!1.5 cm!(d0) $);
	
	\draw[->,double] ($ (d1)!1.5 cm!(d2) $) -- ($ (d2) !1.5 cm! (d1)$) ;
	
	\draw[->] ($ (d2)!1.5 cm!(d3) $) -- ($ (d3) !2 cm! (d2)$) ;
	\draw[->] ($ (d2)!1.5 cm!(d21) $) -- ($ (d21) !1.5 cm! (d2)$) ;
	
	\draw[->,double] ($ (d3)!2 cm!(d4) $) -- ($ (d4) !2.5 cm! (d3)$) ;
	\draw[->] ($ (d3)!2 cm!(d31) $) -- ($ (d31) !2.5 cm! (d3)$) ;
	\draw[->] ($ (d21)!1.5 cm!(d31) $) -- ($ (d31) !2.5 cm! (d21)$) ;
	
	\draw[->] ($ (d4)!3 cm!(d5) $) -- ($ (d5) !3 cm! (d4)$) ;
	\draw[->] ($ (d4)!3 cm!(d41) $) -- ($ (d41) !2.5 cm! (d4)$) ;
	\draw[->,double] ($ (d31)!2 cm!(d41) $) -- ($ (d41) !2.5 cm! (d31)$) ;
	\draw[->,double] ($ (d31)!2 cm!(d32) $) -- ($ (d32) !2.5 cm! (d31)$) ;
	
	\draw[->,double] ($ (d5)!3 cm!(d6) $) -- ($ (d6) !3.5 cm! (d5)$) ;
	\draw[->] ($ (d5)!3 cm!(d51) $) -- ($ (d51) !3 cm! (d5)$) ;
	\draw[->] ($ (d41)!2.5 cm!(d51) $) -- ($ (d51) !2.5 cm! (d41)$) ;
	\draw[->] ($ (d41)!2 cm!(d42) $) -- ($ (d42) !2.5 cm! (d41)$) ;
	\draw[->] ($ (d32)!2 cm!(d42) $) -- ($ (d42) !2.5 cm! (d32)$) ;
	\draw[->] ($ (d32)!2 cm!(d321) $) -- ($ (d321) !2.5 cm! (d32)$) ;

	\end{tikzpicture}
	
	\caption{Part of the Bratteli diagram for superrepresentations of spin group
	alegbras $\SGA{n}$. Shaded diagrams correspond to the elements of $\SP^-$
	\emph{[which correspond to the superrepresentations of type $\mathtt{Q}$]}
	while non-shaded diagrams correspond to the elements of $\SP^+$ \emph{[which
		correspond to the superrepresentations of type $\mathtt{M}$]}. Double edges are
	drawn only on the edges from the elements of $\SP^+$ to the elements of
	$\SP^-$; the remaining edges are simple.} \label{fig-younggraph}

	\vspace{5ex}

	\begin{tikzpicture}[scale=0.25]
	\tikzset{
		odd/.style={}
	}
	
	\tikzset{
		even/.style={}
	}
	
	\coordinate (d0) at (-6,0);
	\draw (d0) node {$\emptyset$};

	\coordinate (d1) at (0,0);
	\draw[even] (d1)  
	++(-0.5,-0.5) rectangle +(1,1);

	\coordinate (d2) at (6,0);
	\draw[odd] 
	(d2) ++(-1,-0.5) 
	rectangle +(1,1) 
	{[current point is local] ++(1,0) rectangle +(1,1)};

	\coordinate (d3) at (12,-2.5);
	\draw[even] (d3) 
	++(-1.5,-0.5) 
	rectangle +(1,1) 
	{[current point is local] ++(1,0) rectangle +(1,1)}
	{[current point is local] ++(2,0) rectangle +(1,1)};

	\coordinate (d21) at (12,2.5);
	\draw[odd] (d21) 
	++(-1,-1) 
	rectangle +(1,1) 
	{[current point is local] ++(1,1) rectangle +(1,1)}
	{[current point is local] ++(1,0) rectangle +(1,1)};

	\coordinate (d4) at (20,-2.5);
	\draw[odd] (d4) ++(-2,-0.5) 
	rectangle +(1,1) 
	{[current point is local] ++(1,0) rectangle +(1,1)}
	{[current point is local] ++(2,0) rectangle +(1,1)}
	{[current point is local] ++(3,0) rectangle +(1,1)};

	\coordinate (d31) at (20,2.5);
	\draw[even] (d31) 
	++(-1.5,-1) 
	rectangle +(1,1) 
	{[current point is local] ++(1,1) rectangle +(1,1)}
	{[current point is local] ++(1,0) rectangle +(1,1)}
	{[current point is local] ++(2,0) rectangle +(1,1)};

	\coordinate (d5) at (28,-5);
	\draw[even] (d5) 
	++(-2.5,-0.5) 
	rectangle +(1,1) 
	{[current point is local] ++(1,0) rectangle +(1,1)}
	{[current point is local] ++(2,0) rectangle +(1,1)}
	{[current point is local] ++(3,0) rectangle +(1,1)}
	{[current point is local] ++(4,0) rectangle +(1,1)};
	
	\coordinate (d41) at (28,0);
	\draw[odd] (d41) 
	++(-2,-1) 
	rectangle +(1,1) 
	{[current point is local] ++(1,1) rectangle +(1,1)}
	{[current point is local] ++(1,0) rectangle +(1,1)}
	{[current point is local] ++(2,0) rectangle +(1,1)}
	{[current point is local] ++(3,0) rectangle +(1,1)};
	
	\coordinate (d32) at (28,5);
	\draw[odd] (d32) 
	++(-1.5,-1) 
	rectangle +(1,1) 
	{[current point is local] ++(1,0) rectangle +(1,1)}
	{[current point is local] ++(2,0) rectangle +(1,1)}
	{[current point is local] ++(1,1) rectangle +(1,1)}
	{[current point is local] ++(2,1) rectangle +(1,1)}
	;

	\coordinate (d6) at (36,-7.5);
	\draw[odd] (d6) 
	++(-3,-0.5) 
	rectangle +(1,1) 
	{[current point is local] ++(1,0) rectangle +(1,1)}
	{[current point is local] ++(2,0) rectangle +(1,1)}
	{[current point is local] ++(3,0) rectangle +(1,1)}
	{[current point is local] ++(4,0) rectangle +(1,1)}
	{[current point is local] ++(5,0) rectangle +(1,1)};
	
	\coordinate (d51) at (36,-2.5);
	\draw[even] (d51) 
	++(-2.5,-1) 
	rectangle +(1,1) 
	{[current point is local] ++(1,1) rectangle +(1,1)}
	{[current point is local] ++(1,0) rectangle +(1,1)}
	{[current point is local] ++(2,0) rectangle +(1,1)}
	{[current point is local] ++(3,0) rectangle +(1,1)}
	{[current point is local] ++(4,0) rectangle +(1,1)}
	;
	
	\coordinate (d42) at (36,2.5);
	\draw[even] (d42) ++(-2,-1) 
	rectangle +(1,1) 
	{[current point is local] ++(1,0) rectangle +(1,1)}
	{[current point is local] ++(2,0) rectangle +(1,1)}
	{[current point is local] ++(3,0) rectangle +(1,1)}
	{[current point is local] ++(1,1) rectangle +(1,1)}
	{[current point is local] ++(2,1) rectangle +(1,1)}
	;
	
	\coordinate (d321) at (36,7.5);
	\draw[odd] (d321) ++(-1.5,-1.5) 
	rectangle +(1,1) 
	{[current point is local] ++(1,0) rectangle +(1,1)}
	{[current point is local] ++(2,0) rectangle +(1,1)}
	{[current point is local] ++(1,1) rectangle +(1,1)}
	{[current point is local] ++(2,1) rectangle +(1,1)}
	{[current point is local] ++(2,2) rectangle +(1,1)}
	;

	\draw[->] ($ (d0)!1.5 cm!(d1) $) -- ($ (d1)!1.5 cm!(d0) $);
	
	\draw[->] ($ (d1)!1.5 cm!(d2) $) -- ($ (d2) !1.5 cm! (d1)$) ;
	
	\draw[->] ($ (d2)!1.5 cm!(d3) $) -- ($ (d3) !2 cm! (d2)$) ;
	\draw[->] ($ (d2)!1.5 cm!(d21) $) -- ($ (d21) !1.5 cm! (d2)$) ;
	
	\draw[->] ($ (d3)!2 cm!(d4) $) -- ($ (d4) !2.5 cm! (d3)$) ;
	\draw[->] ($ (d3)!2 cm!(d31) $) -- ($ (d31) !2.5 cm! (d3)$) ;
	\draw[->] ($ (d21)!1.5 cm!(d31) $) -- ($ (d31) !2.5 cm! (d21)$) ;
	
	\draw[->] ($ (d4)!3 cm!(d5) $) -- ($ (d5) !3 cm! (d4)$) ;
	\draw[->] ($ (d4)!3 cm!(d41) $) -- ($ (d41) !2.5 cm! (d4)$) ;
	\draw[->] ($ (d31)!2 cm!(d41) $) -- ($ (d41) !2.5 cm! (d31)$) ;
	\draw[->] ($ (d31)!2 cm!(d32) $) -- ($ (d32) !2.5 cm! (d31)$) ;
	
	\draw[->] ($ (d5)!3 cm!(d6) $) -- ($ (d6) !3.5 cm! (d5)$) ;
	\draw[->] ($ (d5)!3 cm!(d51) $) -- ($ (d51) !3 cm! (d5)$) ;
	\draw[->] ($ (d41)!2.5 cm!(d51) $) -- ($ (d51) !2.5 cm! (d41)$) ;
	\draw[->] ($ (d41)!2 cm!(d42) $) -- ($ (d42) !2.5 cm! (d41)$) ;
	\draw[->] ($ (d32)!2 cm!(d42) $) -- ($ (d42) !2.5 cm! (d32)$) ;
	\draw[->] ($ (d32)!2 cm!(d321) $) -- ($ (d321) !2.5 cm! (d32)$) ;

	\end{tikzpicture}
	
	\caption{Part of the shifted Young graph. This graph is obtained from Bratteli
		diagram (\cref{fig-younggraph}) by removing multiplicities from the edges.}
	\label{fig:shifted-Young}
\end{figure}

	We start by pointing out that the statement of \cref{prop:Bratteli-equal} is
	not very precise in the case when $\mu\in\SP^-$ since in this case there are
	\emph{two} irreducible representations which correspond to $\mu$. However, in
	the light of \cref{sec:only-ratio-matters} this subtlety is irrelevant because
	both representations have equal character ratios on $\OP$. In the following we
	change our setup to \emph{superrepresentations} of the superalgebras $\SGA{n}$
	and $\SGA{m}$; by the same argument this does not change the probability
	distribution on the right-hand side of \eqref{eq:restriction-random-tableau}.
	
	\medskip
	
	Our strategy is to evaluate the denominator and the numerator on the right-hand
side of \eqref{eq:probability-and-dimensions} in the special case when $V=V^\mu
\left\downarrow^{\SGA{n}}_{\SGA{m}}\right.$ is the restriction of the irreducible
\emph{superrepresentation} corresponding to $\mu$.
	
	\medskip
	
	The denominator, the dimension of the irreducible superrepresentation $V^\mu$ 
	\begin{equation}
	\label{eq:Bratteli-sum} 
	\dimm V^\mu = \sum_{\emptyset=\mu^0 \nearrow \cdots \nearrow \mu^n=\mu}
	m(\mu^0,\mu^1) \cdots m(\mu^{n-1},\mu^n) 
	\end{equation}
	is equal to the sum over all paths in the Bratteli diagram which connect the
	trivial one-dimensional representation $\emptyset$ with $\mu$. The weight of
	each path is equal to the product of the multiplicities of all edges. This
	product is a power of $2$ with the exponent $c$ equal to the number of the
	elements in the set
	\begin{equation}
	\label{eq:interesting-set}
	\big\{ 0\leq i \leq n-1 :  
	\text{$\|\mu^{i}\|$ is even and $\|\mu^{i+1}\|$ is odd} \big\}.
	\end{equation}
	Since
	\[ \|\mu^0\|, \dots, \|\mu^n\| \]
	is a weakly increasing sequence of integers which at each step increases by at
	most one, it follows that the cardinality of \eqref{eq:interesting-set} is equal
	to the number of even integers in the interval $\{0,\dots,\|\mu\|-1\}$
	which is equal to
	\[ c= c_\mu= 1+\left\lfloor \frac{\|\mu\|-1}{2} \right\rfloor. \]
	Note that this number does not depend on the choice of the path in the Bratteli
	diagram with a specified endpoint $\mu$.
	In this way we proved that
	\[ \dimm V^\mu = 
	  2^{c_\mu} \cdot \big( \text{number of paths $\emptyset=\mu^0 \nearrow \cdots \nearrow \mu^n=\mu$} \big). 
	\]
	
	The latter formula does not involve the multiplicities of the edges in the
Bratteli diagram; for this reason it is convenient to replace the latter by
\emph{Schur graph}, see \cref{fig:shifted-Young}. 
The shifted Young diagram is obtained from
Bratteli diagram by removing the multiplicities of the edges. Clearly, there is
a bijective correspondence between paths in the shifted Young graph and shifted
standard Young tableaux. In this way we proved that
	\begin{equation}
	\label{eq:dimension-tableaux}
	\dimm V^\mu = 2^{c_\mu} \cdot \left| \tableaux_\mu \right|. 
	\end{equation}

	\medskip
	
	The numerator on the right-hand side of \eqref{eq:probability-and-dimensions},
the product
	\[ n_\xi \dimm V^\xi = \sum_{\substack{\emptyset=\mu^0 \nearrow \cdots \nearrow \mu^n=\mu, \\ \mu^m = \xi}}
	m(\mu^0,\mu^1) \cdots m(\mu^{n-1},\mu^n) \]
	is a sum analogous to the right-hand side of \eqref{eq:Bratteli-sum} over paths
which pass through the vertex $\xi$. A reasoning analogous to the one above
implies that
	\begin{multline*}	n_\xi \dimm V^\xi = \\
	2^{c_\mu} \cdot \big( 
	     \text{number of paths $\emptyset=\mu^0 \nearrow \cdots \nearrow \mu^n=\mu$ such that $\mu^m=\xi$} \big). 
	\end{multline*}
	The paths which contribute to the second factor on the right-hand side are in a
bijective correspondence with the shifted tableaux
	\[ \left\{ T \in\tableaux_\mu : T_{\leq m} = \xi \right\}.\]
	
	\medskip
	
	The latter observation combined with \eqref{eq:dimension-tableaux} concludes
the proof.
\end{proof}

\subsubsection{Irreducible representations have approximate factorization property}
\label{sec:afp-irreducible}

We will show that a sequence of \emph{irreducible representations}
$(V^{\xi^{(k)}})$ has approximate factorization property
(\cref{def:afp-representations}), provided that the strict partitions
$\xi^{(k)}$ converge to some limit shape. It might seem silly to prove this kind
of result since the main purpose of \emph{approximate factorization property} is
to serve as the assumption of \cref{theo:mainLLN,theo:mainCLT} which are
trivially satisfied for irreducible representations. The key point is that
\cref{theo:key-tool} will allow us to produce new examples (often more
interesting ones) of the sequences with approximate factorization property from
old (and boring) examples.

\bigskip

Clearly, \cref{def:afp-representations} still makes sense and
\cref{theo:key-tool} remains true if we replace the sequence $(V^{(n)})$ of
representations by its subsequence $(V^{(n_k)})$ for some sequence $(n_k)$ of
positive integers which tends to infinity; in the following we will use
\cref{theo:key-tool} in this formulation.

\begin{proposition}
We keep the notations and the assumptions from \cref{theo:randomSYT}. We define
$V^{(n_k)}:=V^{\xi^{(k)}}$ to be the irreducible representation of $\SGA{n_k}$.
Then the sequence $(V^{\xi^{(k)}})$ has the approximate factorization property.
\end{proposition}
\begin{proof}
The corresponding probability measure $\PP^{V^{(n_k)}}$ is the point measure
supported in $\xi^{(k)}$. With the notations of
\cref{sec:cumulants-probabilistic} the corresponding random strict partition is
deterministic, almost surely equal to $\xi^{(k)}$. All higher cumulants
$\kumu_l$ (with $l\geq 2$) of deterministic random variables are equal to zero;
it follows that for $l\geq 2$ condition \ref{item:AFP-gamma-special} of
\cref{theo:key-tool} is trivially fulfilled for any choice of $X\subseteq \Gamma$.

Take $X=\{ S^{\spin}_2,S^{\spin}_4,\dots \}\subseteq\Gamma$; our plan is to
verify that condition \ref{item:AFP-gamma-special} of \cref{theo:key-tool} is
fulfilled also for $l=1$. Firstly, note that by \cref{prop:SR-generate} the set
$X$ indeed generates $\Gamma$ as a filtered algebra. The first cumulant
$\kumu_1$ takes a particularly simple form and
\[ \kumu_1^{\E_\Gamma}( S^{\spin}_{2j} )=
\E_\Gamma( S^{\spin}_{2j} ) = 
\big(\E^{(n_k)} S^{\spin}_{2j}\big)_{k \geq 1}  =
\big( S^{\spin}_{2j}(\xi^{(k)}) \big)_{k \geq 1}.
\]
The balancedness condition implies that the limit
\[ \lim_{k\to\infty} \frac{1}{n_k^{j}} S^{\spin}_{2j}(\xi^{(k)})= S_{2j}(\Omega_1)
\]
exists thus condition \ref{item:AFP-gamma-special} of \cref{theo:key-tool} is
indeed fulfilled.
\end{proof}

\subsubsection{Restriction of representations with approximate factorization property}
\label{sec:afp-restriction}

The following reasoning is based on its non-shifted analogue
\cite[Theorem 8]{Sniady2006}.

We fix $0<\alpha<1$.
We define $n'_k:= \lfloor \alpha n_k \rfloor$
and consider a restriction 
\[ W^{(n'_k)}:=
 V^{\xi^{(k)}} \big\downarrow^{\SGA{n_k}}_{\SGA{n'_k}}\]
which is a representation of $\SGA{n'_k}$. We claim that the sequence
$(W^{(n'_k)})$ also has approximate factorization property. Indeed, in order to
verify approximate factorization property (no matter if we consider
$(V^{\xi^{(k)}})$ or the sequence of its restrictions $(W^{(n'_k)})$) we need to
understand the asymptotics of the left-hand side of
\eqref{eq:decay-cumualts-characters}. The only difference that in the context of
$(V^{\xi^{(k)}})$ the right-hand side is
\begin{equation} 
\label{eq:asymptotics-restriction-A}
O\left( n_k^{- \frac{\|\pi^1\|+\cdots+\|\pi^\ell\| + 2(\ell-1) }{2}} \right) 
\end{equation}
while in the context of $(W^{(n'_k)})$) this right-hand side should be understood as
\begin{equation} 
\label{eq:asymptotics-restriction-B}
 O\left( (n'_k)^{- \frac{\|\pi^1\|+\cdots+\|\pi^\ell\| + 2(\ell-1) }{2}} \right).
\end{equation}
Asymptotically, expressions \eqref{eq:asymptotics-restriction-A} and
\eqref{eq:asymptotics-restriction-B} differ by a multiplicative factor which
concludes the proof of our claim.

It follows that \cref{theo:mainLLN} is applicable and $\omega_{\frac{1}{\sqrt{2
			n_k}} T^{(k)}_{\leq \alpha n_k  }}(x)$ indeed converges in probability to a
limit shape $\Omega_\alpha$ which is uniquely determined by the sequence of its
free cumulants
\[ R_k(\Omega_{\alpha}) = \alpha^{k-1} R_k(\Omega_1).\]

The problem of finding the limit curves explicitly might be computationally
challenging, see \cref{sec:recover-shape}. Nevertheless in some cases we might
be lucky to have a concrete final answer, see below.

\subsubsection{The example of Linusson, Potka and Sulzgruber}
\label{sec:example-LPS}

Pittel and Romik \cite{Pittel2007} studied asymptotics of random standard Young
tableaux having a square shape; in other words it was the non-shifted analogue
of the problem studied by Linusson, Potka and Sulzgruber in \cite{Linusson2018}.
The result of Pittel and Romik implies in particular that the (scaled down)
level curves of such random standard Young tableaux converge in probability to a
family of curves defined for $-1< z < 1$ by
\begin{equation}
\label{eq:PR}
\Omega_{\alpha}(z) = \sup \big\{ x+y: 0\leq x,y\leq 1 \text{ and } 
x-y=z \text{ and } L(x,y)\leq \alpha \big\} 
\end{equation}
for some explicit function $L$ which describes the limit shape of the tableau.

On the other hand, the results of Biane \cite[Theorem 1.5.1]{Biane1998} also
imply convergence of rescaled level curves to some family of curves
$\Omega_\alpha$ which are defined via the free cumulants:
\begin{equation} 
\label{eq:biane-compression}
R_k(\Omega_{\alpha}) = \alpha^{k-1} R_k(\Omega_1),
\end{equation}
where $\Omega_1$ describes the profile of the square Young diagram with a unit
area.

By combining the results of Pittel and Romik with these of Biane it follows that
both limits are equal. In particular, the curves \eqref{eq:PR} have known free
cumulants \eqref{eq:biane-compression}.

The problem studied by Linusson, Potka and Sulzgruber \cite{Linusson2018} of
asymptotics of random shifted staircase tableaux fits into the framework of
\cref{theo:randomSYT} with the limit shape $\Omega_1$ the same as the limit
shape in the problem of Pitman and Romik. Furthermore, in the proof of
\cref{theo:randomSYT} we showed that the level curves $\Omega_\alpha$ fulfil the
same equation  \eqref{eq:biane-compression}. It follows that the level curves
for the asymptotics of shifted staircase tableaux are the same as the limit
level curves for square Young tableaux which was already observed in \cite{Linusson2018}.

\section{How to recover the limit shape?}

\label{sec:recover-shape}

The results of the current paper (such as \cref{theo:mainLLN}) describe the
limit shape of one or another class of large Young diagrams (this limit is given
by some continual Young diagram $\omega\colon\R\to\R_+$) in terms of the
sequence of its free cumulants $R_2(\omega),R_3(\omega),\dots$. For Readers
interested in concrete applications such a description might be not sufficient
and they would prefer to recover the limit shape $\omega$ itself. This problem
has a well-known answer; for Reader's convenience we shall review it briefly.

\subsection{Free cumulants, Cauchy transform, free probability}

In \cref{sec:free-cumulants} we defined free cumulants $R_2,R_3,\dots$ for a
continual Young diagram $\omega$ quite directly via the functionals
$S_2,S_3,\dots$, cf.~\eqref{eq:s-as-moment}, which have a direct geometric
interpretation in terms of the shape of $\omega$. The advantage of this approach
lies in its simplicity. 

Also, as we already discussed in the proof of \cref{theo:mainLLN}, by inverting
the relationship between $(R_n)$ and $(S_n)$, cf.~\eqref{eq:s-in-r}, we see that
the sequence of free cumulants determines uniquely the sequence of moments of
the measure $\sigma_\omega$ considered in \eqref{eq:sigma-measure}. Under some
mild assumptions \emph{Hamburger moment problem} \cite{akhiezer1965classical}
has a unique solution which shows that the measure $\sigma_{\omega}$ is uniquely
determined. In the case when such an existential result is not enough and one
would like to find the measure $\sigma_{\omega}$ explicitly, we will need some
tools of the complex analysis. For this reason we need to recall the equivalent,
original definition of free cumulants for Young diagrams. Originally, such free
cumulants for Young diagrams were defined via a two-step process
\cite{Biane1998} which we recall in the following.

\subsubsection{The first step.}
Kerov \cite[Section 7]{Kerov1998} associated to a given continual Young diagram
$\omega$ two measures on the real line.
The first one, called \emph{Rayleigh measure}, is a signed measure 
\[ \tau=\tau_\omega = \frac{1}{2} \omega'' \]
given by the second derivative of $\omega$ in the distributional sense. The
second one, called \emph{transition measure}, is the unique probability measure
$\mu=\mu_\omega$ for which the corresponding Cauchy--Stieltjes transform is
given by
\begin{equation} 
\label{eq:Cauchy-Stieltjes}
G(z) = G_\mu(z)= \int_{\R} \frac{1}{z-u}\ \mu(\dif u) =
\exp \int_{\R} \log \frac{1}{z-u}\ \tau(\dif u)  
\end{equation}
for $z\in \C\setminus \R$.
In our setup the transition measure is compactly supported.

\subsubsection{The second step.} 

Free cumulants originate in the context of the random matrix theory and
Voiculescu's free probability theory, see  \cite{Mingo2017} for an overview. For
a compactly supported probability measure $\mu$ the corresponding $R$-transform
is an analytic function $R=R_\mu$ on a neighbourhood of $0$ on the complex plane
which fulfils the  equation
\begin{equation} 
\label{eq:cauchy-and-r}
G\left[ R(z)+ \frac{1}{z} \right]= z
\end{equation}
for $z\in\C$ in some neighbourhood of zero. In the specific setup which we
consider in the current paper when $\mu=\mu_\omega$ is the transition measure of
a continual Young diagram
\[ R(z) = \sum_{n\geq 2} R_n z^{n-1} \]
turns out to be the generating function of the sequence of the free cumulants
$R_2,R_3,\dots$ presented in \cref{sec:free-cumulants}; the proof and additional
context can be found in see \cite[Section 3]{DolegaFeraySniady2008}.

\subsection{How to recover the limit shape?}

Equality \eqref{eq:Cauchy-Stieltjes} implies that
\[ (-1) \frac{\dif}{\dif z} \log G(z) = 
\int_\R \frac{1}{z-u} \tau(\dif u);
\]
in other words the left-hand side coincides with the Cauchy--Stieltjes transform
of the signed measure $\tau$. The measure $\tau$ can be recovered from its
Cauchy--Stieltjes tranform by Stieltjes inversion formula \cite[Section
3.1]{Mingo2017} which in our context takes the form
\begin{multline*} \tau[(a,b)] + \frac{\tau(\{a,b\})}{2} =
\frac{1}{\pi}
\lim_{\epsilon\to 0^+} 
\int_a^b \Im \frac{\dif}{\dif z} \log G(z+i \epsilon) = \\ 
\frac{1}{\pi}
\lim_{\epsilon\to 0^+}  \left[ \Im \log G(b+i \epsilon) - \Im 
\log G(a+i \epsilon)\right]=\\
\frac{1}{\pi}
\lim_{\epsilon\to 0^+}  \left[ \arg{ G(b+i \epsilon)} - \arg{ 
	G(a+i \epsilon)} \right] \in [0,1]
\end{multline*}
for any real numbers $a<b$. Note that the left-hand side is equal to
\[ \frac{\omega'(b)-\omega'(a)}{2} \] 
for all $a<b$ for which the derivative of $\omega$ is well-defined.

Thus by considering the limit $a\to-\infty$ we get a
formula for the cumulative distribution function of $\tau$:
\[ \frac{\omega'(z)+1}{2} = \tau[(-\infty,z)] = 
1+ \frac{1}{\pi}
\lim_{\epsilon\to 0^+}  \arg{ G(z+i \epsilon) }.
\]
which holds true for all $z\in\R$ for which $\omega'(z)$ is well-defined.
It follows that the shape of the continual Young diagram is given
by
\begin{multline}
\label{eq:omega-concrete}
 \omega(x) = -x + 2 \int_{-\infty}^x 
\left[ 1+ \frac{1}{\pi}
\lim_{\epsilon\to 0^+}  \arg{ G(z+i \epsilon) }  \right] \dif z = \\
x - \frac{2}{\pi} \int_{x}^{\infty} 
\lim_{\epsilon\to 0^+}  \arg{ G(z+i \epsilon) }   \dif z.
\end{multline}

To summarize: in order to find the continual diagram $\omega$ it is enough to
know its Cauchy transform $G$. For some classical concrete examples of finding
the Cauchy transform when the sequence of free cumulants is known we refer to
\cite[Section 3.1]{Mingo2017}. In the following we will analyse in detail one
concrete example.

\subsection{Example: Schur--Weyl measure for $c=1$}
\label{sec:SW-measure-shape}

We will calculate explicitly the Schur--Weyl limit curve $\Omega^{\SWW}_{c}$
(cf.~\cref{sec:asymptotic-SW}) in the special case $c=1$. We look for a
continual Young diagram with the sequence of free cumulants given by
\[ R_k= \begin{cases}
\left(\frac{1}{\sqrt{2}} \right)^{k-2} & \text{if $k$ is even},\\
0                        & \text{otherwise}. 
\end{cases}\]
This sequence of free cumulants was considered by Deya and Nourdin \cite[Section
2.5]{Deya2012}. With our notations the $R$-transform is given by
\[ R(z)= 
\frac{2z}{2-z^2}; \]
note that \cite{Deya2012} use a definition of $R$-transform which differs by a
factor of $z$. Equation \eqref{eq:cauchy-and-r} shows therefore that $y=G(z)$ is
a solution to the cubic equation
\[ z y^3 +y^2 - 2 z y+2= 0.\]
For $z\neq 0$ this equation has three solutions given by Cardan's formulae. Any
Cauchy transform $G(z)$ fulfils some known properties (such as asymptotics for
$z\to\infty$); by checking them Deya and Nourdin are able to pinpoint the right
solution of the cubic equation which gives $G(z)$. Thus we have a closed formula
for $G(z)$. Regretfully, computation of the argument $\arg G(z)$, even if $z$
converges to the real line, is a computational challenge; in particular we did
not attempt to find a closed formula for the integrals appearing in
\eqref{eq:omega-concrete}. Even though we are not able to evaluate these
integrals they are analytic functions. It follows that the limit curve
$\Omega^{\SWW}_{c}$  is piecewise analytic.

 On the bright side, these integrals are easily
accessible for a numerical integration and the limit curve $\Omega^{\SWW}_{c}$ 
can be efficiently calculated, see \cref{fig:SW-theoretic}.

\subsection{Schur-Weyl measures, the generic case}

The above discussion of the special case $c=1$ is applicable also in the generic
case. Note, however, that for some values of $c$ the behavior of the Cauchy
transform and its singularities might be different.

For example, in the case $c=2$ shown on \cref{fig:SW-theoretic-2} the support of
the Rayleigh measure is a union of two intervals (located symmetrically with
respect to $0$) and of the point $\{0\}$; in particular it follows that
$\Omega^{\SWW}_{2}(x)=|x|+\frac{\sqrt{2}}{2}$ for $x$ in a small neighbourhood
of $0$.

Since this topic is vast and the current paper is already lengthy we postpone
more details to a forthcoming publication.

\appendix

\section{Preliminaries on spin representation theory}
\label{sec:projective-representations}

We reviewing the rudiments of the spin representation theory of the symmetric
groups. For more details and bibliographic references we refer to
\cite{Stembridge1989,Wan2012,Kleshchev2005,Ivanov2004}.

\subsection{Linear and projective representations}

\subsubsection{Linear representations} 

Recall that a \emph{linear}
representation of a finite group $G$ is a group homomorphism $\irrepSp \colon G
\to \GL(V)$ to the group of \emph{linear} transformations $\GL(V)$ of some
finite-dimensional complex vector space~$V$.

\subsubsection{Projective representations}

A \emph{projective} representation of a finite group $G$ 
is a group homomorphism $\irrepSp \colon G \to \PGL(V)$
to the group of \emph{projective linear} transformations 
$\PGL(V)=\GL(V)/\C^{\times} $ of the projective space $P(V)$
for some finite-dimensional complex vector space $V$.
Equivalently, a projective representation can be viewed as a
map $\phi\colon G \to \GL(V)$ to the general linear group with the property that 
\[\irrepSp(x) \irrepSp(y)=c_{x,y}\ \irrepSp(xy) \]
holds true for all $x,y\in G$ for some non-zero scalar $c_{x,y}\in \C^{\times}$. 

\medskip

Each irreducible \emph{linear} representation $\irrepSp\colon G\to\GL(V)$  gives
rise to its projective version $\irrepSp\colon G\to\PGL(V)$. The irreducible
projective representations \emph{which cannot be obtained in this way} are
called \emph{irreducible spin representations} and are in the focus of the
current paper.

\subsection{Spin symmetric group and spin characters}
\label{sec:spin}

The \emph{spin group $\Spin{n}$} \cite{Schur1911}
is a double cover of the symmetric group:
\begin{equation} 
\label{eq:long-sequence}
1 \longrightarrow \Z_2=\{1,z\} \longrightarrow \Spin{n} \longrightarrow \Sym{n} \longrightarrow 1.
\end{equation}
More specifically, it is the 
group generated by $t_1,\dots,t_{n-1},z$
subject to the relations:
\begin{align*}
z^2   &= 1 ,\\
z t_i &= t_i z, & t_i^2&= z & \text{for $i\in [n-1]$}, \\
(t_i t_{i+1})^3 &= z  & & & \text{for $i\in [n-2]$}, \\
t_i t_j &=  z t_j t_i & & & \text{for $|i-j|\geq 2$};
\end{align*}	
we use the convention that $[k]=\{1,\dots,k\}$. Under the mapping
$\Spin{n}\to\Sym{n}$ the generators $t_1,\dots,t_{n-1}$ are mapped to the
Coxeter tranpositions $(1,2),\ (2,3),\dots,(n-1,n)\in\Sym{n}$. 

The main advantage of the spin group comes from the fact that any
\emph{projective} representation $\irrepSp\colon \Sym{n}\to \PGL(V)$ of the
\emph{symmetric group} can be lifted uniquely to a \emph{linear} representation
$\widetilde{\irrepSp}\colon\Spin{n}\to\GL(V)$ of the \emph{spin group} so that
the following diagram commutes:
\tikzexternaldisable
\[ \begin{tikzcd}
\Spin{n} \arrow[r, "\widetilde{\irrepSp}"] \arrow[d]
& \GL(V) \arrow[d] \\
\Sym{n} \arrow[r, "\irrepSp"]
& \PGL(V). \end{tikzcd}
\]
\tikzexternalenable
In this way the \emph{projective} representation theory of \emph{the symmetric
	group $\Sym{n}$} is equivalent to the \emph{linear} representation theory of
\emph{the spin group $\Spin{n}$} which allows to speak about the characters.

\medskip

The
irreducible \emph{spin} representations of $\Sym{n}$ turn out to correspond to
irreducible \emph{linear} representations of the \emph{spin group algebra}
$\SGA{n}:= \C\Spin{n} / \langle z+1 \rangle$ which is the quotient of the group
algebra $\C\Spin{n}$ by the ideal generated by $(z+1)$. 
Equivalently $\SGA{n} = 
\langle 1-z \rangle \subset \C\Spin{n}$ 
can be identified with the ideal generated by the projection $\frac{1-z}{2}$.

\subsection{Conjugacy classes of $\Spin{n}$}

We denote by $\OP$ the set of
\emph{odd partitions}, i.e.~partitions which consist only of odd parts and by
$\OP_n$ the set of odd partitions of a given integer $n\geq 0$.

We denote by $\SP_n^+$ (respectively, $\SP_n^-$) the set of strict partitions
$\xi\in\SP_n$ with the property that the \emph{length} \[ \|\xi\|:=|\xi|-\ell(\xi) \]
is even (respectively, odd), see \cref{fig-younggraph}.

\medskip

For a partition $\pi\vdash n$ we denote by $C_\pi\subset \Spin{n}$ the set of
elements of the spin group which are mapped --- under the canonical homomorphism
$\Spin{n}\to\Sym{n}$ --- to permutations with the cycle-type given by $\pi$.

\medskip

Schur \cite{Schur1911} 
proved the following dichotomy for $\pi\vdash n$:
\begin{itemize}
	\item if one of the following two conditions is fulfilled:
	\begin{itemize}[label=\ding{212}]
		\item $\pi\in \OP_n$, or
		\item $\pi\in \SP_n^-$
	\end{itemize}
	then $C_\pi$ splits into a pair of conjugacy classes of $\Spin{n}$ which will
	be denoted by $C_\pi^\pm$;
	
	\item otherwise, $C_\pi$ is a conjugacy class of $\Spin{n}$. 
	
\end{itemize}

\subsection{Conjugacy classes and spin characters}

Any spin character vanishes on the conjugacy class $C_\pi$ which does not split,
cf.~\cite[p.~95]{Stembridge1989}. For this reason, \emph{from the viewpoint of
	the spin character theory only the conjugacy classes $C_\pi^\pm$ are
	interesting}.

\medskip

Spin representations are exactly the ones which map the central element
$z\in\Spin{n}$ to $-\operatorname{Id}\in\GL(V)$. Since $C_\pi^- = z C_\pi^+$, it
follows that the value of any spin character on $C_\pi^-$ is the opposite of its
value on $C_\pi^+$. For this reason, \emph{from the viewpoint of the spin
	character theory the conjugacy classes $C_\pi^-$ are redundant and it is enough
	to consider the character values only on the conjugacy classes $C_\pi^+$}.

\medskip

From the viewpoint of the asymptotic representation theory it is natural to
consider some \emph{sequence} of groups together  with some natural inclusions;
in our case this is the sequence
\[ \Spin{1} \subset \Spin{2} \subset \Spin{3} \subset \cdots \]
of spin groups. Such a setup allows to relate a conjugacy class of a smaller
group to some conjugacy class in the bigger group and, in this way, to evaluate
the irreducible characters of the bigger group on the conjugacy classes of the
smaller one.

Regretfully, the conjugacy classes $C_\pi^+$ which correspond to $\pi\in\SP_n^-$
do not behave nicely under such inclusions. Indeed, on the level of the
symmetric groups the inclusion $\Sym{n}\subset \Sym{n+k}$ corresponds to adding
$k$ fixpoints to a given permutation; in other words the set $C_\pi\subset
\Spin{n}$ corresponds to $C_{\pi,1^k}\subset \Spin{n+k}$ and the latter does not
split because $(\pi,1^k)\notin \SP_n$ (at least for $k\geq 2$) and
$(\pi,1^k)\notin \OP_n$ (because $n-\ell(\pi)$ is odd which implies that at
least one part of $\pi$ is even).

For this reason, \emph{for the purposes of the asymptotic representation theory
	it is enough to consider only the conjugacy classes $C_\pi^+$ for
	$\pi\in\OP_n$.}

\subsection{Irreducible spin representations}
\label{sec:irreducible-spin-representations}

The relationship between strict partitions and the irreducible spin
representations of the symmetric groups is \emph{not} a bijective one.
Nevertheless, as we shall discuss below, this non-bijectivity can be ignored to
large extent.

\medskip

More specifically (see \cite[p.~235]{Schur1911} and \cite[Theorem
7.1]{Stembridge1989}), each $\xi\in\SP_n^+$ corresponds to a \emph{single}
irreducible representation. We denote by its character by $\phi^{\xi}$.

\medskip

On the other hand, each $\xi\in\SP_n^-$ corresponds to a \emph{pair} of
irreducible spin representations with equal dimensions; we denote their
characters by $\phi^{\xi}_+$ and $\phi^{\xi}_-$. These two characters coincide
on the conjugacy classes $C^{\pm}_{\pi}$ over $\pi\in\OP_n$. For the purposes of
the current paper we not need to evaluate the characters on $C^{\pm}_{\pi}$ for
$\pi\in \SP_n^-$; for this reason we do not have to distinguish between
$\phi^{\xi}_+$ and $\phi^{\xi}_-$ and we may denote them by the same symbol
$\phi^\xi$. 

\subsection{Superrepresentations}

In order to avoid the aforementioned difficulty related to the fact that the
relationship between the irreducible representations of $\SGA{n}$ and $\SP_n$ is
\emph{not} bijective we may change our setup to \emph{superalgebras} and their
\emph{superrepresentations}. The following presentation is based on
\cite[Chapters 12, 13, 22]{Kleshchev2005}.

\medskip

We recall that a \emph{superalgebra} is defined as an algebra $\A$ which is
equipped with some $\Z_2$-grading $\A=\A_{\zero}\oplus\A_{\one}$. Similarly, a
\emph{superspace} is a linear space $V$ equipped with some decomposition
$V=V_{\zero}\oplus V_{\one}$. The algebra $\operatorname{End}(V)$ of
endomorphisms of a superspace carries a natural structure of a superalgebra by
declaring that $X\in \operatorname{End}(V)$ is homogeneous of degree
$i\in\Z_2=\{\zero,\one\}$ if and only if for any homogeneous vector $v\in V_j$
with $j\in \Z_2$ we have that $X(v)\in V_{i+j}$.

A \emph{superrepresentation} $\psi\colon \A \to \operatorname{End}(V)$ of a
superalgebra is an algebra homomorphism which has the additional property that
for any homogeneous element $x\in \A_i$ of degree $i\in\Z_2$ its image $\psi(x)
\in  \operatorname{End}(V)$ is also homogeneous of degree $i$.

\medskip

We define a superalgebra structure on the spin group algebra $\C\Spin{n}$ by
declaring that the linear space of homogeneous elements of degree $\zero$
(respectively, $\one$) is spanned by the group elements $g\in\Spin{n}$ which
under the canonical projection $\Spin{n}\to\Sym{n}$ are mapped to \emph{even}
(respectively, \emph{odd}) permutations. Then irreducible superrepresentations
of $\SGA{n}$	are in a canonical bijective correspondence with strict partitions
in $\SP_n$ \cite[Theorem 22.3.1]{Kleshchev2005}. 

\medskip

\emph{[Here and in the following, the text in square brackets --- such as 
	this one --- is intended only for Readers who are proficient in the
	terminology related to superalgebras. It provides some additional context
	but is not necessary for our purposes.]} 

We consider the irreducible superrepresentation which corresponds to
$\xi\in\SP^+$. \emph{[It turns out to be of type $\mathtt{M}$.]} If we forget the
superalgebra structure, it becomes an irreducible representation of the
algebra $\SGA{n}$ with the character $\phi^\xi$, as discussed in
\cref{sec:irreducible-spin-representations}.

We consider now the irreducible superrepresentation which corresponds to
$\xi\in\SP^-$. \emph{[It turns out to be of type $\mathtt{Q}$.]} If we forget
the superalgebra structure, it becomes a direct sum of the two irreducible
representations of the algebra $\SGA{n}$ with the characters $\phi^\xi_\pm$, as
discussed in \cref{sec:irreducible-spin-representations}. In particular, its
character is equal to $\phi^\xi_+ + \phi^\xi_-$.

\subsection{Spin characters: conclusion}
\label{appendix:conclusion}

For each strict partition $\xi\in\SP_n$ and each odd partition $\pi\in\OP_n$ the
value of the irreducible spin character
\[ \phi^{\xi}(\pi)= \Tr \irrepSp^\xi(c^\pi) \]
is well defined, where $c^\pi\in C_\pi^+$ is a representative of the of the
conjugacy class $C_\pi^+$, cf.~\cite[Eq.~(2.1)]{Stembridge1989}.

\section*{Acknowledgments}

Research of SM was supported by JSPS KAKENHI Grant Number 17K05281.
Research of PŚ was supported by \emph{Narodowe Centrum Nauki}, 
grant number 2017/26/A/ST1/00189.

\biblio

\end{document}